\documentclass[11pt,a4paper]{article}
\pdfoutput=1
\usepackage{amsmath,amssymb,amsthm,latexsym}
\usepackage{a4wide}
\usepackage{mathrsfs}
\usepackage{color}
\usepackage{pifont}
\usepackage[left=2cm, right=2cm,top=2.5cm, bottom=2.5cm]{geometry}
\newtheorem{theorem}{Theorem}[section]
\newtheorem{corollary}{Corollary}[section]
\newtheorem{lemma}{Lemma}[section]
\newtheorem{proposition}{Proposition}[section]
\theoremstyle{definition}
\newtheorem{remark}{Remark}[section]

\numberwithin{equation}{section}
\usepackage{pgfplots}
\pgfplotsset{compat=1.8}
\usepackage{float}
\usepackage{caption}

\newcommand{\bbr}{\mathbb{R}}

\newcommand{\ve}{\varepsilon}

\newcommand{\N}{{\mathbb N}}

\newcommand{\R}{{\mathbb R}}

\newcommand{\D}{\mathcal{D}}

\usepackage[normalem]{ulem}
\usepackage[english]{babel}
\usepackage[abbrev]{amsrefs}
\addto\appendix{
}

\usepackage{hyperref}
\hypersetup{
	colorlinks=true,
	linkcolor=blue,
	filecolor=magenta,      
	urlcolor=red,
	citecolor=red,
	pdftitle={Overleaf Example},
	pdfpagemode=FullScreen,
}

\title{Optimal decay and regularity for a Thomas--Fermi type variational problem}
\author{Damiano Greco \footnote{Swansea University, Fabian Way, Swansea SA1 8EN, Wales, UK,
		email: damiano.greco@swansea.ac.uk}}
\date{\today}

\begin{document}


\maketitle

\begin{abstract}
We study existence and qualitative properties of the minimizers for a Thomas--Fermi type energy functional defined by 
$$\mathcal{E}_\alpha(\rho):=\frac{1}{q}\int_{\mathbb{R}^d}|\rho(x)|^q dx+\frac{1}{2}\iint_{\mathbb{R}^d\times\mathbb{R}^d}\frac{\rho(x)\rho(y)}{|x-y|^{d-\alpha}}dx dy-\int_{\mathbb{R}^d}V(x)\rho(x)dx,$$
where $d\ge 2$, $\alpha\in (0,d)$ and $V$ is a potential. Under broad assumptions on $V$ we establish existence, uniqueness and qualitative properties such as positivity, regularity and decay at infinity of the global minimizer. The decay at infinity depends in a non--trivial way on the choice of $\alpha$ and $q$.  If $\alpha\in (0,2)$ and $q>2$ the global minimizer is proved to be positive under mild regularity assumptions on $V$, unlike in the local case $\alpha=2$ where the global minimizer  has typically compact support. We also show that if $V$ decays sufficiently fast the global minimizer is sign--changing even if $V$ is non--negative. In such regimes we establish a relation between the positive part of the global minimizer and the support of the minimizer of the energy, constrained on the non--negative functions. Our study is motivated by recent models of charge screening in graphene, where sign--changing minimizers appear in a natural way.
\end{abstract}


\tableofcontents

\newpage

\section{Introduction}
\noindent In this paper we analyse minimizers of a Thomas--Fermi type energy of the form
\begin{equation}\label{TF_initial}
\mathcal{E}^{TF}_\alpha(\rho)=\frac{1}{q}\int_{\mathbb{R}^d}|\rho(x)|^qdx-\int_{\bbr^d}\rho(x)V(x)dx +\frac{1}{2} \int_{\bbr^d}(I_\alpha*\rho)(x)\rho(x)dx
\end{equation}
where  $d\ge 2$, $0<\alpha<d$, $\rho:\bbr^d\rightarrow \bbr$ is a measurable function, $V:\bbr^d\rightarrow \bbr$ is a potential and
$I_\alpha$ is defined as
\begin{equation}\label{coulomb}
 I_\alpha(x)=\frac{A_\alpha}{|x|^{d-\alpha}},\qquad A_\alpha=\frac{\Gamma(\tfrac{d-\alpha}{2})}{\Gamma(\tfrac{\alpha}{2})\pi^{d/2}2^{\alpha}}.
\end{equation}
The quantity  $I_\alpha : \R^d\setminus\left\{0\right\} \to \R$ is the Riesz potential of order $\alpha$ and
the choice of the normalisation constant $A_\alpha$ ensures that the convolution kernel $I_\alpha$ satisfies the semigroup property $I_{\alpha+\beta}=I_\alpha*I_\beta$ for each $\alpha,\beta\in(0,d)$ such that $\alpha+\beta<d$ (see  for example \cite{Stein}*{eq. $(6.6)$, p. $118$}).
Furthermore, under suitable assumptions, the Riez potential operator can be associated to the inverse fractional Laplace operator $(-\Delta)^{-\frac{\alpha}{2}}$ (see  \cite{Stein}*{Lemma $2$}).

Our study is motivated by the Thomas--Fermi type model of charge screening in graphene \cite{graphene} (see also  \cites{Katsnelson,orbital}) that corresponds to the special two--dimensional case $q=\tfrac{3}{2}$ and $\alpha=1$,  
\begin{equation}\label{graphene}
    \mathcal{E}_{1}^{TF}(\rho)=\frac{2}{3}\int_{\mathbb{R}^2}|\rho(x)|^{\frac{3}{2}}dx-\int_{\bbr^2}\rho(x)V(x)dx +\frac{1}{4\pi}\int_{\bbr^2}\int_{\bbr^2}\frac{\rho(x)\rho(y)}{|x-y|}dxdy.
\end{equation}
The function $\rho(x)$ here has a meaning of a charge density of fermionic quasiparticles (electrons and holes) in a two--dimensional graphene layer. In general (and unlike in the classical Thomas--Fermi models of atoms and molecules \cite{Lieb-Loss}), the density $\rho$ is a sign--changing function, with $\rho>0$ representing electrons and $\rho<0$ representing holes. The first term in \eqref{graphene} is an approximation of the kinetic energy of a uniform gas of non--interacting particles. The middle term is the interaction with an external potential $V$, and the last quantity is a non--local Coulombic interaction between quasi--particles. In \cite{graphene}, the authors  prove the existence, uniqueness, and decay of the minimizer of the energy \eqref{graphene}. In particular, they establish a \textit{log--correction} to the decay rates of the minimizer, previously identified by M. Katsnelson in \cite{Katsnelson}.


\indent The main goal of this paper is to deduce existence, uniqueness and qualitative properties of the minimizer in the general admissible range
\begin{equation}\label{cond_iniziali}
    \frac{2d}{d+\alpha}<q<\infty,
\end{equation}
where $\frac{2d}{d+\alpha}$ is a critical exponent with respect to the Hardy--Littlewood--Sobolev inequality (see \cite{Lieb-Loss}*{Theorem $4.3$, p. $106$}).
Our main result  is Theorem \ref{decay_intro} where we establish five different asymptotic regimes for the minimizer, depending on the values of $d$, $\alpha$ and $q$. Two of the regimes are entirely new and not visible in the ``local'' case $\alpha=2$, when the non--local term  is the standard Newtonian potential.

\indent Before presenting our results in details, let's emphasise the main differences between  our case and the classical three dimensional TF--theory of atoms and molecules \cites{lieb_3, lieb_dim3,Lieb-Simon} (see also  \cites{goldestein,Rieder} for generalisations of the mentioned three dimensional models).

Unlike in \cites{lieb_3, lieb_dim3,Lieb-Simon}, we are looking at 
\begin{itemize}
	\item global minimizers without a mass constraint since there is no a priori reason for the density function $\rho(x)$ to be integrable;
	\item potentially sign--changing profiles  $\rho(x)$, since electrons and holes could co--exist in a graphene layer;
	\item general range of $\alpha\in (0,d)$, as for istance in \eqref{graphene} $\alpha=1$.
\end{itemize}
We further mention that several difficulties arise if sign--changing profiles $\rho(x)$ are allowed.
 To explain why, we first introduce the Banach space $\mathring{H}^{-\frac{\alpha}{2}}(\bbr^d)$ which can be defined as the space of tempered distributions such that $\hat{\rho}\in L^1_{loc}(\bbr^d)$ and 
\begin{equation}
	\left\|\rho\right\|^2_{\mathring{H}^{-\frac{\alpha}{2}}(\bbr^d)}:=\int_{\bbr^d}{|\xi|}^{-\alpha}|\widehat{\rho}(\xi)|^2 d\xi<\infty,
\end{equation} 
where $\hat{\rho}$ stands for the Fourier transform of $\rho$ (see \cite{classicbook}*{Definition 1.31}). 
It's easy to see (cf. \cite{Stein}*{Lemma 2, Section 5}) that for smooth functions we have the equality 
\begin{equation}\label{riesz_sobolev_norm}
	\D_\alpha(\rho,\rho)=A_\alpha \int_{\bbr^d}\int_{\bbr^d}\frac{\rho(y)\rho(x)}{|x-y|^{d-\alpha}}dxdy=\left\|\rho\right\|^2_{\mathring{H}^{-\frac{\alpha}{2}}(\bbr^d)}.
\end{equation}
However, in general $\mathcal{D}_\alpha$ can not be extended as an integral to the whole of $\mathring{H}^{-\frac{\alpha}{2}}(\R^d)$ without any fast decay or sign restriction. We refer to Section 5 in \cite{BB} and references therein for a more detailed analysis of the topic.
As a consequence, the term $\D_\alpha$ is replaced by $\left\|\cdot\right\|^2_{\mathring{H}^{-\frac{\alpha}{2}}(\mathbb{R}^d)}$ in the sequel.

To be precise, we consider the energy functional
\begin{equation}\label{energy_general}
	\mathcal{E}^{TF}_\alpha(\rho):=\frac{1}{q}\int_{\bbr^d}|\rho(x)|^q dx-\langle \rho, V \rangle+\frac{1}{2}\left\|\rho\right\|^2_{\mathring{H}^{-\frac{\alpha}{2}}(\bbr^d)}
\end{equation}
 on the  domain $\mathcal{H}_\alpha$  defined as follows
 \begin{equation}\label{dominio}
	\mathcal{H}_{\alpha}:=L^{q}(\mathbb{R}^d)\cap \mathring{H}^{-\frac{\alpha}{2}}(\mathbb{R}^d),
\end{equation}
and endowed with the norm $\left\|\cdot\right\|_{\mathcal{H}_\alpha}:= \left\|\cdot\right\|_{\mathring{H}^{-\frac{\alpha}{2}}(\mathbb{R}^d)}+\left\|\cdot\right\|_{L^{q}(\mathbb{R}^d)}$. In \eqref{energy_general} the function $V$ belongs to the dual space $(\mathcal{H}_{\alpha})'$ and $\langle \cdot, \cdot \rangle$ denotes the duality between $\mathcal{H}_\alpha$ and $(\mathcal{H}_{\alpha})'$. We recall that  $(\mathcal{H}_{\alpha})'$ identifies with $L^{q'}(\mathbb{R}^d)+\mathring{H}^{\frac{\alpha}{2}}(\mathbb{R}^d)$. More precisely, for every function $V=V_1+V_2\in L^{q'}(\mathbb{R}^d)+\mathring{H}^{\frac{\alpha}{2}}(\mathbb{R}^d)$, $V_1\in \mathring{H}^{\frac{\alpha}{2}}(\mathbb{R}^d)$ and $V_2\in L^{q'}(\mathbb{R}^d) $ we define 
\begin{equation}\label{l_p_duality}
	 \langle \rho, V\rangle:=\langle \rho,
V_1\rangle+\int_{\mathbb{R}^d}\rho(x) V_2(x)dx,
\end{equation}
where in \eqref{l_p_duality} by $\langle \cdot,\cdot \rangle$ we identify the duality between $\mathring{H}^{\frac{\alpha}{2}}(\R^d)$ and $\mathring{H}^{-\frac{\alpha}{2}}(\R^d)$.
\footnote{We further recall that the space $L^{q'}(\bbr^d)+\mathring{H}^{\frac{\alpha}{2}}(\bbr^d)$ is a Banach space  endowed with the norm
\begin{equation*}\label{duality_energy}
	\left\|V\right\|_{L^{q'}(\bbr^d)+\mathring{H}^{\frac{\alpha}{2}}(\bbr^d)}\!=\!\inf\left\{\left\|V_1\right\|_{\mathring{H}^{\frac{\alpha}{2}}(\bbr^d)}\!+\!\left\|V_2\right\|_{L^{q'}(\bbr^d)}:V=V_1+V_2, V_1\in \mathring{H}^{\frac{\alpha}{2}}(\bbr^d), V_2\in L^{q'}(\bbr^d)   \right\}.
\end{equation*}}

For the precise definitions and properties of the  fractional homogeneous Sobolev space $\mathring{H}^{\frac{\alpha}{2}}(\bbr^d)$ we refer to Section \ref{prel}.

The first result is to establish existence and uniqueness of a minimizer for $\mathcal{E}^{TF}_\alpha$.
\begin{theorem}\label{minimo_esiste}
Assume $q>\frac{2d}{d+\alpha}$. Then, for every $V\in L^{q'}(\mathbb{R}^{d})+ \mathring{H}^{\frac{\alpha}{2}}(\mathbb{R}^{d})$, the functional $\mathcal{E}^{TF}_\alpha$ admits a unique minimizer $\rho_V\in \mathcal{H}_\alpha$ that satisfies \begin{equation}\label{eq01}
\int_{\mathbb{R}^d} \text{sign}(\rho_V(x))|\rho_V(x)|^{q-1} \varphi(x)dx -\langle \varphi, V\rangle+\langle \rho_V, \varphi\rangle_{\mathring{H}^{-\frac{\alpha}{2}}(\bbr^d)}=0\quad
\forall \varphi\in \mathcal{H}_\alpha.
\end{equation}
\end{theorem}
An equivalent way to read \eqref{eq01} is by  the following relation 
\begin{equation}\label{rho_a.e}
    \text{sign}(\rho_V)|\rho_V|^{q-1}=V-(-\Delta)^{-\frac{\alpha}{2}}\rho_V \quad \text{in}\ \mathcal{D}'(\bbr^d),
\end{equation}
where by $(-\Delta)^{-\frac{\alpha}{2}}\rho_V $ we understand the unique element of $\mathring{H}^{\frac{\alpha}{2}}(\R^d)$ weakly solving 
$$(-\Delta)^{\frac{\alpha}{2}}u=\rho_V.$$ See Section \ref{prel} for all the details.

If, e.g., $\rho_V\in L^{q}(\bbr^d)$ with $q<\frac{d}{\alpha}$, the Riesz potential of $\rho_V$ is well defined (see \cite{Stein}*{Theorem $1$, Sect. $5$}) and  identifies  with $(-\Delta)^{-\frac{\alpha}{2}}\rho_V$ (cf. \cite{BB}*{Corollary 5.3} and \cite{mazya}*{Lemma 1.8}). Moreover, without any restriction on $q$, 
if  $\rho_V\ge 0$,  the  Riesz potential $I_\alpha*\rho_V$ can be identified again  with the operator $(-\Delta)^{-\frac{\alpha}{2}}\rho_V$  (see again \cite{Fukushima}*{Example $2.2.1$} for the case $\alpha\in (0,2]$ and \cite{BB}*{Remark 5.1} or \cite{mazya}*{Lemma 1.1} for the general case). Thus, in the above cases,  \eqref{rho_a.e} reads also pointwisely  as
\begin{equation}\label{u_quasi_ovunque}
	u_V(x)=V(x)-A_\alpha \int_{\bbr^d}\frac{\rho_V(y)}{|x-y|^{d-\alpha}}dy\quad \text{a.e. in}\  \bbr^d, 
\end{equation}
where $u_V$ is defined by 
 \begin{equation}\label{u_def}
	u_V:=\text{sign}(\rho_V)|\rho_V|^{q-1}.
\end{equation}
Such function $u_V$ plays an important role for the analysis of the sign and decay at infinity of the minimizer. As a matter of fact, if  $V\in \mathring{H}^{\frac{\alpha}{2}}(\mathbb{R}^d)$ we prove that  the Euler--Lagrange equation \eqref{eq01} is equivalent to the fractional semilinear PDE 
\begin{equation}\label{PDE_lap_frac}(-\Delta)^{\frac{\alpha}{2}}u+ \text{sign}(u)|u|^{\frac{1}{q-1}}=(-\Delta)^{\frac{\alpha}{2}}V\quad \text{in}\ \mathring{H}^{-\frac{\alpha}{2}}(\mathbb{R}^d) 
\end{equation}
in the sense that $u_V$ weakly solves it.
We further prove that  \eqref{PDE_lap_frac} satisfies a weak comparison principle provided $\alpha\in (0,2]$. See  for example  \cite{Brezis_1}*{Lemma $2$,  $(27)$, p. $8$} for some general results in the case $\alpha=2$.

Section \ref{sec.4} is devoted to the study of the regularity of the minimizer. The way in which we obtain the regularity is based on a classical bootstrap argument similar to the one in \cite{Carrillo-CalcVar}*{Theorem $8$}. However, some differences occur. The main one is the presence of the function $V$ which clearly might be an obstruction for the regularity of the minimizer. Indeed,  from \eqref{rho_a.e}, it's clear that in general the best regularity that $u_V$ can achieve is the same regularity as $V$. Such  regularity is achieved for example in Corollary \ref{non_local_regularity}.
We further summarise below the basic regularity result for the minimizer $\rho_V$. We stress that such regularity can be improved and specialised according to the regularity assumptions we impose on $V$.
\begin{theorem}\label{regularity_for_intro}
	Assume $q>\frac{2d}{d+\alpha}$. If $V\in \mathring{H}^{\frac{\alpha}{2}}(\R^d)\cap C_b(\bbr^d)$ then $\rho_V\in C(\R^d)$.
\end{theorem}

 In Section \ref{sec.5} we focus  on the relation between unconstrained minimizers and minimizers of $\mathcal{E}_{\alpha}^{TF}$ subject to contraint to non--negative functions, defined by $$\mathcal{H}^{+}_{\alpha}:=\left\{f\in \mathcal{H}_\alpha: f\ge 0 \right\}.$$
Then, we establish existence and uniqueness of a minimizer in $\mathcal{H}^{+}_\alpha$.
 Namely, denoted by $[x]_{+}:=~\max\left\{0,x\right\}$, we prove the following:
\begin{theorem}\label{second_theorem}
Assume $q>\frac{2d}{d+\alpha}$. Then, for every $V\in L^{q'}(\mathbb{R}^d)+\mathring{H}^{\frac{\alpha}{2}}(\mathbb{R}^d) $ the function $\mathcal{E}^{TF}_\alpha$ admits a unique minimizer  $\rho^{+}_V$ belonging to $\mathcal{H}^{+}_\alpha$.  Moreover, the minimizer $\rho^{+}_V$ satisfies 
\begin{equation}\label{ineq_5}
	\int_{\mathbb{R}^d}{(\rho^{+}_V)}^{q-1}\varphi-\langle \varphi, V\rangle+\langle \rho^{+}_V, \varphi \rangle_{\mathring{H}^{-\frac{\alpha}{2}}(\mathbb{R}^d)}\ge 0\quad \forall \varphi \in  \mathcal{H}^{+}_\alpha,
\end{equation}
and 
\begin{equation}\label{eq_intro}
\rho^{+}_V=\left[V-(-\Delta)^{-\frac{\alpha}{2}}\rho^{+}_{V}\right]^{\frac{1}{q-1}}_{+}\quad \text{in}\ \mathcal{D}'(\mathbb{R}^d)
.\end{equation}
\end{theorem}

Next, we discuss the relation between the non--negative and free minimizer. We also prove that the equality  $[\rho_V]_{+}=\rho^{+}_V$ holds only in the trivial case in which $\rho_V\ge 0$ (see Theorem \ref{minimi diversi}) while we obtain the inequality $[\rho_V]_{+}\ge \rho^{+}_V$ provided the potential $V$ is regular and decays sufficiently fast (see e.g., Remark \ref{not_sharp_V}). Namely, the following holds:
\begin{theorem}\label{parte pos}
Assume $d\ge 2$, $0<\alpha\le 2$. Let  $V\in \mathring{H}^{\frac{\alpha}{2}}(\R^d)\cap C(\mathbb{R}^d)$ be compactly supported and not identically zero.
Then, 
  \begin{equation}\label{ineq_pointwise}
[\rho_V]_{+}\ge \rho^{+}_V\quad \text{in}\ \bbr^d,
\end{equation}
the inequality being strict in a set of positive measure.
\end{theorem}


In Section \ref{sec.6} we study decay properties of  $\rho_V$ in the fractional framework $\alpha\in (0,2)$. 
First of all we study positivity of $\rho_V$. If $q\le 2$, the total power of $u$ in \eqref{PDE_lap_frac} is bigger than $1$ and  equation \eqref{PDE_lap_frac} is superlinear. Then,  the positivity follows from Proposition \ref{u_pos}.
If $q>2$ then equation \eqref{PDE_lap_frac} is sublinear which implies that studying the positivity of $\rho_V$ is more delicate and strongly relates to the non--local nature of the fractional Laplacian (see Proposition \ref{superlinear} and Lemma \ref{lower_bound_48}). Indeed, in the local case $\alpha=2$ the support of a non--negative solution of \eqref{PDE_lap_frac} could be compact (see \cite{maximum}*{Theorem $1.1.2$} and  \cite{Diaz}*{Corollary $1.10$, Remark $1.5$}).

Next,  in Theorem \ref{decay_intro} below, we outline the decay of the minimizer $\rho_V$. However,  for the convenience of the reader, we state the result by requiring that $V$ satisfies the stronger decay assumption \eqref{laplac_V_fast} and $C^{\alpha+\ve}_{loc}(\R^d)$ regularity even though such assumptions are not sharp and can be weakened in some regimes of $q$ (see  Section \ref{prel} eq. \eqref{holder_integrable} for the defintion of the the latter space and Remarks \ref{ass_reg_weak}, \ref{restrizione_deca}). The same considerations applies to Theorem \ref{decay_intro2}.
\begin{theorem}\label{decay_intro}
	Assume $d\ge 2$, $0<\alpha<2$ and $q>\frac{2d}{d+\alpha}$ . If $V$ is  a non--indentically zero element of $\mathring{H}^{\frac{\alpha}{2}}(\mathbb{R}^d)\cap C_{loc}^{\alpha+\ve}(\mathbb{R}^d)$ and $ (-\Delta)^{\frac{\alpha}{2}}V$ is a non--negative function such that \begin{equation}\label{laplac_V_fast}\limsup_{|x|\to +\infty}|x|^{d+\alpha}(-\Delta)^{\frac{\alpha}{2}}V(x)<+\infty,
	\end{equation}
	then $\rho_V\in L^{1}(\R^d)$. Moreover, the following cases occur:
	\begin{itemize}
		\item[$(i)$] If $ \frac{2d}{d+\alpha}<q<\frac{2d-\alpha}{d}$ then 
		\begin{equation*}
			0<\liminf_{|x|\to +\infty}|x|^{\frac{d-\alpha}{q-1}}\rho_V(x)\le \limsup_{|x|\to +\infty}|x|^{\frac{d-\alpha}{q-1}}\rho_V(x)<+\infty;
		\end{equation*}
	
		\item[$(ii)$] If $q=\frac{2d-\alpha}{d}$ and, either $1<\alpha<2$ and $d>\alpha+1$ or $q=\tfrac{3}{2}$, then
		\begin{equation*}
			0<\liminf_{|x|\to +\infty}|x|^{d}(\log|x|)^{\frac{d}{\alpha}}\rho_V(x)\le \limsup_{|x|\to +\infty}|x|^{d}(\log|x|)^{\frac{d}{\alpha}}\rho_V(x)<+\infty;
		\end{equation*}
		\item[$(iii)$] If $\frac{2d-\alpha}{d}<q<\frac{2d+\alpha}{d+\alpha}$ then 
		\begin{equation*}
			0<\liminf_{|x|\to +\infty}|x|^{\frac{\alpha}{2-q}}\rho_V(x)\le\limsup_{|x|\to +\infty}|x|^{\frac{\alpha}{2-q}}\rho_V(x)<+\infty;  
		\end{equation*}
		\item[$(iv)$] If $q=\frac{2d+\alpha}{d+\alpha}$ then
		$$0<\liminf_{|x|\to +\infty}|x|^{d+\alpha}(\log|x|)^{-\frac{d+\alpha}{\alpha}}\rho_V(x)\le  \limsup_{|x|\to +\infty}|x|^{d+\alpha}(\log|x|)^{-\frac{d+\alpha}{\alpha}}\rho_V(x)<+\infty; $$
		\item[$(v)$] If $q>\frac{2d+\alpha}{d+\alpha}$ then
		$$0<\liminf_{|x|\to +\infty}|x|^{d+\alpha}\rho_V(x)\le \limsup_{|x|\to +\infty}|x|^{d+\alpha}\rho_V(x)<+\infty.$$
	\end{itemize}
In particular, if $ \frac{2d}{d+\alpha}<q<\frac{2d-\alpha}{d}$ then 
\begin{equation*}\label{sharp_decay}
	\lim_{|x|\to +\infty}|x|^{\frac{d-\alpha}{q-1}}\rho_V(x)=\left[A_{\alpha}\left(\big| \! \big|(-\Delta)^{\frac{\alpha}{2}}V\big| \!\big|_{L^{1}(\mathbb{R}^d)}-\left\|\rho_V\right\|_{L^{1}(\mathbb{R}^d)}\right)\right]^{\frac{1}{q-1}}>0.
\end{equation*}
\end{theorem}
\vspace{0.2cm}

Before showing some applications of Theorem \ref{decay_intro}, we briefly underline how the proof develops. 
We intend to construct positive super and subsolutions of  \eqref{PDE_lap_frac} in exterior domains of the form $\Omega=\overline{B_R}^{c}$. The scenario is the following:
\begin{itemize}
\item If $q\in \big(\frac{2d}{d+\alpha},\frac{2d+\alpha}{d+\alpha}\big)\setminus\left\{\frac{2d-\alpha}{d}\right\}$, we employ  polynomial--decaying functions as barriers. Namely, we use a linear combination of the functions  $f$ and $g$ where    $f(x)=(1+|x|^2)^{-\beta}$, for some suitable $\beta>0$, and  $g$ is defined as the optimizer for  the following problem 
\begin{equation}\label{inf_sobolev_intro}
    \mathcal{I}:=\inf_{u\in \mathring{H}^{\frac{\alpha}{2}}(\bbr^d)}\left\{\|u\|^2_{\mathring{H}^{\frac{\alpha}{2}}(\bbr^d)}:u\ge \chi_{B_1}\right\}
\end{equation}
(see Lemma \ref{range_1}, eq. \eqref{v_intro}). Decay estimates for such funtions were derived in  \cite{Ferrari}*{equation 
	$(4.4)$} and  \cite{optimal}*{Proposition $3.6$};
\item If $q\in \left\{\frac{2d-\alpha}{d},\frac{2d+\alpha}{d+\alpha}\right\}$, polynomial--behaving barriers are not enough to obtain sharp estimates. Indeed, we apply 
\cite{Ferrari}*{Theorem $1.1$} to deduce asymptotic decay for functions with logarithmic behaviour (see Lemma \ref{alfa_grande} and \ref{2_limit});
\item If $q\in \big(\frac{2d+\alpha}{d+\alpha},\infty\big)\setminus \left\{2\right\}$, by applying again  \cite{Ferrari}*{Theorem $1.1$},  we employ polynomial--behaving functions with fast decay (see Lemma \ref{max_decay});
\item If $q=2$, the conclusion is essentially a consequence of \cite{positive solutions}*{Lemmas 4.2--4.3}.
\end{itemize}
Note also that, differently from the local case $\alpha=2$, the exponent   $q=\frac{2d+\alpha}{d+\alpha}$ represents a critical value. Furthermore,  if $q=\frac{2d-\alpha}{d}$ we have the slowest decay and if $q>\frac{2d+\alpha}{d+\alpha}$ the decay is the same as  in \cite{positivesolutions}. 
 We also recall that  log--corrections appear at the critical values $\frac{2d-\alpha}{d}$ and  $\frac{2d+\alpha}{d+\alpha}$.

	\usetikzlibrary{
		calc,
		patterns,
		positioning
	}
	\pgfplotsset{
		compat=1.16,
		samples=200,
		clip=false,
		my axis style/.style={
			axis x line=middle,
			axis y line=middle,
			legend pos=outer north east,
			axis line style={
				->,
			},
			legend style={
				font=\footnotesize
			},
			label style={
				font=\footnotesize
			},
			tick label style={
				font=\footnotesize
			},
			xlabel style={
				at={
					(ticklabel* cs:1)
				},
				anchor=west,
				font=\footnotesize,
			},
			ylabel style={
				at={
					(ticklabel* cs:1)
				},
				anchor=west,
				font=\footnotesize,
			},
			xlabel=$q$
		
		},
	}
	\tikzset{
		>=stealth
	}


	\captionsetup{
		format=plain,
		labelfont=bf,
		font=small,
		justification=centering
	}
	


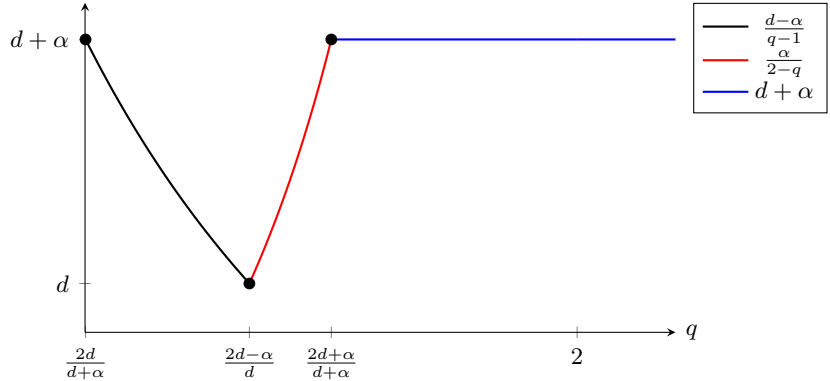
\begin{figure}[ht]
	\centering
	\begin{tikzpicture}
	\begin{axis}[
		my axis style,
		width=.55\textwidth,
		height=.35\textwidth,
            xtick={3/2,5/3,7/4,2},
            xticklabels={$\frac{2d}{d+\alpha}$, $\frac{2d-\alpha}{d}$, $\frac{2d+\alpha}{d+\alpha}$, $2$},
            ymin=2.8,
            ytick={3,4},
            yticklabels={$d$,$d+\alpha$},
            ymax=4.15,
		legend entries={
			$\frac{d-\alpha}{q-1}$,
                $\frac{\alpha}{2-q}$,
                $d+\alpha$
		},
        legend pos= outer north east
	]
	
	\addplot[
		domain=3/2-1/2000:5/3,
		thick,
	]
	{2/(x-1)};

	\addplot[
                domain=5/3:7/4,
                thick,
                red,
	]
	{1/(2-x)};

 \addplot[
                domain=7/4:2,
                thick,
                blue,
	]
	{4};

 \addplot[
domain=2:2.1,
thick,
blue,
]
{4};	

 \addplot[only marks] coordinates {(3/2,4)};
 \addplot[only marks] coordinates {(5/3,3)};
 \addplot[only marks] coordinates {(7/4,4)};
	\end{axis}
	\end{tikzpicture}
	\caption{Plot of the decay's exponent in the range $q\in \big(\frac{2d}{d+\alpha},\infty\big)\setminus\left\{\frac{2d-\alpha}{d},\frac{2d+\alpha}{d+\alpha}\right\}$.}
	\label{fig:my-awesome-graph}
\end{figure}
	As we have already stressed above, in the superlinear regime $q>2$ even positivity of the minimizer  is not trivial. 
Let's further recall that when $\alpha=2$ the decay of non--negative solutions of \eqref{PDE_lap_frac} is polynomial provided $q<2$ and exponential provided $q=2$ (see \cite{Veron}).



An important class of potentials satisfying the assumptions of Theorem \ref{decay_intro} is the following:
\begin{equation}\label{family_minimizers}
    V_Z(x):=Z A_\alpha\left(1+|x|^2\right)^{-\frac{d-\alpha}{2}}, 
\end{equation}
where $A_\alpha$ is defined in \eqref{coulomb} and $Z$ is a positive constant. Then, we can prove that




\begin{itemize}
\item[$(i)$] If $\frac{2d}{d+\alpha}<q<\frac{2d-\alpha}{d}$ then $$\left\|\rho_{V_Z}\right\|_{L^{1}(\bbr^d)}<Z;$$
    \item[$(ii)$] If $q>\frac{2d-\alpha}{d}$ then $$\left\|\rho_{V_Z}\right\|_{L^{1}(\bbr^d)}=Z.$$
\end{itemize}
For the precise statements we refer to Corollaries \ref{L_1_bound} and  \ref{V_z}.

Finally, we study asymptotic decay of sign--changing minimizers where we don't expect the existence of a universal lower bound. Namely, we obtain a similar version of Theorem \ref{decay_intro} with  upper bounds only. Note that, in view of Remark \ref{remark_soglia1}, case $(ii)$ of Theorem \ref{decay_intro2} requires less restrictions than case $(ii)$ of Theorem \ref{decay_intro}.
\begin{theorem}\label{decay_intro2}
	Assume $d\ge 2$, $0<\alpha<2$ and $q>\frac{2d}{d+\alpha}$. Assume that  $V$ belongs to $\mathring{H}^{\frac{\alpha}{2}}(\R^d)\cap C^{\alpha+\ve}_{loc}(\mathbb{R}^d)$ and  \begin{equation}\label{laplac_V_fast_2}\limsup_{|x|\to +\infty}|x|^{d+\alpha}|(-\Delta)^{\frac{\alpha}{2}}V(x)|<+\infty.
	\end{equation}
	If $|\rho_V|\in \mathring{H}^{-\frac{\alpha}{2}}(\R^d)\cap L^{\infty}(\R^d)$  then the following possibilities hold:
	\begin{itemize}
		\item[$(i)$] If $ \frac{2d}{d+\alpha}<q<\frac{2d-\alpha}{d}$ then 
		\begin{equation*}
			\limsup_{|x|\to +\infty}|x|^{\frac{d-\alpha}{q-1}}|\rho_V(x)|<+\infty;
		\end{equation*}
		\item[$(ii)$] If $q=\frac{2d-\alpha}{d}$ and, either $1<\alpha<2$ or $q=\tfrac{3}{2}$, then
		\begin{equation*}
			\limsup_{|x|\to +\infty}|x|^{d}(\log|x|)^{\frac{d}{\alpha}}|\rho_V(x)|<+\infty;
		\end{equation*}
		\item[$(iii)$] If $\frac{2d-\alpha}{d}<q<\frac{2d+\alpha}{d+\alpha}$ then 
		\begin{equation*}
			\limsup_{|x|\to +\infty}|x|^{\frac{\alpha}{2-q}}|\rho_V(x)|<+\infty;  
		\end{equation*}
		\item[$(iv)$] If $q=\frac{2d+\alpha}{d+\alpha}$ then
		$$ \limsup_{|x|\to +\infty}|x|^{d+\alpha}(\log|x|)^{-\frac{d+\alpha}{\alpha}}|\rho_V(x)|<+\infty; $$
		\item[$(v)$] If $q>\frac{2d+\alpha}{d+\alpha}$ then
		$$\limsup_{|x|\to +\infty}|x|^{d+\alpha}|\rho_V(x)|<+\infty.$$
	\end{itemize}
\end{theorem}
The proof is essentialy based on two facts:
\begin{itemize}
	\item similarly to the proof of Theorem \ref{parte pos}, we employ a Kato--type inequality of the form 
	$$ (-\Delta)^{\frac{\alpha}{2}}|u|\le \text{sign}(u)(-\Delta)^{\frac{\alpha}{2}} u\quad \text{in}\  \mathcal{D}'(\bbr^d),$$
	which allows us to conclude that the function $|u_V|$ is a non--negative weak subsolution of the semilinear PDE
	\begin{equation}\label{PDE_abs}
		(-\Delta)^{\frac{\alpha}{2}}u+\text{sign}(u)|u|^{\frac{1}{q-1}}= |(-\Delta)^{\frac{\alpha}{2}}V|\quad \text{in}\ \mathring{H}^{-\frac{\alpha}{2}}(\R^d);
	\end{equation}
	\item we employ  the same barriers used for the proof of Theorem \ref{decay_intro} as supersolutions of \eqref{PDE_abs} in exterior domains of the form $\Omega=\overline{B_R}^{c}$.
\end{itemize}

\vspace{0.3cm}

\noindent \textbf{Organization of the paper.} In Section \ref{prel} we introduce some definitions useful for the rest of the paper. In Sections \ref{sec.3} and \ref{sec.5} we prove respectively Theorem \ref{minimo_esiste} and \ref{second_theorem}. Section \ref{sec.4} is entirely dedicated to prove regularity of the minimizer.  In Section \ref{sec.6} we  study the asymptotic behaviour of the minimizer and we prove Theorems \ref{decay_intro}, \ref{decay_intro2}. As a result, in Corollary \ref{L_1_bound} we also estimate the $L^1$--norm of the minimizer in terms of the $L^{1}$--norm of $(-\Delta)^{\frac{\alpha}{2}}V$. Finally, in the Appendix \ref{Appendix}, we derive asymptotic estimates for the fractional Laplacian of power--type functions.
\vspace{0.3cm}

\noindent \textbf{\large{Notations.}} 
\noindent \underline{Sets}:
\noindent 
\begin{itemize}
	\item For every $x\in \R^d$ and $R>0$, by $B_R(x)$ we understand the set defined as 
	$$B_R(x)=\left\{y\in \bbr^d:|x-y|<R\right\}.$$ We will often set  ${B_R:=B_R(0)}$.
	As it is customary we denote by $\omega_d$ the measure of the unit ball in $\R^d$. 
	\item Let $\Omega$ be a subset of $\R^d$.  By  $|\Omega|$, $\Omega^c$ and $\overline{\Omega}$
	we understand respectively the $d$--dimensional Lebsegue measure, the complementary and the closure of $\Omega$. Moreover, we say that $\Omega$ is smooth if its boundary $\partial \Omega$ locally identifies with a smooth function.
\end{itemize}

\noindent \underline{Function spaces}:
Let $\Omega$ be an open set in $\R^d$.
\begin{itemize}
    \item  $C^{\infty}_c(\Omega)$ denotes the space of smooth functions having compact support contained in $\Omega$; 
    \item If $k\in \N$, the space $C^{k}(\Omega)$ (respectively $C^{k}_b(\Omega)$)  represents the space of functions whose derivatives are continuous up to order $k$ on $\Omega$ (respectively continuous and bounded);
    \item $\mathcal{D}'(\Omega)$ denotes the dual space of $C^{\infty}_c(\Omega)$. Moreover, if $T\in \mathcal{D}'(\Omega)$, by the the notation $T\ge 0$ in $\mathcal{D}'(\Omega)$ we understand the following inequality
    \begin{equation}\label{ineq_8}
    	\langle T,\varphi \rangle\ge 0,\quad \forall \varphi\in C^{\infty}_c(\Omega), \, \varphi\ge 0,
    \end{equation}
where by the symbol $\langle \cdot , \cdot  \rangle $ we denote the action of $T$ against the test function $\varphi$;
    \item If $\gamma\in (0,1]$ and  $k\in \N$, we define  $C^{k,\gamma}({\Omega})$ (respectively $C^{k,\gamma}_{loc}({\Omega}))$ by
    \begin{equation}
    \begin{split}	
    	C^{0,\gamma}({\Omega})& :=\left\{u\in C_b({\Omega}): 
    	\sup_{x, y \in \Omega,\, x\neq y}\frac{|u(x)-u(y)|}{|x-y|^{\gamma}}<\infty \right\}, \\
  		C^{k,\gamma}({\Omega})& :=\left\{u\in C^{k}_b({\Omega}): D^{\beta}u\in C^{0,\gamma}({\Omega})\ \text{for}\ \text{all}\ |\beta|=k \right\}, \ k\ge 1,
    \end{split}
    \end{equation}
endowed with the norm 
$$\|u\|_{C^{k,\gamma}({\Omega})}=\sum_{|\beta|\le k}\|D^{\beta}u\|_{L^{\infty}({\Omega})}+\sum_{|\beta|=k}[D^{\beta}u]_{C^{0,\gamma}({\Omega})};$$
    \item $L^{p}(\Omega)$ with $p\in [1,+\infty]$ is the standard Lebesgue space and $p'$ denotes the H\"older conjugate of $p$.
\end{itemize}
\noindent \underline{Asymptotic notations}:
\vspace{0.2cm}
\noindent
Let $f,g:\mathbb{R}^d\rightarrow \mathbb{R}$.
We write
\begin{itemize}
\item
$f(x)\lesssim g(x) $ in $\Omega$ 
if there exists $C>0$ such that 
$f(x)\le C g(x)$ for every $x\in \Omega.$
\item 
$g(x)\simeq f(x)$ in $\Omega$ if both
$f(x)\lesssim g(x)$ and $g(x)\lesssim f(x)$ in $\Omega$.
\end{itemize}
In particular, by writing $f(x)\lesssim g(x)$ as $|x|\to +\infty$ we mean
that there exists $C>0$ such that 
$f(x)\le C g(x)$ for every $x$ sufficiently large and by writing $f(x)\simeq g(x)$ as $|x|\to +\infty$ we understand that $f(x)\lesssim g(x)$ and $g(x)\lesssim f(x)$ as $|x|\to +\infty$.
Similarly we write
$f(x)\sim g(x) $ as $|x|\to +\infty$
if 
$$\lim_{|x|\to +\infty}\frac{f(x)}{g(x)}=1.$$

\section{Preliminaries}\label{prel}

\subsection*{Fractional Sobolev spaces}
\noindent Let $d\ge 2$ and $0<\alpha<d$. The fractional Laplacian $(-\Delta)^{\frac{\alpha}{2}}$  on $\mathbb{R}^d$ is defined by means of the Fourier transform  (cf. 
 \cite{mazya}*{eq $(1.1$)}
 ) as follows
\begin{equation}\label{laplac_fourier}
\widehat{(-\Delta)^{\frac{\alpha}{2}}u}(\xi):={|\xi|}^{\alpha}\widehat{u}(\xi),
\end{equation}
where as usual, if $u$ is smooth enough, $\widehat{u}$ denotes the Fourier transform of $u$.

The homogeneous fractional Sobolev space $\mathring{H}^{\frac{\alpha}{2}}(\bbr^d)$ can be defined 
as the completion of $C^{\infty}_c(\R^d)$ under the norm
\begin{equation}\label{norma_H}
\left\|u\right\|^2_{\mathring{H}^{\frac{\alpha}{2}}(\bbr^d)}:=\int_{\bbr^d} |\xi|^{\alpha}|\widehat{u}(\xi)|^2 d\xi,
\end{equation}
or equivalently as
\begin{equation*}
	\mathring{H}^{\frac{\alpha}{2}}(\R^d)=\left\{u\in L^{\frac{2d}{d-\alpha}}(\R^d): |\xi|^{\frac{\alpha}{2}}\hat{u}(\xi)\in L^{2}(\R^d) \right\}
\end{equation*}
(see \cite{classicbook}*{Definition $1.31$} and \cite{multiplier}*{p.10}).
The space  $\mathring{H}^{\frac{\alpha}{2}}(\bbr^d)$ defined above is a Hilbert space (\cite{classicbook}*{Proposition 1.34})
with inner product given by 
\begin{equation}\label{inner_10}
    \langle u,v\rangle_{\mathring{H}^{\frac{\alpha}{2}}(\bbr^d)}=\int_{\bbr^d}|\xi|^{\alpha}\widehat{u}(\xi)\overline{\widehat{v}(\xi)} d\xi = \langle (-\Delta)^{\frac{\alpha}{4}}u, (-\Delta)^{\frac{\alpha}{4}}v\rangle_{L^{2}(\bbr^d)},
\end{equation}
and it is continuosly embedded into $L^{\frac{2d}{d-\alpha}}(\bbr^d)$ (\cite{classicbook}*{Theorems $1.38$--$1.43$}).
The dual space to $\mathring{H}^{\frac{\alpha}{2}}(\bbr^d)$ identifies with $\mathring{H}^{-\frac{\alpha}{2}}(\bbr^d)$ (\cite{classicbook}*{Proposition $1.36$}) and, by the Riesz representation theorem, for every $T\in \mathring{H}^{-\frac{\alpha}{2}}(\bbr^d)$ there exists a unique element ${U}^{\alpha}_{T}\in \mathring{H}^{\frac{\alpha}{2}}(\bbr^d)$ such that 
\begin{equation}\label{weak_laplac}
\langle T,\varphi \rangle=\langle {U}^{\alpha}_{T}, \varphi \rangle_{\mathring{H}^{\frac{\alpha}{2}}(\bbr^d)}\quad  \forall \varphi\in \mathring{H}^{\frac{\alpha}{2}}(\bbr^d),
\end{equation}
where by $\langle \cdot , \cdot  \rangle$ we understand the duality between $\mathring{H}^{-\frac{\alpha}{2}}(\bbr^d)$ and $\mathring{H}^{\frac{\alpha}{2}}(\bbr^d)$. Moreover, 
$$\left\|{U}^{\alpha}_{T}\right\|_{\mathring{H}^{\frac{\alpha}{2}}(\bbr^d)}=\left\|T\right\|_{\mathring{H}^{-\frac{\alpha}{2}}(\bbr^d)},$$
so that the duality \eqref{weak_laplac} is an isometry. In particular, for every $T_1,T_2\in \mathring{H}^{-\frac{\alpha}{2}}(\R^d)$ we have
\begin{equation}\label{isometry}
	\langle {U}^{\alpha}_{T_1} ,{U}^{\alpha}_{T_2} \rangle_{\mathring{H}^{\frac{\alpha}{2}}(\R^d)}=\langle T_1, T_2 \rangle_{\mathring{H}^{-\frac{\alpha}{2}}(\R^d)}.
\end{equation}
The function ${U}^{\alpha}_T$ solving \eqref{weak_laplac} is called a weak solution of the equation 
\begin{equation}\label{weak_10}
    (-\Delta)^{\frac{\alpha}{2}}u=T \quad \text{in}\  \mathring{H}^{-\frac{\alpha}{2}}(\bbr^d).
\end{equation}
In what follows, we will often replace the notation ${U}^{\alpha}_T$ with the more explicit one $(-\Delta)^{-\frac{\alpha}{2}}T$.
Assume now $0<\alpha<2$.
If $u$ is a smooth function, we have a pointwise representation of the operator $(-\Delta)^{\frac{\alpha}{2}}$ given by 
\begin{equation}\label{sing_laplac}
    (-\Delta)^{\frac{\alpha}{2}}u(x)=\frac{C_{d,\alpha}}{2}\int_{\bbr^d}\frac{2u(x)-u(x+y)-u(x-y)}{|y|^{d+\alpha}}dy
\end{equation}
where $$C_{d,\alpha}=\left(\int_{\bbr^d}\frac{1-\cos(y_d)}{|y|^{d+\alpha}}dy\right)^{-1}$$
(see \cite{Valdinoci}*{Sect. $3$, (3.1)--(3.2)}).

If $\ve>0$ is sufficiently small, we define the space $C^{\alpha+\varepsilon}_{loc}(\R^d)$ as
\begin{equation}\label{holder_integrable}
    C^{\alpha+\varepsilon}_{loc}(\bbr^d):=\begin{cases} C^{0,\alpha+\varepsilon}_{loc}(\bbr^d), & \mbox{if }\ 0<\alpha<1 \\ C^{1,\alpha+\varepsilon-1}_{loc}(\bbr^d), & \mbox{if}\,\,\, 1\le \alpha<2,
    \end{cases}
\end{equation}
and 

\begin{equation}\label{L_pesato}
    \mathcal{L}^{1}_{\alpha}(\bbr^d):=\left\{u\in L^1_{loc}(\bbr^d):\int_{\bbr^d}\frac{|u(x)|}{(1+|x|)^{d+\alpha}}dx<+\infty\right\}.
\end{equation}
By the Sobolev embedding, we further have  \begin{equation}\label{incl_11}
	\mathring{H}^{\frac{\alpha}{2}}(\bbr^d)\subset \mathcal{L}^{1}_{\alpha}(\bbr^d),\quad \alpha>0.
	\end{equation}
In particular, if $u\in \mathcal{L}^{1}_{\alpha}(\bbr^d)$, $\alpha>0$, the distributional fractional Laplacian $(-\Delta)^{\frac{\alpha}{2}}u$ can be defined as follows
\begin{equation}\label{11_lap}
	\langle (-\Delta)^{\frac{\alpha}{2}}u, \varphi \rangle :=\int_{\R^d}u(x)(-\Delta)^{\frac{\alpha}{2}}\varphi(x) dx\quad \forall \varphi\in C^{\infty}_c(\R^d),
\end{equation}
See  e.g., \cite{garofalo}*{Corollary $2.16$} for further details. In view of \eqref{11_lap}, we emphasise that
 if ${U}^{\alpha}_{T}$ solves \eqref{weak_10} then  it also solves 
\begin{equation}\label{distr_10}
	(-\Delta)^{\frac{\alpha}{2}}{U}^{\alpha}_{T}=T \quad \text{in}\  \mathcal{D'}(\bbr^d).
\end{equation}
Indeed, by combining \eqref{inner_10} with \eqref{weak_laplac}, for every $\varphi\in C^{\infty}_c(\R^d)$ we have 
\begin{equation*}
	\langle T,\varphi \rangle=\langle (-\Delta)^{\frac{\alpha}{4}}{U}^{\alpha}_{T}, (-\Delta)^{\frac{\alpha}{4}}\varphi\rangle_{L^{2}(\bbr^d)}=\langle {U}^{\alpha}_{T}, (-\Delta)^{\frac{\alpha}{2}}\varphi \rangle\!=\!\int_{\R^d} {U}^{\alpha}_{T}(x) (-\Delta)^{\frac{\alpha}{2}}\varphi (x)dx,
\end{equation*}
where $\langle \cdot, \cdot \rangle$ denotes the duality between $\mathring{H}^{\frac{\alpha}{2}}(\R^d)$ and $\mathring{H}^{-\frac{\alpha}{2}}(\R^d)$ and  the last equality follows from the inclusion \eqref{incl_11}.

 Moreover, the following facts are classical:
\begin{itemize}
	\item if $u\in \mathcal{L}^{1}_{\alpha}(\bbr^d)\cap C^{\alpha+\varepsilon}_{loc}(\bbr^d)$, $\alpha>0$,  then the operator $(-\Delta)^{\frac{\alpha}{2}}$ can be understood pointwise since the integral \eqref{sing_laplac} is convergent, and $(-\Delta)^{\frac{\alpha}{2}}u\in C(\bbr^d)$
(\cite{garofalo}*{Proposition $2.15$});
\item if $u\in \mathcal{L}^{1}_{-\alpha}(\bbr^d)$, $\alpha>0$, then the Riesz potential of $u$ is a convergent integral for almost every $x\in \R^d$. Conversely, if $u\notin  \mathcal{L}^{1}_{-\alpha}(\bbr^d)$ then $(I_{\alpha}*|u|)(x)=\infty$ for a.e. $x\in \R^d$ (\cite{landkof}*{equation $(1.3.10)$}).
\end{itemize}
 If $0<\alpha<2$, instead of using the characterisation of the $\mathring{H}^{\frac{\alpha}{2}}(\bbr^d)$--norm given by \eqref{norma_H}, we have the following equality 
\begin{equation}\label{seminorm}
	\|u\|^2_{\mathring{H}^{\frac{\alpha}{2}}(\bbr^d)}=\frac{C_{d,\alpha}}{2}\int_{\bbr^d}\int_{\bbr^d}\frac{|u(x)-u(y)|^2}{|x-y|^{d+\alpha}}dxdy
\end{equation}
see \cite{classicbook}*{Proposition $1.37$}. 
In particular, if we denote by $[\cdot]_{+},[\cdot]_{-}$ respectively  positive and negative part, we have that  \begin{equation}\label{scalar_H}\|[u]_{\pm}\|_{\mathring{H}^{\frac{\alpha}{2}}(\bbr^d)}\le \|u\|_{\mathring{H}^{\frac{\alpha}{2}}(\bbr^d)},\quad 
	\langle [u]_{+} ,[u]_{-} \rangle_{\mathring{H}^{\frac{\alpha}{2}}(\bbr^d)}\le 0 \quad \forall u\in \mathring{H}^{\frac{\alpha}{2}}(\bbr^d)
\end{equation}
(cf. \cite{Musina}*{page $3$}).

\vspace{0.2cm}

In the following lines, before introducing the notion of regular distribution, we recall that, in view of \cite{BB}*{Lemma 5.1}, the space defined in \eqref{dominio} is the natural domain  where to extend our Thomas--Fermi energy \eqref{energy_general}. 
%
 Furthermore, if  $T\in \mathring{H}^{-\frac{\alpha}{2}}(\bbr^d)\cap {L}^{q}(\bbr^d)$, $\alpha>0$, and  $q<\frac{d}{\alpha}$ the Riesz potential of $T$ is well defined and the function ${U}^{\alpha}_{T}$ solving \eqref{weak_laplac} identifies with it. Namely, 
\begin{equation}\label{riesz_pote}
{U}^{\alpha}_{T}(x)=(I_{\alpha}*T)(x)=A_{\alpha}\int_{\bbr^d}\frac{T(y)}{|x-y|^{d-\alpha}}dy\quad \text{a.e.}\ \text{in}\ \R^d
\end{equation}
(see for example 
\cite{mazya}*{Lemma 1.8} and \cite{BB}*{Corollary 5.3}).

Similarly, if $T\in \mathring{H}^{-\frac{\alpha}{2}}(\bbr^d)\cap L^{1}_{loc}(\R^d)$, $T\ge 0$,  the quantity $I_{\alpha}*T$ identifies with an element of $\mathring{H}^{\frac{\alpha}{2}}(\R^d)$ coinciding with ${U}^{\alpha}_{T}$ and 
\begin{equation}\label{key_norm}
	\left\|T\right\|^2_{\mathring{H}^{-\frac{\alpha}{2}}(\R^d)}=A_{\alpha}\int_{\R^d}\int_{\R^d}\frac{T(x)T(y)} {|x-y|^{d-\alpha}}dxdy.
	\end{equation}
  See \cite{Fukushima}*{Example $2.2.1$, p. 87} and \cite{BB} for a deep analysis on the topic.


In addition, by further assuming continuity of $T$, the equation 
$(-\Delta)^{\frac{\alpha}{2}}(I_{\alpha}*T)=T$
holds pointwisely (see \cite{dyda}*{Proposition 1}),
where now $(-\Delta)^{\frac{\alpha}{2}}$ is understood by means of \eqref{sing_laplac}.

However, if $T\in \mathring{H}^{-\frac{\alpha}{2}}(\bbr^d)\cap \mathcal{L}^1_{-\alpha}(\bbr^d)$ (which again implies that  $I_{\alpha}*T$ is an $L^{1}_{loc}(\R^d)$ function by \cite{landkof}*{equation $(1.3.10)$}) we can not conclude that $T(I_{\alpha}*T)$ is integrable or equivalently that equality \eqref{key_norm} holds (see \cite{BB}*{Example 5.1, Corollary 5.2}).

\subsection*{Regular distributions}
\noindent
Assume now $0<\alpha<d$. Recall that by an element  $T$ belonging to $\mathring{H}^{-\frac{\alpha}{2}}(\bbr^d)\cap L^{1}_{loc}(\bbr^d) $ we understand a tempered distribution such that 
\begin{equation}\label{rho_distribution}
\langle T,\varphi\rangle=\int_{\bbr^d}T(x)\varphi(x)dx\quad \forall \varphi\in C^{\infty}_c(\bbr^d),
\end{equation}
and there exists a positive constant $C$ independent of $T$  such that
\begin{equation}\label{ext_linear}
   |\langle T,\varphi\rangle|\le C\left\|\varphi\right\|_{\mathring{H}^{\frac{\alpha}{2}}(\bbr^d)} \quad \forall \varphi\in C^{\infty}_c(\bbr^d).
\end{equation}
From \eqref{ext_linear}, we conclude that $T$ can be identified as the unique continuous extension with respect to the $\mathring{H}^{\frac{\alpha}{2}}(\bbr^d)$--norm of the linear functional defined by \eqref{rho_distribution}.
An element $T$ satisfying \eqref{rho_distribution} is called a \textit{regular distribution}.

Next,  it's useful to recall some well known properties of the Riesz potential operator when it acts on $L^{q}$ functions.
\subsection*{Regularity of Riesz potential}
\indent Let $0<\alpha<d$, $1<q<\frac{d}{\alpha}$ and $\frac{1}{t}=\frac{1}{q}-\frac{\alpha}{d}$. Then $I_{\alpha}*(\cdot)$ defined by
\begin{equation*}
\begin{split}
    I_{\alpha}*(\cdot):  L^q(\R^d)& \longrightarrow L^t (\R^d)\\
 \rho& \longmapsto I_{\alpha}*\rho
\end{split}
\end{equation*}
is a bounded linear operator. Namely there exists $C>0$ independent of $\rho$ such that
\begin{equation}\label{riesz_l_p}
    \left\|I_{\alpha}*\rho\right\|_{L^{t}(\bbr^d)}\le C\left\|\rho\right\|_{L^{q}(\bbr^d)}\quad \forall \rho\in L^{q}(\bbr^d)
\end{equation}
(see \cite{Stein}*{Theorem $1$, Section $5$}).
If $q\ge\frac{d}{\alpha}$ then $I_{\alpha}*(\cdot)$ is  in general not well defined on the whole space $L^q(\R^d)$.

However, if  $I_{\alpha}*\rho$ is almost everywhere finite on $\R^d$, $f\in L^{q}(\R^d)$ and $\frac{d}{q}<\alpha<\frac{d}{q}+1$ then
$I_{\alpha}*\rho\in L^\infty(\R^d)$ and is H\"older continuous of order $\alpha-\frac{d}{q}$ (see \cite{P55}*{Theorem $2$} or \cite{holderloc}*{Theorem $3.1$} for a local version). 


\section{Existence of a minimizer}\label{sec.3}
In the following section, we prove existence and uniqueness of a minimizer for the Thomas--Fermi energy $\mathcal{E}^{TF}_{\alpha}$. Furthermore, we derive the fractional semilinear PDE \eqref{eq01} which will play a key role for the entire paper and especially for the study of asymptotic decay of the minimizer (see Sect. \ref{sec.6}).
We also recall that the basic  assumptions throughout  all the paper are
$$d\ge 2, \quad 0<\alpha<d,\quad  \frac{2d}{d+\alpha}<q<\infty,$$
and 
$$ V\in \big(L^{q'}(\bbr^d)+\mathring{H}^{\frac{\alpha}{2}}(\R^d)\big)\setminus \left\{0\right\}$$ where as usual $q'=\frac{q}{q-1}$. We also recall that 
$\mathcal{H}_{\alpha}:= L^{q}(\mathbb{R}^d)\cap \mathring{H}^{-\frac{\alpha}{2}}(\mathbb{R}^d)$, the energy functional $\mathcal{E}^{TF}_{\alpha}$ is defined by  $$\mathcal{E}^{TF}_\alpha(\rho):=\frac{1}{q}\int_{\bbr^d}|\rho(x)|^q dx-\langle \rho, V \rangle+\frac{1}{2}\left\|\rho\right\|^2_{\mathring{H}^{-\frac{\alpha}{2}}(\bbr^d)}\quad \forall \rho\in \mathcal{H}_\alpha,$$ 
and by $\rho_V$ we denote  a minimizer (unique by Theorem \ref{minimo_esiste}) of $\mathcal{E}^{TF}_\alpha$ in $\mathcal{H}_\alpha$ emphasising the dependence on the potential $V$. Note that we don't consider $V\equiv 0$ since in this case we trivially have $\rho_V\equiv 0.$ Further restrictions and notations  will be explicitly written. We refer to \cite{graphene}*{Proposition 3.1} for a special case of Theorem \ref{minimo_esiste}.
\vspace{0.3cm}

\noindent \textbf{\large{Proof of Theorem \ref{minimo_esiste}}.}
\begin{proof}
It is standard to see that $\inf_{\mathcal{H}_{\alpha}}\mathcal{E}^{TF}_\alpha(\rho)>-\infty.$
Moreover, If $\left\{\rho_n\right\}_n\subset \mathcal{H}_\alpha$ is a minimizing sequence, then
\begin{equation}
    \sup_{n\in \N}\left(\left\|\rho_n\right\|_{L^{q}(\bbr^d)}+\left\|\rho_n\right\|_{\mathring{H}^{-\frac{\alpha}{2}}(\bbr^d)}\right)<+\infty.
\end{equation}
In particular, up to subsequence we can assume that 
\begin{equation}
\begin{split}
    &\rho_n\rightharpoonup \rho_V\quad \text{in}\ L^{q}(\bbr^d),\\
    & \rho_n\rightharpoonup F\quad \text{in}\ \mathring{H}^{-\frac{\alpha}{2}}(\bbr^d).
\end{split}
\end{equation}
Thus, 
\begin{equation*}
\begin{split}
    \int_{\mathbb{R}^d}\rho_n(x)\varphi(x)dx & \rightarrow  \int_{\mathbb{R}^d}\rho_V(x)\varphi(x)dx\quad \forall \varphi\in L^{q'}(\mathbb{R}^d),\\
     \langle \rho_{n},\varphi\rangle & \rightarrow \langle F,\varphi \rangle \quad \forall \varphi\in \mathring{H}^{\frac{\alpha}{2}}(\bbr^d),
\end{split}
\end{equation*}
from which 
\begin{equation*}
    \int_{\mathbb{R}^d}\rho_V(x)\varphi(x)dx=\langle F,\varphi \rangle\quad \forall\varphi\in L^{q'}(\mathbb{R}^d)\cap \mathring{H}^{\frac{\alpha}{2}}(\bbr^d).
\end{equation*}
In particular, the distribution $F$ is $\mathcal{D}(\R^d)$--regular  and we can indentify $F=\rho_V$ with an element of $L^{q}(\mathbb{R}^d)\cap \mathring{H}^{-\frac{\alpha}{2}}(\bbr^d)$ such that 
$\rho_n\rightharpoonup \rho_V$ in $L^{q}(\mathbb{R}^d)\cap \mathring{H}^{-\frac{\alpha}{2}}(\bbr^d).$
Finally, from the weak lower semicontinuity of $\left\|\cdot\right\|_{L^{q}(\bbr^d)}$ and $\left\|\cdot\right\|_{\mathring{H}^{-\frac{\alpha}{2}}(\bbr^d)}$, and the weak continuity of the linear operator $\cdot\mapsto~  \langle \cdot , V \rangle$.   we obtain
\begin{equation*}
    \mathcal{E}^{TF}_\alpha(\rho_V)\le \liminf_{n\to +\infty}\mathcal{E}^{TF}_\alpha(\rho_n).
\end{equation*}
Uniqueness is a consequence of the strict convexity of $\mathcal{E}^{TF}_\alpha$ while the derivation of the Euler-Lagrange equation is standard.
\end{proof}
\remark Note that Theorem \ref{minimo_esiste} in particular implies that the map 
\begin{equation*}
	\begin{split}
		L^{q'}(\bbr^d)+\mathring{H}^{\frac{\alpha}{2}}(\bbr^d) & \longrightarrow L^{q}(\bbr^d)\cap \mathring{H}^{-\frac{\alpha}{2}}(\bbr^d)\\
		V & \longmapsto \rho_V
	\end{split} 
\end{equation*}
is a bijection.

\remark As we pointed out in the introduction, by testing the  Euler--Lagrange equation \eqref{eq01} against $\varphi\in C^{\infty}_c(\R^d)$ we derive that 
\begin{equation}\label{distribution_19}
	\text{sign}(\rho_V)|\rho_V|^{q-1}=V-(-\Delta)^{-\frac{\alpha}{2}}\rho_V \quad \text{in}\ \mathcal{D}'(\bbr^d),
\end{equation}
where we recall that $(-\Delta)^{-\frac{\alpha}{2}}\rho_V\in \mathring{H}^{\frac{\alpha}{2}}(\R^d)$ is the unique solution of \eqref{weak_10}.

\vspace{0.3cm}

Let $\rho_V\in \mathcal{H}_\alpha$ be the unique minimizer of $\mathcal{E}^{TF}_\alpha$. We introduce the function
\begin{equation}\label{U} 
u_V:=\text{sign}(\rho_V)|\rho_V|^{q-1}.
\end{equation}
Then, writing the Euler--Lagrange equation \eqref{eq01} in terms of $u_V$ yields
\begin{equation}\label{PDE_int}\int_{\mathbb{R}^d} u_V(x) \varphi(x) dx- \langle \varphi, V \rangle +\langle \text{sign}(u_V)|u_V|^{\frac{1}{q-1}} , \varphi\rangle_{\mathring{H}^{-\frac{\alpha}{2}}(\R^d)}=0\quad \forall \varphi\in \mathcal{H}_\alpha.
\end{equation}
As a consequence of \eqref{PDE_int}, in Proposition \ref{V ristretto} we deduce that the function $u_V$ weakly solves the semilinear PDE defined in \eqref{PDE_lap_frac}.
\begin{proposition}\label{V ristretto}
Let $V \in \mathring{H}^{\frac{\alpha}{2}}(\mathbb{R}^d)$ and $u_V$ defined by \eqref{U}. Then the function $u_V$ belongs to  $\mathring{H}^{\frac{\alpha}{2}}(\mathbb{R}^d)$ and  weakly solves
\begin{equation}\label{PDE_lap_fra}(-\Delta)^{\frac{\alpha}{2}}u+ \text{sign}(u)|u|^{\frac{1}{q-1}}=(-\Delta)^{\frac{\alpha}{2}}V\quad  in\ \mathring{H}^{-\frac{\alpha}{2}}(\mathbb{R}^d).
\end{equation}
\end{proposition}
\begin{proof}
Let $\psi\in C^{\infty}_c(\bbr^d)$. Then, the quantity $\varphi:=(-\Delta)^{\frac{\alpha}{2}}\psi$ belongs to $C^{\infty}(\bbr^d)\cap L^{1}(\bbr^d)$ and
$${|(-\Delta)^{\frac{\alpha}{2}}\psi(x)|\lesssim |x|^{-(d+\alpha)}}\quad \text{as}\ |x|\rightarrow +\infty$$ (see \cite{garofalo}*{Proposition $2.9$, Corollary $2.10$}, \cite{mazya}*{Lemma $1.2$} or \cite{abatangelo_2}*{Lemma 3.9} for the details).
We can therefore test \eqref{PDE_int} against $\psi$ to get 
\begin{equation}\label{PDE_int_2}
    \int_{\bbr^d}u_V(x)(-\Delta)^{\frac{\alpha}{2}}\psi(x)dx-\int_{\bbr^d} V(x)(-\Delta)^{\frac{\alpha}{2}}\psi(x)dx+\langle  \rho_V ,  (-\Delta)^{\frac{\alpha}{2}}\psi \rangle_{\mathring{H}^{-\frac{\alpha}{2}}(\bbr^d)}=0.
\end{equation}
Furthermore, by combining \eqref{weak_laplac} with \eqref{isometry} we obtain
\begin{equation}\label{quasi_PDE}
    \langle  \rho_V ,  (-\Delta)^{\frac{\alpha}{2}}\psi \rangle_{\mathring{H}^{-\frac{\alpha}{2}}(\bbr^d)}=\langle (-\Delta)^{-\frac{\alpha}{2}}\rho_V, \psi \rangle_{\mathring{H}^{\frac{\alpha}{2}}(\R^d)}=
\langle \rho_V, \psi \rangle =\int_{\bbr^d}\rho_V(x)\psi(x)dx,
\end{equation}
where by $\langle \cdot, \cdot  \rangle$ we understand the duality between $\mathring{H}^{\frac{\alpha}{2}}(\bbr^d)$ and $\mathring{H}^{-\frac{\alpha}{2}}(\bbr^d)$.
By combining \eqref{PDE_int_2} with  \eqref{quasi_PDE} we obtain that $u_V$ solves the equation
\begin{equation}\label{PDE_distr}
    (-\Delta)^{\frac{\alpha}{2}}u_V+ \text{sign}(u_V)|u_V|^{\frac{1}{q-1}}= (-\Delta)^{\frac{\alpha}{2}}V\quad \text{in}\ \mathcal{D}'(\mathbb{R}^d),
\end{equation}
where $\text{sign}(u_V)|u_V|^{\frac{1}{q-1}}=\rho_V\in \mathring{H}^{-\frac{\alpha}{2}}(\R^d)$ and  $(-\Delta)^{\frac{\alpha}{2}}V\in \mathring{H}^{-\frac{\alpha}{2}}(\R^d)$. 
 Then, from \eqref{PDE_distr} we infer that $(-\Delta)^{\frac{\alpha}{2}}u_V\in \mathring{H}^{-\frac{\alpha}{2}}(\R^d)$ and  (cf. \cite{BB}*{Theorem 1.1}) $u_V$ identifies with $U^{\alpha}_{T}$, with $T=(-\Delta)^{\frac{\alpha}{2}}V-\rho_V$. In particular,  $u_V\in \mathring{H}^{\frac{\alpha}{2}}(\R^d)$ and weakly solves \eqref{PDE_distr}.
\end{proof}

\indent In what follows, assuming $\alpha\in (0,2]$, we prove that the PDE defined by \eqref{PDE_lap_fra} satisfies a weak comparison principle (Lemma \ref{comparison}) which will allow us to deduce information about the minimizer such as asymptotic behaviour and non negativity (under some suitable assumptions). 
\begin{proposition}\label{prop_intro} Let $V \in \mathring{H}^{\frac{\alpha}{2}}(\mathbb{R}^d)$, $\alpha\in (0,2]$. If $(-\Delta)^{\frac{\alpha}{2}}V\ge 0$,  then  $u_V\ge 0$.
\begin{proof} Let $\alpha\in (0,2)$.
From the decomposition 
$u_V=[u_V]_{+}-[u_V]_{-}$ and \eqref{PDE_lap_fra}, we have  
\begin{equation}\label{PDE_piu_meno}
	(-\Delta)^{\frac{\alpha}{2}}([u_V]_{+}-[u_V]_{-} )+  \text{sign}(u_V)|u_V|^{\frac{1}{q-1}}=(-\Delta)^{\frac{\alpha}{2}}V\ge 0\quad \text{in}\ \mathring{H}^{-\frac{\alpha}{2}}(\mathbb{R}^d) .
\end{equation}
Consequently, testing \eqref{PDE_piu_meno} by $[u_V]_{-}$ (which belongs to $\mathring{H}^{\frac{\alpha}{2}}(\mathbb{R}^d)\cap L^{q'}(\mathbb{R}^d)$) we have
\begin{equation}
    \label{pos_no_confronto}
    \langle(-\Delta)^{\frac{\alpha}{2}}([u_V]_{+}-[u_V]_{-}),[u_V]_{-}\rangle+  \int_{\mathbb{R}^d}\text{sign}(u_V)|u_V|^{\frac{1}{q-1}}[u_V(x)]_{-}\ge 0,
\end{equation}
where the above integral representation is due to the obvious inclusions
$$[u_V]_{-}\in L^{q'}(\mathbb{R}^d),\quad \text{sign}(u_V)|u_V|^{\frac{1}{q-1}}\in L^{q}(\mathbb{R}^d).$$ 
Therefore, by combining  \eqref{pos_no_confronto} with $\langle [u_V]_{+},[u_V]_{-} \rangle_{\mathring{H}^{\frac{\alpha}{2}}(\R^d)}\le 0$, we obtain
$$\left\|[u_V]_{-}\right\|_{\mathring{H}^{\frac{\alpha}{2}}(\R^d)}\le 0 ,$$
which implies that $[u_V]_{-}=0$.
The case $\alpha=2$ easily follows as well by combining the above argument with the equality $$\langle [u_V]_{+},[u_V]_{-} \rangle_{\mathring{H}^{1}(\R^d)}=\int_{\R^d}\nabla [u_V]_{+} \nabla [u_V]_{-}=0.$$
\end{proof}
\end{proposition}
Next, we prove that the PDE \eqref{PDE_lap_frac} satisfies a weak comparison principle on smooth domains $\Omega$. In this paper, such result will be applied on domains of the form $\Omega=\overline{B_R}^{c}$, $R>0$ (see for example Lemma \ref{range_2}), or $\Omega=\R^d$ (Theorem \ref{parte pos}).

 \begin{lemma}\label{comparison} Let $V\in \mathring{H}^{\frac{\alpha}{2}}(\bbr^d)$, $\alpha\in (0,2]$. Let $u,v\in \mathring{H}^{\frac{\alpha}{2}}(\mathbb{R}^d)\cap L^{q'}(\mathbb{R}^d)$ such that $\text{sign}(u)|u|^{\frac{1}{q-1}}$ and $\text{sign}(v)|v|^{\frac{1}{q-1}}$ belong to $\mathring{H}^{-\frac{\alpha}{2}}(\bbr^d)$. Assume that $u,v$ are respectively super and subsolution of \eqref{PDE_lap_fra} in a smooth domain $\Omega \subset \mathbb{R}^d$.  Namely,
\begin{equation}\label{sup_18}
	(-\Delta)^{\frac{\alpha}{2}}u + \text{sign}(u)|u|^{\frac{1}{q-1}}\ge (-\Delta)^{\frac{\alpha}{2}}V\qquad in\ \mathcal{D}'(\Omega), 
\end{equation}
\begin{equation}\label{sub_18}
	(-\Delta)^{\frac{\alpha}{2}}v + \text{sign}(v)|v|^{\frac{1}{q-1}}\le (-\Delta)^{\frac{\alpha}{2}}V\qquad in\ \mathcal{D}'(\Omega).
\end{equation}
  If $\mathbb{R}^d\setminus \Omega\neq \emptyset$ we further require that $u\ge v$ in $\mathbb{R}^d\setminus \Omega$.  Then  $u\ge v$ in $\bbr^d$.
 \end{lemma}
\begin{proof}
Assume that $\alpha\in (0,2)$. Then,  by subtracting  \eqref{sup_18} to \eqref{sub_18} we obtain
\begin{equation}\label{PDE_ineq}
    (-\Delta)^{\frac{\alpha}{2}}(v-u)+ (\text{sign}(v)|v|^{\frac{1}{q-1}}-\text{sign}(u)|u|^{\frac{1}{q-1}})\le 0\qquad\ \text{in}\ \mathcal{D}'(\Omega).
\end{equation}
Next, we define the space $\mathring{H}_{0}^{\frac{\alpha}{2}}(\Omega)$ as the completion of $C^{\infty}_c(\Omega)$ w.r.t. the $\|\cdot\|_{\mathring{H}^{\frac{\alpha}{2}}(\R^d)}$--norm.
Then,  by density of $C^{\infty}_c(\Omega)$ in $\mathring{H}_{0}^{\frac{\alpha}{2}}(\Omega) $ and  \eqref{PDE_ineq} we conclude that
\begin{equation}\label{PDE_0}
    \langle v-u,  \varphi \rangle_{\mathring{H}^{\frac{\alpha}{2}}(\bbr^d)}+ \langle \text{sign}(v)|v|^{\frac{1}{q-1}}-\text{sign}(u)|u|^{\frac{1}{q-1}} , \varphi \rangle \le 0\ \, \forall \varphi\in \mathring{H}_{0}^{\frac{\alpha}{2}}(\Omega), \varphi\ge 0, 
\end{equation}
where by $\langle \cdot, \cdot \rangle $ we denoted the duality between $\mathring{H}^{\frac{\alpha}{2}}(\R^d)$ and $\mathring{H}^{-\frac{\alpha}{2}}(\R^d)$.  
Recall that $[v-u]_{+}\in \mathring{H}^{\frac{\alpha}{2}}(\bbr^d)$.
Moreover, since $u\ge v$ in $\bbr^d\setminus \Omega$ and $\partial \Omega$ is regular, the function  $[v-u]_{+} \in~\mathring{H}_{0}^{\frac{\alpha}{2}}(\Omega)$ (see e.g., \cite{BB}*{eq. (2.12)--(2.13)}). 

Hence, we can  test \eqref{PDE_0} against $[v-u]_{+}$. Note also that, since $\text{sign}(v)|v|^{\frac{1}{q-1}}-\text{sign}(u)|u|^{\frac{1}{q-1}}\in \mathring{H}^{-\frac{\alpha}{2}}(\R^d)\cap L^{q}(\R^d)$ and $[v-u]_{+}\in \mathring{H}^{\frac{\alpha}{2}}_0(\Omega)\cap  L^{q'}(\R^d)$, we have the equality
\begin{equation}\label{norm_equality}
\langle \text{sign}(v)|v|^{\frac{1}{q-1}}-\text{sign}(u)|u|^{\frac{1}{q-1}} , [v-u]_{+} \rangle =\int_{\mathbb{R}^d}(\text{sign}(v)|v|^{\frac{1}{q-1}}-\text{sign}(u)|u|^{\frac{1}{q-1}}) [v-u]_{+}.
\end{equation}
Thus by combining \eqref{PDE_0}, \eqref{norm_equality}  and \eqref{scalar_H} we infer
\begin{equation*}
\begin{split}
    0& \ge  \langle v-u,  [v-u]_{+} \rangle_{\mathring{H}^{\frac{\alpha}{2}}(\bbr^d)}+ \int_{\mathbb{R}^d}(\text{sign}(v)|v|^{\frac{1}{q-1}}-\text{sign}(u)|u|^{\frac{1}{q-1}}) [v-u]_{+} \ge\\ 
& \ge \langle [v-u]_{+}, [v-u]_{+}\rangle_{\mathring{H}^{\frac{\alpha}{2}}(\bbr^d)}=\left\| [v-u]_{+}\right\|^2_{\mathring{H}^{\frac{\alpha}{2}}(\mathbb{R}^d)},
\end{split}
\end{equation*}
from which $[v-u]_{+}=0$. This concludes the proof for the case $\alpha\in (0,2)$.

The case $\alpha=2$ follows by the same argument.
\end{proof} 

\begin{corollary}\label{optimizer_1}
	 Let $\alpha\in (0,2]$.
For every $f\in \mathring{H}^{-\frac{\alpha}{2}}(\bbr^d)$ there exists a unique $u_f\in L^{q'}(\bbr^d)\cap \mathring{H}^{\frac{\alpha}{2}}(\bbr^d)$ weakly solving 
\begin{equation}\label{PDE_f}(-\Delta)^{\frac{\alpha}{2}}u+\text{sign}(u)|u|^{\frac{1}{q-1}}=f\quad  in\ \mathring{H}^{-\frac{\alpha}{2}}(\mathbb{R}^d).
\end{equation}
\end{corollary}
\begin{proof}
Let  $V:=(-\Delta)^{-\frac{\alpha}{2}}f$, and $\rho_V$ be  the corresponding unique minimizer provided by Theorem \eqref{minimo_esiste}. Then, Proposition \ref{V ristretto} applied to $\Omega=\bbr^d$ implies that $u_V$ is the desired unique solution of \eqref{PDE_f}.
\end{proof}

\begin{corollary}\label{V_sopra}
Let $V\in \mathring{H}^{\frac{\alpha}{2}}(\mathbb{R}^d)\cap L^{q'}(\mathbb{R}^d)$, $\alpha\in (0,2]$. If $V^{\frac{1}{q-1}}\in \mathring{H}^{-\frac{\alpha}{2}}(\R^d)$, $V\ge 0$ then
$$u_V\le V\quad \text{in}\ \mathbb{R}^d.$$
\end{corollary}
\begin{proof}
Since $V$ is non negative it clearly satisfies
$$(-\Delta)^{\frac{\alpha}{2}}V+V^{\frac{1}{q-1}}\ge (-\Delta)^{\frac{\alpha}{2}}V\quad
\text{in}\ \mathcal{D}{'}(\mathbb{R}^{d}).$$ 
Then, Lemma \ref{comparison} implies
$$u_V\le V\quad \text{in}\ \mathbb{R}^d.  $$
\end{proof}
Next we prove that the minimizer $\rho_V$ is non negative and radially non increasing provided $(-\Delta)^{\frac{\alpha}{2}}V$ is non negative \and radially non increasingas well. In particular Corollary \ref{rad_laplac} below  generalises \cite{graphene}*{Corollary 4.4} to every $\alpha\in (0,2]$ without requiring the additional decay assumption $(-\Delta)^{\frac{\alpha}{2}}V\in L^{\frac{2d}{d+\alpha}}(\R^d)$.

\begin{corollary}\label{rad_laplac}
Let $\alpha\in (0,2]$ and $V\in \mathring{H}^{\frac{\alpha}{2}}(\bbr^d)$. If  $(-\Delta)^{\frac{\alpha}{2}}V$ is a non negative, locally integrable, and  radially non increasing function then $u_V$ is the unique minimizer for the following problem
\begin{equation}
	\mathcal{J}:=\inf_{u\in \mathring{H}^{\frac{\alpha}{2}}(\bbr^d)\cap L^{q'}(\bbr^d)}J(u),
\end{equation}
where $J$ is defined by
\begin{equation}\label{energy_J}
J(u):=\frac{1}{2}\left\|u\right\|^2_{\mathring{H}^{\frac{\alpha}{2}}(\bbr^d)}+\frac{1}{q'}\int_{\bbr^d}|u(x)|^{q'}dx-\langle u, V\rangle_{\mathring{H}^{\frac{\alpha}{2}}(\bbr^d)}.
\end{equation}
Furthermore, $u_V$ is non negative and radially non increasing.
\end{corollary}
\begin{proof} 
The unique minimizer $u$ in $\mathring{H}^{\frac{\alpha}{2}}(\R^d)\cap L^{q'}(\R^d)$ for the convex energy functional $J$ is  characterised by 
\begin{equation}
	\langle u, \varphi\rangle_{\mathring{H}^{\frac{\alpha}{2}}(\R^d)}+\int_{\R^d}\text{sign}(u(x))|u(x)|^{\frac{1}{q-1}}\varphi(x)dx-\langle \varphi, V\rangle_{\mathring{H}^{\frac{\alpha}{2}}(\R^d)}=0\quad \forall  \varphi\in  \mathring{H}^{\frac{\alpha}{2}}(\R^d)\cap L^{q'}(\R^d).
\end{equation}
Hence, by Proposition \ref{V ristretto}, we conclude that $u=u_V$ . In particular, by Proposition \ref{prop_intro}, the function $u_V$  is the unique minimizer of $J$ on the set $\mathcal{P}$ defined as follows  $$\mathcal{P}:=\left\{u\ge 0: u\in \mathring{H}^{\frac{\alpha}{2}}(\bbr^d)\cap L^{q'}(\bbr^d)\right\}.$$ Thus, by combining \eqref{weak_laplac} with \cite{BB}*{Corollary 4.1} we have that $u \,(-\Delta)^{\frac{\alpha}{2}}V\in L^{1}(\bbr^d)$ and 
\begin{equation*}
    \langle u, V\rangle_{\mathring{H}^{\frac{\alpha}{2}}(\bbr^d)}=\langle  (-\Delta)^{\frac{\alpha}{2}}V, u \rangle= \int_{\bbr^d}(-\Delta)^{\frac{\alpha}{2}}V(x)u(x)dx\quad \forall u\in \mathcal{P}.
\end{equation*}
This implies that the  functional $J$ restriced on $\mathcal{P}$ takes the form of
\begin{equation}
    J(u)=\frac{1}{2}\left\|u\right\|^2_{\mathring{H}^{\frac{\alpha}{2}}(\bbr^d)}+\frac{1}{q'}\int_{\bbr^d}|u(x)|^{q'}dx-\int_{\bbr^d}(-\Delta)^{\frac{\alpha}{2}}V(x)u(x)dx\quad \forall u\in \mathcal{P}.
\end{equation}
We now claim that $u^{*}_V=u_V$, where by $u^{*}_V$ we denote the radially symmetric rearrangement of $u_V$. It holds that (see  \cite{riarrangio_1}*{Theorem $9.2$} or \cite{Vincenzo}*{Proposition $2.1$} and \cite{Lieb-Loss}*{Lemma $7.17$,  eq $(4)$ p.$81$})
\begin{equation}\label{eq_20}
\begin{split}
\left\|u^{*}_V\right\|_{\mathring{H}^{\frac{\alpha}{2}}(\bbr^d)}\le \left\|u_V\right\|_{\mathring{H}^{\frac{\alpha}{2}}(\bbr^d)}, \quad \left\|u^*_V\right\|_{L^{q'}(\bbr^d)}= \left\|u_V\right\|_{L^{q'}(\bbr^d)},
\end{split}
\end{equation}
and 
\begin{equation}\label{eq_21}
    \int_{\bbr^d}(-\Delta)^{\frac{\alpha}{2}}V(x)u_V(x)dx  \le \int_{\bbr^d}(-\Delta)^{\frac{\alpha}{2}}V(x)u^{*}_V(x)dx,
\end{equation}
where \eqref{eq_21} follows by radial symmetry and monotonicity of $(-\Delta)^{\frac{\alpha}{2}}V$. Then, 
 putting together \eqref{eq_20} and \eqref{eq_21} we deduce that $u^{*}_V\in \mathcal{P}$ and 
$$J(u^*_V)\le J(u_V)=\mathcal{J}.$$
Hence,  $u^{*}_V$ is also a minimizer and  $u^*_V=u_V$ (by uniqueness).
\end{proof}
\section{Regularity}\label{sec.4}
This section is entirely devoted to prove regularity of the minimizer $\rho_V$. 
The idea is to take advantage of the Euler--Lagrange equation
\begin{equation}\label{distributional_20}
	\text{sign}(\rho_V)|\rho_V|^{q-1}=V-(-\Delta)^{-\frac{\alpha}{2}}\rho_V \quad \text{in}\ \mathcal{D}'(\bbr^d),
\end{equation}
where we recall that by $(-\Delta)^{-\frac{\alpha}{2}}\rho_V$ we denote the unique element of $\mathring{H}^{\frac{\alpha}{2}}(\R^d)$ whose laplacian is $\rho_V$ in the sense of \eqref{weak_10}.
Indeed, in view of \eqref{distributional_20}, it's readily seen that the regularity of $(-\Delta)^{-\frac{\alpha}{2}}\rho_V$ implies regularity of $\rho_V$ (assuming some regularity of the potential $V$).
Thus, in the first two subsections (see \ref{reg_non_pos} and \ref{q=d_over_alpha} below), we focus our attention on the case  $q\ge \frac{d}{\alpha}$ where in general the Riesz potential $I_\alpha*\rho_V$ is not well defined as a Lebesgue integral and can not be identified with the operator $(-\Delta)^{-\frac{\alpha}{2}}\rho_V$.

In the last subsection, we study regularity of the minimizer in the subcritical regime $q<\frac{d}{\alpha}$ (Corollaries \ref{V_cont} and  \ref{non_local_regularity}), where \eqref{distributional_20} reads as 
\begin{equation}\label{quasi_ovunque_sec_4}
	\text{sign}(\rho_V)|\rho_V|^{q-1}=V(x)-A_\alpha \int_{\bbr^d}\frac{\rho_V(y)}{|x-y|^{d-\alpha}}dy\quad \text{a.e.}\ \text{in}\  \bbr^d,
\end{equation}
(see again  \cite{BB}*{Corollary 5.3}).
We further point out that, from now on,  the basic assumption on $V$ is to be a non zero element of $\mathring{H}^{\frac{\alpha}{2}}(\bbr^d)$ while other assumptions will be explicitly written when necessary. 
\remark\label{more_regular}
Before presenting the regularity results of $\rho_V$ on the subcritical regime, we emphasise the fact that Corollaries \ref{reg_non_pos} and \ref{q=d_over_alpha} can be improved by simply requiring more regularity (e.g.,  H\"older regularity) of the potential $V$ in the spirit of Corollary  \ref{non_local_regularity}.
\subsection*{Supercritical cases:  $q> \frac{d}{\alpha},\  \alpha-\frac{d}{q}\notin \N$}
In this regime, the Euler--Lagrange equation reads as \eqref{distributional_20} and, by a classical localisation argument that, we can easily see (Proposition \ref{prop_23}) that the function $(-\Delta)^{-\frac{\alpha}{2}}\rho_V$ is regular. 
\begin{proposition}\label{prop_23}
		Assume $0<\alpha<d$, $q> \frac{d}{\alpha}$ and $\alpha-\frac{d}{q}\notin \N$. Let $u\in \mathcal{L}^{1}_{\alpha}(\R^d)$ be a function solving 
		$$(-\Delta)^{\frac{\alpha}{2}}u=f\quad \text{in}\ \mathcal{D}'(B_R)
		$$
		for some $f\in L^{q}(B_R)$, $R>0$. Then  $u\in C^{\alpha-\frac{d}{q}}(\overline{B_{R/2}})$.
\end{proposition}
\begin{proof}
	 Let's define the function 
	$$u_R(x):=A_{\alpha}\int_{\R^d} \frac{f(y)\chi_{B_R}(y)}{|x-y|^{d-\alpha}}dy.$$
	Then, $u-u_R\in \mathcal{L}^{1}_{\alpha}$ is $\alpha$--harmonic in $B_{R}$. In particular, from  \cite{BB}*{Lemma 3.1}, the latter function  is smooth in $B_{R/2}$.
	(We mention also the more general result \cite{garofalo}*{Theorem 12.19} providing regularity in a more general framework of pseudodifferential operators and \cite{Silvestre}*{Proposition 2.2}.) Moreover, since $f\in L^q(B_R)$, the function $f \chi_{B_R}$ has compact support and belongs to $ ~{L^{1}(\R^d)\cap L^{q}(\R^d)}$. Then, since $\alpha-\frac{d}{q}$ is not an integer, we infer $u_R\in C^{\alpha-\frac{d}{q}}(\R^d)$ (cf. \cite{abatangelo_2}*{Remark 5.18} and references therein). In particular, $u\in C^{\alpha-\frac{d}{q}}(\overline{B_{R/2}})$ concluding the proof.
\end{proof}

In view of Proposition \ref{prop_23}, in Corollary \ref{reg_non_pos} we prove the regularity of  $\rho_V$ when $\alpha-\frac{d}{q}$ is not an integer.

\begin{corollary}\label{reg_non_pos}
	Assume $0<\alpha<d$, $q> \frac{d}{\alpha}$ and $\alpha-\frac{d}{q}\notin \N$. If $V\in \mathring{H}^{\frac{\alpha}{2}}(\R^d)\cap C(\bbr^d)$ then the minimizer $\rho_V\in C(\bbr^d)$.
\end{corollary}
\begin{proof}
By Proposition \ref{prop_23} we obtain continuity of $(-\Delta)^{-\frac{\alpha}{2}}\rho_V$. Then, if $V\in C(\R^d)$, the Euler--Lagrange equation \eqref{distributional_20} implies continuity of $\rho_V$.
\end{proof}
Next, we prove the regularity of the minimizer $\rho_V$ in the cases $\alpha-\frac{d}{q}$ being an integer. This completes the study of regularity in the regime $q\ge \frac{d}{\alpha}$. 
\subsection*{Critical cases: $\alpha-\frac{d}{q}\in \N$}
In this cases we can  reduce the analysis of the regularity to the supercritical case already analysed above. Namely, we prove the following:
    \begin{corollary}\label{q=d_over_alpha}
        Assume $0<\alpha<d$ and $\alpha-\frac{d}{q}\in \N$. If $V\in \mathring{H}^{\frac{\alpha}{2}}(\R^d)\cap C(\bbr^d)$ then $\rho_V\in C(\bbr^d)$.
    \end{corollary}
\begin{proof}
 By the Sobolev embedding, $V$ and $(-\Delta)^{-\frac{\alpha}{2}}\rho_V$ belong to $ L^{\frac{2d}{d-\alpha}}(\bbr^d)$. In particular, from the Euler--Lagrange equation \eqref{distributional_20} we obtain that $\rho_V\in L^{q}(\bbr^d)\cap L^{\frac{2d}{d-\alpha}(q-1)}(\bbr^d)$. Then, $\rho_V\in ~L^{q+\varepsilon}(\bbr^d)$ for  every $\varepsilon>0$ sufficiently small.
	Consequently, from Proposition \ref{prop_23} we infer that $(-\Delta)^{-\frac{\alpha}{2}}\rho_V\in C^{\alpha-\frac{d}{q+\varepsilon}}_{loc}(\R^d)$ for every $\varepsilon>0$ small such that $\alpha-\frac{d}{q+\varepsilon}$ is not an integer. Then, by continuity of $V$ and the Euler--Lagrange equation \eqref{distributional_20}, we deduce the continuity of $\rho_V$.
\end{proof}
\subsection*{Subcritical case: $q< \frac{d}{\alpha}$}
\indent In this regime, the Riesz potential $I_{\alpha}*\rho_V$ is well defined and equation \eqref{distributional_20} reads as
\begin{equation}\label{rho_punt_2}
    \text{sign}(\rho_V)|\rho_V|^{q-1}=V-I_\alpha*\rho_V\quad \text{a.e.}\ \text{in}\ \bbr^d.
\end{equation}
We begin the subsection section by proving H\"older continuity of the Riesz potential provided that the potential $V$ belongs to $L^p(\R^d)$ for some $p$ sufficiently large. 
\begin{lemma}\label{rho regularity}
Let $\rho_V\in \mathcal{H}_\alpha$ be the minimizer and $q<\frac
{d}{\alpha}$.
Then, there exists $p>\frac{2d}{d-\alpha}$ such that if  $V\in \mathring{H}^{\frac{\alpha}{2}}(\R^d)\cap L^{p}(\bbr^d)$ then $I_\alpha*\rho_V$ belongs to  $C^{0,\gamma}(\mathbb{R}^d)$ for some $\gamma\in (0,1]$.
\end{lemma}
\begin{proof}
By \eqref{riesz_l_p},  $I_\alpha*\rho_V\in L^{t_1}(\mathbb{R}^d)$ where $\frac{1}{t_1}=\frac{1}{q}-\frac{\alpha}{d}$. Assume that $V\in L^{t_1}(\mathbb{R}^d)$. Then $\rho_V\in L^{t_1(q-1)}(\mathbb{R}^d).$ 
Moreover, we set $s_1:=t_1(q-1)>q:=s_0$
and we note that  $t_1>\frac{2d}{d-\alpha}$. 

Let's start now from $s_1$ above defined. We split the analysis into three cases:
\begin{itemize}
\item[$(i)$]  $s_1>\frac{d}{\alpha}$,
\item[$(ii)$]  $s_1=\frac{d}{\alpha}$,
\item[$(iii)$]  $s_1<\frac{d}{\alpha}$.
\end{itemize} 
If $(i)$ holds, we have that $\rho_V\in L^{q}(\R^d)\cap L^{s_1}(\R^d)$ from which  \cite{P55}*{Theorem $2$} implies H\"older continuity of $I_\alpha*\rho_V$.

\noindent If $(ii)$ holds,  we have that $\rho_V\in L^{q}(\bbr^d)\cap L^{\frac{d}{\alpha}}(\bbr^d)$ which in turn implies that $\rho_V\in L^{\frac{d}{\alpha}-\varepsilon}(\bbr^d)$ for every $\varepsilon>0$ small.  Hence, \eqref{rho_punt_2} implies that  $I_{\alpha}*\rho_V\in L^{t_{\varepsilon}}(\bbr^d)$ where $t_\varepsilon$ is defined by 
$$t_{\varepsilon}=\frac{d(d-\alpha \varepsilon)}{\alpha^2 \varepsilon}. $$  Thus, if we assume that $V\in L^{t_{\varepsilon}}(\mathbb{R}^d)$ we obtain  $\rho_V\in L^{(q-1)t_{\varepsilon}}(\mathbb{R}^d)$. Consequently, if we choose $\varepsilon$ small enough  (i.e. $0<\varepsilon<\frac{d(q-1)}{\alpha q }$) it holds that $(q-1)t_{\varepsilon}>\frac{d}{\alpha}$. Again from \cite{P55}*{Theorem $2$}   the desired H\"older regularity holds.

\noindent  If  $(iii)$ holds  then we iterate the process.
 In this way we construct a sequence $\left\{s_k\right\}_k$, $k\in \mathbb{N}$ following the rules $(i)-(ii)-(iii).$
 Thus, in view of \cite{P55}*{Theorem $2$}, it's enough to prove that the process stops, or equivalently that there exists $k\in \N$ such that $s_k\ge \frac{d}{\alpha}$. The proof goes by contradiction. Assume that the process doesn't stop. Then,  we obtain a sequence $\left\{s_k\right\}_k$ such that $s_k<\frac{d}{\alpha}$ for every $k\in \N.$ Since
the function $f$ defined by $$x\mapsto\frac{d(q-1)x}{d-\alpha x}:=f(x)$$ is increasing in the interval $(0,\tfrac{d}{\alpha})$, the inequality
    $s_{k+1}>s_{k}$ is equivalent to  the inequality $f(s_k)>f(s_{k-1})$ that is  in turn true by induction since $s_1>s_0=q$.
Therefore, the sequence $\left\{s_k\right\}_{k}$ is increasing and there exists $L:=\lim_{k\to +\infty}s_k.$

\noindent\underline{Case 1.} If $L=\frac{d}{\alpha}$ then
$$\frac{d}{\alpha}=\lim_{k\to +\infty}\frac{ds_k}{d-\alpha s_k}(q-1)=+\infty$$
which is a contradiction.

\noindent\underline{Case 2.} If $0<L<\frac{d}{\alpha}$, passing to the limit on the equation
   $$s_{k+1}=\frac{ds_{k}}{d-\alpha s_{k}}(q-1),$$
   we obtain the following relation
\begin{equation}\label{Limit_regul}
	L=\frac{d(2-q)}{\alpha}.
	\end{equation}
 On the other hand, since the sequence $\left\{s_k\right\}_k$ is increasing, the left hand side of \eqref{Limit_regul} is strictly bigger than $q$ and (by the assumption $q>\frac{2d}{d+\alpha}$) the right hand side of \eqref{Limit_regul} is strictly smaller than $\frac{2d}{d+\alpha}$ which contradicts the assumption $q>\frac{2d}{d+\alpha}$.
 
  Thus the process must end and  \cite{P55}*{Theorem $2$}  implies the regularity. Moreover, $I_\alpha*\rho_V\in L^{\infty}(\bbr^d)$. Indeed, $\rho_V\in L^t(\bbr^d)$ for some $t>\frac{d}{\alpha}$ and 
  \begin{equation}
  	\begin{split}
  		|(I_\alpha*\rho_V)(x)| & =A_\alpha\left|\int_{\bbr^d}\frac{\rho_V(x-y)}{|y|^{d-\alpha}}dy\right|\\
  		&    \le A_\alpha\left(\left\|\rho_V\right\|_{L^{q}(\bbr^d)}\left(\int_{\bbr^d\setminus B_1}\frac{dy}{|y|^{(d-\alpha)q'}}\right)^{\frac{1}{q'}}+\left\|\rho_V\right\|_{L^{t}(\bbr^d)}\left(\int_{ B_1}\frac{dy}{|y|^{(d-\alpha)t'}}\right)^{\frac{1}{t'}}\right)\\
  		& <\infty,
  	\end{split}
  \end{equation}
from which there exists $\gamma\in (0,1]$ such that $\|I_\alpha*\rho_V\|_{C^{0,\gamma}(\R^d)}<\infty$. Furthermore, from the inclusion $I_{\alpha}*\rho_V\in L^{\frac{dq}{d-\alpha q}}(\R^d)\cap C^{0,\gamma}(\R^d)$ we conclude that $I_{\alpha}*\rho_V$ vanishes at infinity.
\end{proof}
Next, to simplify the notation, in Corollaries \ref{V_cont}, \ref{non_local_regularity} and Proposition \ref{holder_riesz} below we restrict the statements to the case $0<\alpha<2$. The other cases can be obtained by an iterations of the same arguments in the spirit of \cite{Carrillo-CalcVar}*{Theorem $10$}.

\begin{corollary}\label{V_cont}
Let $0<\alpha<2$ and $q<\frac{d}{\alpha}$. Let $\rho_V\in \mathcal{H}_\alpha$ be the global minimizer. If  $V\in \mathring{H}^{\frac{\alpha}{2}}(\R^d)\cap C_b(\bbr^d)$ then $\rho_V\in C_b(\R^d)$ and  the following possibilities hold:
\begin{itemize}
	\item[$(i)$] If $\alpha\le 1$   then $I_\alpha*\rho_V\in C^{0,\gamma}(\bbr^d)$ for every $\gamma<\alpha$;
	\item [$(ii)$] If $\alpha>1$ then $I_\alpha*\rho_V\in C^{1,\gamma}(\bbr^d)$ for every $\gamma<\alpha-1$.
	\end{itemize}
In particular,  $\rho_V(x)\rightarrow 0$ as $|x|\rightarrow +\infty$ provided $V(x)\rightarrow 0$ as $|x|\rightarrow +\infty$.
\end{corollary}
\begin{proof}
Assuming $V$ is also bounded continuous we obtain that $V\in L^{s}(\bbr^d)$ for every $s\ge \frac{2d}{d-\alpha}$. Thus, by Lemma \ref{rho regularity}, $I_\alpha*\rho_V$ is H\"older continuous of some exponent $\gamma$. 
Then, from \eqref{rho_punt_2} we obtain $\rho_V\in C_b(\bbr^d)$. This precisely implies the desired regularity of  $I_\alpha*\rho_V $ (see e.g. \cite{Silvestre}*{Proposition $2.9$} or \cite{holderloc}*{Theorem $3.1$}).
Furthermore, since $I_\alpha*\rho_V$ goes to zero, the same holds for $\rho_V$ provided $V$ goes to zero. 
\end{proof}

Next, we recall the following well known result whose proof can be found in  
\cite{Ros}*{Theorem $1.1$--$(b)$, Corollary $3.5$} or \cite{Silvestre}*{Proposition $2.8$}. (See also \cite{Carrillo-CalcVar}*{eq. (3.23)}).
\begin{proposition}\label{holder_riesz} Let $0<\alpha<2.$
Let $\beta$ be a positive number, $[\beta]:=\max\left\{n\in \N_{0}: n<\beta\right\}$ and  $\left\{\beta\right\}:=\beta-[\beta]$.  If $\rho\in C^{[\beta],\left\{\beta\right\}}(B_1)$ with $\beta+\alpha\notin \mathbb{N}$ and $u$ solves
$$(-\Delta)^{\frac{\alpha}{2}}u=\rho\quad \text{}in\ \mathcal{D}'(B_1) $$ then there exists $C>0$ such that
\begin{equation}\label{ineq_locale}
  \|u\|_{C^{[\beta+\alpha], \left\{\beta+\alpha\right\}}({B_{1/2}})}\le C\left(\|u\|_{L^{\infty}(\bbr^d)}+\left\|\rho\right\|_{C^{[\beta],\left\{\beta\right\}}(B_1)}\right).
\end{equation}
\end{proposition}

\remark Proposition \ref{holder_riesz} will be applied to the Riesz potential of $\rho$. Note also that, if $\beta=1$, the space 
$C^{[\beta],\left\{\beta\right\}}(B_1)$ is understood as the space of bounded Lipschitz continuous functions on $B_1$ (cf.  \cite{Carrillo-CalcVar}*{Section 3}).  Moreover  as it has been already discussed in \cite{Carrillo-CalcVar}*{eq. $(3.24)$}, 
if we further assume that $\rho\in C^{[\beta],\left\{\beta\right\}}(\R^d)$, inequality \eqref{ineq_locale} becomes
\begin{equation}\label{ineq_globale}
	\left\|I_\alpha*\rho\right\|_{C^{[\beta+\alpha], \left\{\beta+\alpha\right\}}(\R^d)}\le C\left(\left\|I_\alpha*\rho\right\|_{L^{\infty}(\bbr^d)}+\left\|\rho\right\|_{C^{[\beta],\left\{\beta\right\}}(\R^d)}\right).
\end{equation}


\vspace{0.3cm}

In what follows, by applying Corollary \ref{V_cont} and Proposition \ref{holder_riesz}, we show that assuming  some additional regularity of the potential $V$ yields more regularity of $\rho_V$.
In particular, Corollary \ref{non_local_regularity} proves that  if $V$ is bounded locally H\"older continuous of exponent $\gamma\in (0,1]$, and $q<q_*$ where $q_*$ is defined in \eqref{upper_q},  then $\rho_V$ is locally  H\"older continuous of exponent $\min\left\{\tfrac{\gamma}{q-1},1\right\}$, which is in general the best regularity achievable.
We also notice that if $V$ is locally Lipschitz continuous then the exponent $q_*$ coincides with the critical exponent already mentioned  in \cite{Carrillo-CalcVar}*{Theorem $8$}.

Let $0<\alpha<d$, $0<\gamma\le 1$ and $q>\frac{2d}{d+\alpha}$. We define the quantities $q_*,\gamma_*$ as follows
\begin{equation}\label{upper_q}
	q_*:=\begin{cases} \frac{2\gamma-\alpha}{\gamma-\alpha}, & \mbox{if }\alpha<\gamma, \\ +\infty, & \mbox{if }\alpha\ge \gamma,
	\end{cases}
\end{equation}
\begin{equation}\label{exp_hold}
	\gamma_*:=\text{min}\left\{\frac{\gamma}{q-1},1\right\}.
\end{equation}
Taking into account the above notations, we formulate the main regularity result of the section requiring low regularity of the potential $V$.

\begin{corollary}[H\"older regularity]\label{non_local_regularity}
Let $0<\alpha<2$. Assume $V\in \mathring{H}^{\frac{\alpha}{2}}(\R^d)\cap C^{0,\gamma}_{loc}(\mathbb{R}^d)\cap L^{\infty}(\bbr^d)$ for some $\gamma\in (0,1]$. If $q<\frac{d}{\alpha}$ then  $\rho_V$ is bounded and  locally H\"older continuous. Moreover, if we further assume $q<q_*$, then $\rho_V\in C^{0,\gamma_*}_{loc}(\bbr^d)\cap L^{\infty}(\R^d)$
where $q_*,\gamma_*$ are defined in \eqref{upper_q}-\eqref{exp_hold}.
\end{corollary}
\begin{proof}
The first part of the statement is a direct consequence of Lemma \ref{rho regularity} and Corollary \ref{V_cont}. For the second part, we follow the proof of  \cite{Carrillo-CalcVar}*{Theorem $8$}.

\noindent\underline{Case $\alpha<1$}. Let's start with $q\le 2$. By Lemma \ref{rho regularity} we have $I_\alpha*\rho_V\in C^{0,\beta}(\bbr^d)$ for every $\beta\in (0,\alpha).$  Then, we define the quantity
\begin{equation}\label{beta_it}
	\beta_n=\beta+(n-1)\alpha\in (1-\alpha,1) ,
\end{equation}
where $n\in \mathbb{N}$ is  such that $\frac{1}{n+1}\le \alpha\le \frac{1}{n}$. Clearly we can assume that $\beta\le \gamma$. Indeed, from \eqref{rho_punt_2}, if $\gamma<\beta$
we immediatly get $\text{sign}(\rho_V)|\rho_V|^{q-1}\in C^{0,\gamma}_{loc}(\bbr^d) $.

Assume that $\beta\le \gamma$. By the choice of $n$ and the definition of $\beta_n$ in \eqref{beta_it}, it holds that $\beta_n<1$ and $\beta_n+\alpha>1$. Thus, starting from $\beta$ and applying $(n-1)$-times Proposition \ref{holder_riesz} we get $I_\alpha*\rho_V\in C^{0,\beta_n}_{loc}(\bbr^d)$. Hence,  \eqref{rho_punt_2} yields $\text{sign}(\rho_V)|\rho_V|^{q-1}\in C^{0,\beta_n}_{loc}(\bbr^d)$. Again, if $\gamma<\beta_n$ we conclude. If not, another iteration of Proposition \ref{holder_riesz} leads to $I_\alpha*\rho_V\in C^{0,1}_{loc}(\R^d)$. The desired regularity is therefore obtained by \eqref{rho_punt_2}.

Assume now $q>2$. Let $\beta$ be any number in $(0,\alpha)$. Then, as before $I_\alpha*\rho_V\in C^{0,\beta}(\R^d)$. Moreover,  after $n$--iterations of the previous argument, we obtain that $I_\alpha*\rho_V\in C^{0,\beta_n}_{loc}(\R^d) $ where $\beta_n$ is defined by
 $$\beta_n=\sum_{j=0}^n\frac{\beta}{(q-1)^j}.$$
 Consequently, to obtain the thesis it's sufficient to prove that there exists $\beta\in (0,\alpha)$ and $\bar{n}\in \N$ such that $\beta_{\bar{n}}\ge \gamma$.
Note that the latter inequality is valid if $\beta_{\infty}>\gamma$ where $\beta_{\infty}$ is defined as
   $$\beta_{\infty}:=\sum_{j=0}^{\infty}\frac{\beta}{(q-1)^j}=\frac{\beta(q-1)}{q-2}.$$
Then,  since $\beta\in (0,\alpha)$ is arbitrary, 
it's enough to require $\alpha_{\infty}>\gamma$  which is equivalent to $q<q^*.$
Hence,  if $q<q_{*}$ then $I_{\alpha}*\rho_{V}\in C^{0,\gamma}_{loc}(\R^d)$ and so the thesis follows again from \eqref{rho_punt_2}.

\noindent\underline{Case $\alpha=1$}. By Corollary \ref{V_cont} we get $I_\alpha*\rho_V\in C^{0,\beta}_{loc}(\bbr^d)$ for every $\beta\in (0,1)$. Hence, if $\gamma<1$ the thesis follows. If $\gamma=1$ and $q\le 2$, the Euler-Lagrange equation \eqref{rho_punt_2} implies $I_{\alpha}*\rho_V\in C^{1,\beta}_{loc}(\bbr^d)$ for every $\beta\in (0,1)$ from which $\rho_V\in C^{0,1}_{loc}(\bbr^d)$ again by the Euler-Lagrange equation \eqref{rho_punt_2}. If $\gamma=1$ and $q> 2$, from \eqref{rho_punt_2} we deduce that $\rho_V\in C^{0,\frac{\beta}{q-1}}_{loc}(\R^d)$ for every $\beta\in (0,1)$. In particular, Proposition \ref{holder_riesz}  yields $I_\alpha*\rho_V\in C^{1,\frac{\beta}{q-1}}_{loc}(\bbr^d)$. The conclusion follows again by  \eqref{rho_punt_2}.

\noindent \underline{Case $\alpha>1.$} This case follows again by Proposition \ref{holder_riesz} and by using the same arguments as in the previous cases.
\end{proof}

\remark\label{regul_q_big} If $q\ge q_*$ and $\alpha<1$ we can still get information about the H\"older exponent of $I_\alpha*\rho_V$. Namely $I_\alpha*\rho_V\in C^{0,\beta}_{loc}(\bbr^d)$ for every $0<\beta<\frac{\alpha(q-1)}{q-2}$. In particular, since $\frac{\alpha(q-1)}{
q-2}>\alpha$, we improved the H\"older exponent of $I_{\alpha}*\rho_V$ from $\alpha$ to $\frac{\alpha(q-1)}{
q-2}$. This fact will be important for Proposition \ref{superlinear}, where such regularity is enough to compute the fractional Laplacian via the singular integral \eqref{sing_laplac}.

\section{Non negative minimizer}\label{sec.5}
We begin the section by minimizing the same energy functional $\mathcal{E}^{TF}_\alpha$ over the cone of non negative functions. To be precise, we minimize the energy  $\mathcal{E}^{TF}_\alpha$ over the set  $$\mathcal{H}^{+}_{\alpha}:=\left\{\rho\in \mathcal{H}_\alpha: \rho\ge 0 \right\},$$
where we recall that $\mathcal{H}_\alpha$ is defined by \eqref{dominio}.
 Note also that if $V\le 0$ in the whole of $\R^d$ we clearly obtain the inequality $$\mathcal{E}^{TF}_\alpha(\rho)\ge 0\quad \forall \rho\in \mathcal{H}^{+}_{\alpha},$$
and hence $\rho^{+}_V=0$ is the unique trivial minimizer of $\mathcal{E}^{TF}_\alpha$ in $\mathcal{H}^{+}_\alpha.$ Then, for the rest of the section we will always assume
$V$ to be an element of $\mathring{H}^{\frac{\alpha}{2}}(\R^d)+L^{q'}(\R^d)$ which is positive on a set of positive Lebesgue measure. Further restrictions will be emphasised when needed.


Next, in Theorem \ref{parte pos} and Theorem \ref{minimi diversi}, we study the relation between the free and  constrained minimizer

\vspace{0.3cm}

\noindent \textbf{\large{Proof of Theorem \ref{second_theorem}}.}
\begin{proof}
Once noticed that $\mathcal{H}^{+}_{\alpha}$ is weakly closed in $\mathcal{H}_\alpha$, the proofs of the existence and uniqueness follow the same lines of the proof of Theorem \ref{minimo_esiste}. In particular, \eqref{ineq_5} follows.  Moreover, by arguing as in \cite{Lieb-Loss}*{Theorem $11.13$} or  \cite{Carrillo-NA}*{Proposition 3.6}, we further obtain that the minimizer $\rho^{+}_V$ solves 
$$\rho^{+}_V=\left[V-(-\Delta)^{-\frac{\alpha}{2}}\rho^{+}_{V}\right]^{\frac{1}{q-1}}_{+}\quad \text{in}\ \mathcal{D}'(\mathbb{R}^d)$$
which conludes the proof.

\end{proof}

\begin{theorem}\label{minimi diversi}
Assume $q>\frac{2d}{d+\alpha}$. Let $V$ be a function in $ L^{q'}(\mathbb{R}^d)+\mathring{H}^{\frac{\alpha}{2}}(\mathbb{R}^d)$.
Let $\rho_V$, $\rho^{+}_{V}$ be respectively  the minimizer of $\mathcal{E}^{TF}_\alpha$ in $\mathcal{H}_\alpha$ and in $\mathcal{H}^{+}_{\alpha}$. Then,  $[\rho_{V}]_{+}=\rho^{+}_{V}$ if and only if $\rho_V\ge 0$.
\end{theorem}
\begin{proof}
If $\rho_V$ is nonnegative the conclusion is trivial since $\rho_V$ is a minimizer in the larger set $\mathcal{H}_\alpha$ and by uniqueness it has to coincide with $\rho^{+}_V$.  On the contrary,  assume $[\rho_{V}]_{+}=\rho^{+}_{V}$.  By definition $\rho_V$ and $\rho^{+}_{V}$ satisfy respectively 
\begin{equation}\label{prima eq}
\int_{\mathbb{R}^d}\text{sign}(\rho_V(x))|\rho_V(x)|^{q-1}\varphi(x)dx-\langle \varphi, V\rangle+\langle \rho_V, \varphi\rangle_{\mathring{H}^{-\frac{\alpha}{2}}(\bbr^d)}=0\quad \forall \varphi\in \mathcal{H}_{\alpha},
\end{equation}
\begin{equation}\label{seconda eq}
\int_{\mathbb{R}^d}(\rho^{+}_V)^{q-1}(x)\varphi(x)dx-\langle \varphi, V\rangle+  \langle \rho^{+}_V, \varphi\rangle_{\mathring{H}^{-\frac{\alpha}{2}}(\bbr^d)}\ge 0\quad \forall\varphi\in \mathcal{H}^{+}_\alpha.
\end{equation}
In view of \eqref{prima eq}, we rewrite the quantity $\langle\varphi, V\rangle$ and we replace it into equation \eqref{seconda eq} to obtain
$$
\int_{\mathbb{R}^d}[\rho_V]^{q-1}_{+}\varphi+ \langle [\rho_V]_{+}, \varphi\rangle_{\mathring{H}^{-\frac{\alpha}{2}}(\R^d)}- \langle \rho_V, \varphi\rangle_{\mathring{H}^{-\frac{\alpha}{2}}(\bbr^d)}-\int_{\mathbb{R}^d}\text{sign}(\rho_V)|\rho_V|^{q-1}\varphi\ge 0\quad \forall\varphi\in \mathcal{H}^{+}_\alpha,
$$
which implies that 
\begin{equation}\label{terza eq}
	\int_{\mathbb{R}^d}[\rho_V]^{q-1}_{-}\varphi+ \langle [\rho_V]_{-}, \varphi\rangle_{\mathring{H}^{-\frac{\alpha}{2}}(\bbr^d)}\ge 0\quad\forall\varphi\in \mathcal{H}^{+}_\alpha.
\end{equation}
Hence, if $[\rho_V]_{+}=\rho^{+}_V$,  the function $[\rho_V]_{-}$ belongs to $ \mathring{H}^{-\frac{\alpha}{2}}(\bbr^d)\cap L^{q}(\bbr^d)$ and  satisfies \eqref{terza eq} that is the Euler--Lagrange equation related to the convex functional 
$$\mathcal{F}(\rho)=\frac{1}{q}\int_{\mathbb{R}^d}{|\rho(x)|}^q+\frac{1}{2}\left\|\rho\right\|_{\mathring{H}^{-\frac{\alpha}{2}}(\bbr^d)}^{2}.$$
Clearly,  the unique minimizer of $\mathcal{F}$ in $\mathcal{H}^{+}_{\alpha}$ is the zero function and so $[\rho_V]_{-}$ (which is the minimizer since it satisfies \eqref{terza eq}) must be the zero function.
\end{proof}

Similarly to what has been in the previous section, we can prove regularity of $\rho^{+}_V$. We formulate the following:
\begin{corollary}\label{regularity_pos}
Let $\rho^{+}_V\in \mathcal{H}^{+}_\alpha$ be the non negative minimizer.  If  $V\in\mathring{H}^{\frac{\alpha}{2}}(\R^d)\cap C_b(\bbr^d)$ 
then 
$\rho^{+}_V\in C(\bbr^d)$.
\end{corollary}
\begin{proof}
First of all let's notice that $\cdot \mapsto [\cdot]_{+}$ is a Lipschitz continuous functions. Thus, 
if $q<\frac{d}{\alpha}$ the proof follows the same argument as in Corollary \ref{V_cont} while if $q\ge \frac{d}{\alpha}$, the continuity is a  direct consequence of Corollaries \ref{reg_non_pos} and \ref{q=d_over_alpha}.
\end{proof}

\remark In view of the argument provided in Corollary \ref{regularity_pos}, it's easy to see that we can prove that the analogous version of Corollary \ref{non_local_regularity} for the constrained minimizer $\rho^{+}_V$.

Next, we prove a more general version of Lemma \ref{comparison} that will be used for the proof of Theorem \ref{parte pos}.
\begin{lemma}\label{confronto_generale}
Let $u,v\in \mathring{H}^{\frac{\alpha}{2}}(\mathbb{R}^d)$, $\alpha\in (0,2]$, and $f:\R^d\times \R\to \R_{+}$ such that 
$t\mapsto f(x,t)$ is nondecreasing for a.e. $x\in \R^d$.
Assume that $f(x,u(x)), f(x,v(x))\in  L^{\frac{2d}{d+\alpha}}(\bbr^d)$  and that there exists a smooth domain  $\Omega\subset \mathbb{R}^d$ where 
\begin{equation}\label{ineq_1}
(-\Delta)^{\frac{\alpha}{2}}u+f(x,u)\ge 0\quad \text{in}\ \mathcal{D'}(\Omega), 
\end{equation}
\begin{equation} \label{ineq_2}
(-\Delta)^{\frac{\alpha}{2}}v+f(x,v)\le 0\quad \text{in}\ \mathcal{D}'(\Omega).
\end{equation}
If $\mathbb{R}^d\setminus \Omega\neq \emptyset$ we further require $u\ge v$ in $\mathbb{R}^d\setminus \Omega$. Then,  $u\ge v$ in $\R^d$.
\end{lemma}
\begin{proof}
	The proof is similar to the one of Lemma \ref{comparison}.
\end{proof}
\begin{proposition}\label{segno}
Let $V\in  \mathring{H}^{\frac{\alpha}{2}}(\R^d)\cap C_b(\mathbb{R}^d)$, $\alpha\in (0,d)$. Assume that
\begin{equation}\label{decay}\lim_{|x|\to +\infty}|x|^{d-\alpha}V(x)=0.
\end{equation}
Then $\rho_V$ is sign changing.
\end{proposition}
\begin{proof}
First of all we recall that the basic assumption on $V$ is to be positive on a set of positive Lebesgue measure. In particular, $V$ (and hence $\rho_V$) is not identically zero.

Assume now that $\rho_V\ge 0$.  From \cite{BB}*{Corollary 5.2}, \ref{reg_non_pos}, \ref{q=d_over_alpha} and \ref{V_cont} we infer  that $\rho_V$ is a continuous solution of 
\begin{equation}\label{rho_puntuale}
    \rho^{q-1}_V=V-I_{\alpha}*\rho_V\quad \text{in}\ \bbr^d.
\end{equation}
Moreover, if $|x|$ is sufficiently large, there exist $C,R>0$ independent on $x$ such that
\begin{equation}\label{liminf}
\begin{split}
    (I_{\alpha}*\rho_V)(x)& =A_\alpha\int_{\bbr^d}\frac{\rho_V(y)}{|x-y|^{d-\alpha}}dy\ge \frac{A_\alpha}{(2|x|)^{d-\alpha}}\int_{B(x,2|x|)}\rho_V(y)dy\\
    & \ge \frac{C}{|x|^{d-\alpha}}\int_{B_R}\rho_V(y)dy.
\end{split}
    \end{equation}
   In view of \eqref{liminf} we obtain 
   
   \begin{equation}
       \liminf_{|x|\to +\infty}|x|^{d-\alpha} (I_\alpha*\rho_V) (x)>0.
   \end{equation}
   Hence by, \eqref{decay}, \eqref{rho_puntuale} and \eqref{liminf}
   \begin{equation}
       0\le \limsup_{|x|\to +\infty}|x|^{d-\alpha}\rho_V^{q-1}(x)=\lim_{|x|\to +\infty}|x|^{d-\alpha}V(x)-\limsup_{|x|\to +\infty}|x|^{d-\alpha}(I_{\alpha}*\rho_V(x))<0,
   \end{equation}
   that is a contradiction.
   The same contradiction arises assuming $\rho_V\le 0$.
\end{proof}
  
 \indent  In what follows we are going to prove Theorem \ref{parte pos} by further assuming $V$ being a continuous function with compact support.
To this aim, we will employ a fractional Kato--type inequality. For the proof of such inequality in our setting we refer to the proof \cite{Vincenzo}*{Theorem $3.3$} which adapts to our case by simply taking $m=0$.
Moreover,  for a local version  (case $\alpha=2$) we refer for istance to \cite{Brezis_1}*{Lemma $A.1$}.
 Note also that, the condition on $V$ being compactly supported in particular implies
  \begin{equation}\label{V_green}
 	\lim_{|x|\to +\infty}|x|^{d-\alpha}V(x)=0,
 \end{equation}
from which  $\rho_V$ is sign changing (Proposition \ref{segno}) and the statement of Theorem \ref{parte pos} is not trivial.
 (Note that if $\rho_V$ is non negative then Theorem \ref{parte pos} is trivial since $\rho_V=\rho^{+}_V$ by uniqueness.)
%
 

\vspace{0.3cm}

\noindent \textbf{\large{Proof of Theorem \ref{parte pos}}.}
\begin{proof}
By Corollaries \ref{reg_non_pos}, \ref{q=d_over_alpha}, \ref{V_cont} and \ref{regularity_pos}, the functions $\rho_V$ and $\rho^{+}_V$ are continuous. Moreover, by Proposition \ref{segno}, we have that $[\rho_V]_{+}\not\equiv \rho^{+}_V$ and they solve
\begin{equation}\label{sist_minimi}
\begin{split}
    \rho^{+}_V &=[V-(-\Delta)^{-\frac{\alpha}{2}}\rho^+_V]^{\frac{1}{q-1}}_{+}\quad \text{in}\ \mathcal{D}'(\bbr^d),\\
    \text{sign}(\rho_V)|\rho_V|^{q-1} &=V-(-\Delta)^{-\frac{\alpha}{2}}\rho_V\quad\quad \ \ \, \,  \text{in}\ \mathcal{D}'(\bbr^d).
\end{split}
\end{equation}
Now, we introduce the functions $\varphi_V:=(-\Delta)^{-\frac{\alpha}{2}}\rho_V$ and  $\varphi_V^+:=(-\Delta)^{-\frac{\alpha}{2}}\rho^{+}_V$.  Consequently, we obtain
\begin{equation}\label{sistema_diff}
	\begin{split}
		(-\Delta)^{\frac{\alpha}{2}}\varphi_V  &\le [V-\varphi_V]_{+}^{\frac{1}{q-1}}\quad\ \text{in}\ \mathcal{D}'(\bbr^d),\\
		(-\Delta)^{\frac{\alpha}{2}}\varphi^{+}_V & =[V-\varphi^{+}_V]_{+}^{\frac{1}{q-1}}\quad\ \text{in}\ \mathcal{D}'(\bbr^d).
	\end{split}
\end{equation}
Then, if $\alpha\in (0,2)$, by combining an adaptation of  the proof of \cite{Vincenzo}*{Theorem $3.3$} with \eqref{sistema_diff} we deduce that 
\begin{equation}\label{kato_full}
	(-\Delta)^{\frac{\alpha}{2}}[\varphi_V]_{+}\le \chi_{\left\{\varphi_V\ge 0\right\}}\,\rho_V\le [V-[\varphi_V]_+]_{+}^{\frac{1}{q-1}} \quad  \text{in}\ \mathcal{D}'(\bbr^d).
\end{equation}
On the other hand, the case $\alpha=2$ is classical and follows the same lines.
Note also that the assumption on the support of $V$ combined with its continuity ensures that
 \begin{equation}\label{V_sobolev}
	\begin{split}
		[V-[\varphi_V]_{+}]^{\frac{1}{q-1}}_{+} & \in L^{1}(\bbr^d)\cap L^{\infty}(\R^d)\subset  L^{\frac{2d}{d+\alpha}}(\bbr^d),\\
		[V-\varphi^{+}_V]^{\frac{1}{q-1}}_{+} & \in  L^{1}(\bbr^d)\cap L^{\infty}(\R^d)\subset L^{\frac{2d}{d+\alpha}}(\bbr^d).
	\end{split}
\end{equation} 
Moreover, from \eqref{sist_minimi} and non negativity of $\varphi^{+}_V$, we deduce that $\rho^{+}_V$ has compact support.
Thus, by  \eqref{scalar_H} and the above analysis,  the functions $[\varphi_V]_+$, $\varphi_V^+$ belong to $\mathring{H}^{\frac{\alpha}{2}}(\bbr^d)$ and satisfy
\begin{equation}
	\begin{split}
		(-\Delta)^{\frac{\alpha}{2}}[\varphi_V]_{+} & \le \left[V-[\varphi_V]_{+}\right]_{+}^{\frac{1}{q-1}}\quad \text{in}\ \mathcal{D}'(\bbr^d), \\
		(-\Delta)^{\frac{\alpha}{2}}\varphi_V^+ & =[V-\varphi_V^+]_{+}^{\frac{1}{q-1}}\quad\,\,\,\,\,\,\,\ \text{in}\ \mathcal{D}'(\bbr^d) .
	\end{split}
\end{equation}
Then, Lemma \ref{confronto_generale} applied to $\Omega=\R^d$ and $f(x,u)=-[V(x)-u]^{\frac{1}{q-1}}_{+}$ implies the inequality
\begin{equation}
	\varphi_V^+\ge [\varphi_V]_{+}\quad \text{in}\ \bbr^d.
\end{equation}
Hence, from \eqref{sist_minimi},  inside the open set   $\left\{\rho^{+}_V>0\right\}$ we have
\begin{equation*}
	\begin{split}
		\text{sign}(\rho_V)|\rho_V|^{q-1} - (\rho^{+}_V)^{q-1} & = 	(-\Delta)^{-\frac{\alpha}{2}}\rho^{+}_V-	(-\Delta)^{-\frac{\alpha}{2}}\rho_V\\	
		&= \varphi_V^+-\varphi_V\\
		& \ge  \varphi_V^+ -[\varphi_V]_{+}\ge 0,
	\end{split}
\end{equation*}
proving the desired inequality. 

To conclude, it remains to prove  the existence of a set of positive Lebesgue measure where the inequality is strict. To this aim we argue by contradiction. 

Assume that $[\rho_V]_{+}=\rho^{+}_V$ in the whole of $\R^d$. Then, by Theorem \ref{minimi diversi} we deduce that $\rho_V$ is non negative. However, since $V$ has compact support, Proposition \ref{segno} implies that $\rho_V$ must be sign changing leading to a contradiction.

\end{proof}
\remark\label{not_sharp_V} Note that the assumptions on $V$ to prove the inequality \eqref{ineq_pointwise} can be relaxed. Indeed, the above proof remains the same as long  as \eqref{V_sobolev} is valid and this is true for example if $[V]_{+}\in L^{\frac{2d}{(d+\alpha)(q-1)}}(\R^d)$. 

\vspace{0.3cm}

%

Next, we move our attention on the study of the asymptotic behaviour of the minimizer.

\section{Decay estimates}\label{sec.6}

The main goal of the following section is to study the asymptotic behaviour of the minimizer $\rho_V$ extending the results of \cite{graphene} concerning the special case $d=2$, $\alpha=1$ and $q=\tfrac{3}{2}$. In order to take advantage of the elliptic formulation \eqref{PDE_lap_frac}, we focus on the case $\alpha\in (0,2)$. We further  recall that, in the local case $\alpha=2$, the asymptotic behaviour of solutions for \eqref{PDE_lap_frac} has been widely studied in \cite{maximum}*{Theorem $1.1.2$}  and \cite{Veron} where log--corrections in the decay in the spirit of Theorem \ref{decay_intro} have been proved.  Furthermore, except for Theorem \ref{decay_intro2} where we study decay properties of sign--changing minimizers, we will always assume that $(-\Delta)^{\frac{\alpha}{2}}V\ge 0$ (which ensures non negativity of the minimizer by Proposition \ref{prop_intro}).

 For simplicity's sake, we summerise below the basic assumptions of this subsection. We denote by \hypertarget{condA}{$(\mathcal{A})$} the following set of conditions

\begin{equation*}\label{cond_A}
	(\mathcal{A}) =\begin{cases}  d\in \N\cap [2,\infty) ;\\  \alpha\in (0,2); \\  q\in \big(\frac{2d}{d+\alpha},\infty);
		\vspace{0.1cm}\\
		V\in \big(\mathring{H}^{\frac{\alpha}{2}}(\mathbb{R}^d)\cap C^{\alpha+\ve}_{loc}(\R^d)\cap L^{\infty}(\R^d)\big)\setminus \left\{0\right\};\\
(-\Delta)^{\frac{\alpha}{2}}V\ge 0.
	\end{cases}
\end{equation*}


We further recall that in this context, given two non negative functions $f$ and $g$,  by writing $f(x)\lesssim g(x)$ as $|x|\to +\infty$ we mean 
that there exists $C>0$ such that 
$f(x)\le C g(x)$ for every $x$ sufficiently large and similarly by writing $f(x)\simeq g(x)$ as $|x|\to +\infty$ we understand that $f(x)\lesssim g(x)$ and $g(x)\lesssim f(x)$ as $|x|\to +\infty$.

\remark\label{ass_reg_weak} We further notice that, the regularity assumptions on $V$ imposed above can be weakened according to Proposition \ref{u_pos}.

\vspace{0.2cm}
We begin by studying the sublinear regime corresponding to $q<2$.
\subsection*{Sublinear case: $q<2$}
 \begin{proposition}\label{u_pos}
Let $\alpha\in (0,2)$. Let $V\in \mathring{H}^{\frac{\alpha}{2}}(\R^d)$, $V$ vanishing at infinity and such that $0\le (-\Delta)^{\frac{\alpha}{2}}V\in L^p(\R^d)$ for some $p>\frac{d}{\alpha}$. If $q<2$ then, 
\begin{equation}\label{strett_prop}
	0<u_V(x)<V(x)\quad for\ all\ x\in \bbr^d.
\end{equation}
\end{proposition}
\begin{proof} The proof is similar to \cite{orbital}*{Lemma $7.1$}. 
From Proposition \ref{prop_23} we derive continuity of $V$. This, combined with the fact that $V$ vanishes at infinity in particular implies boundedness of $V$ as well. Thus,  from regularity and superharmonicity of $V$ we derive regularity and non negativity of $u_V$ (see Corollaries \ref{reg_non_pos}, \ref{q=d_over_alpha} and \ref{V_cont} for regularity and Proposition \ref{prop_intro} for non negativity). Moreover, since $I_{\alpha}*\rho_V$ is clearly non negative,  the Euler--Lagrange equation \eqref{u_quasi_ovunque} implies that the function $u_V$ is dominated by the bounded function $V$. In particular, since $V$ vanishes at infinity the same happens to $u_V$. Next, if we define $c:=\left\|u_V\right\|^{\frac{2-q}{q-1}}_{L^{\infty}(\bbr^d)}$, the function $u_V$ solves 
\begin{equation}\label{u_strett_pos}
	((-\Delta)^{\frac{\alpha}{2}}+c)u_V=u_V(c-u^{\frac{2-q}{q-1}}_V)+(-\Delta)^{\frac{\alpha}{2}}V\ge 0\quad \text{in}\ \mathcal{D}'(\bbr^d).
\end{equation}
It is known that the operator $(-\Delta)^{\frac{\alpha}{2}}+c$ has a positive and radially non increasing  Green function $\mathcal{K}_c$ behaving as follows
\begin{equation}\label{green}
	\mathcal{K}_c(r)\lesssim\begin{cases}r^{-d+\alpha} , & \mbox{as }r\to 0,\\ r^{-d-\alpha}, & \mbox{as }r\to +\infty
	\end{cases}
\end{equation}
(see \cite{positive solutions}*{Lemma $4.2$} or \cite{Frank}*{Lemma $C.1$} for more details). From \eqref{green} we deduce that $\mathcal{K}_c\in L^1(\R^d)\cap L^t(\R^d)$ for every $t<\frac{d}{d-\alpha}$.

Let $g$ be the right hand side of \eqref{u_strett_pos}. Then $g$ takes the form of
\begin{equation*}
	g=g_1+g_2, 
\end{equation*}
where $$g_1=u_V(c-u^{\frac{2-q}{q-1}}_V),\quad g_2=(-\Delta)^{\frac{\alpha}{2}}V.$$
As a consenquence, $g_1\in L^{q'}(\R^d)\cap L^{\infty}(\R^d)$ and $g_2\in L^p(\R^d)$, $p>\frac{d}{\alpha}$.
Thus,  $\mathcal{K}_c*g$ defines a bounded uniformly continuous function which solves by construction
\begin{equation}\label{eqK}
	((-\Delta)^{\frac{\alpha}{2}}+c)(\mathcal{K}_c*g)=g\quad \text{in}\ \mathcal{D}'(\R^d),
\end{equation}
cf. \cite{orbital}*{Lemma 7.1, eq. (7.16)}.
Note also that by the properties of convolutions we have that both $\mathcal{K}_c*g_1$ and $\mathcal{K}_c*g_2$ vanish at infinity, i.e., $\mathcal{K}_c*g$ vanishes at infinity as well.
Then, by combining \eqref{u_strett_pos} with \eqref{eqK} we infer that $u_V$ can be written as
\begin{equation}\label{u_quasi_conv}
	u_V=\mathcal{K}_c*g+\bar{u},
\end{equation}
where $\bar{u}$ solves the eigenvalue problem
\begin{equation}\label{eigenv}
	(-\Delta)^{\frac{\alpha}{2}}\bar{u}=-c\bar{u}\quad \text{in}\ \mathcal{D}'(\R^d), \quad c=\|u_V\|^{\frac{2-q}{q-1}}_{L^{\infty}(\R^d)}>0.
\end{equation}
Next, since by \eqref{u_quasi_conv} we further deduce that 
that $\bar{u}\in L^{\infty}(\R^d)$ and $\bar{u}\to 0$ as $|x|\to \infty$ we conclude that $\bar{u}=0$. See \cite{PhDthesis}*{Proposition 3.5.1} for more details. We have therefore prove that 
$u_V$ satisfies 
\begin{equation}\label{kernel}
	u_V(x)=\int_{\bbr^d}\mathcal{K}_c(|x-y|)g(y)dy \quad \forall x\in \R^d.
\end{equation}
From \eqref{kernel}, non negativity of $g$ and positivity of $\mathcal{K}_c$ we derive positivity of $u_V$.
%
%
\end{proof}
 \remark \label{easy_remark} Note that the above proof trivially adapts to the linear case $q=2$.
\begin{lemma}\label{newton_theorem}
Let $ f\in L^{1}(\mathbb{R}^d)$ be non negative  and $\varphi$ be a radially non increasing function defined on  $\R^d$. If $f(x)\le \varphi(|x|)$ for every $x\in \R^d$ and 
$$\lim_{|x|\to +\infty}\varphi(|x|)|x|^{d}=0, $$
then
$$(I_{\alpha}*f)(x)= \frac{A_\alpha\left\|f\right\|_{L^{1}(\mathbb{R}^d)}}{|x|^{d-\alpha}}+o\left(\frac{1}{|x|^{d-\alpha}}\right)\quad \text{as}\ |x|\rightarrow +\infty. $$
\end{lemma}
\begin{proof}
The proof is similar to the one of \cite{graphene}*{Lemma 4.9}.
\end{proof}

Now we state two useful results proved in \cite{Ferrari}. We also refer to \cite{special_function} for the basic definitions and properties of Hypergeometric and Gamma functions that will be used in the sequel.
\begin{theorem}\label{lapl_ferrari}
Assume  $d\ge 2$. Then,  for every  radial function $u\in C^{2}(\overline{\mathbb{R}_{+}})$ such that 
$$\int_{r=0}^{\infty}\frac{|u(r)|}{(1+r)^{d+\alpha}}r^{d-1}dr<\infty,$$
the following equality holds 
\begin{equation}\label{Ferrari_integral}
(-\Delta)^{\frac{\alpha}{2}}u(r)=c_{d,\alpha}r^{-\alpha}\int_{\tau=1}^{\infty}\left[u(r)-u(r\tau)+\left(u(r)-u(r/\tau)\right)\tau^{-d+\alpha}\right]\tau (\tau^2-1)^{-1-\alpha}H(\tau)d\tau.
\end{equation}
where $c_{d,\alpha}$ is a positive constant and $H(\tau)$ is a positive continuous function such that $H(\tau)\sim \tau^{\alpha}$ as $\tau$ goes to infinity.
\end{theorem}
Next, we refer again to \cite{Ferrari}*{eq. $(4.4)$} for Lemma \ref{laplaciano2} below, providing the behaviour at infinity of the fractional Laplacian of polynomial behaving functions. We further refer to the Appendix \ref{Appendix},  Lemma \ref{general_power}, for a similar result providing the decay of the fractional Laplacian a power-type function for a wider range of exponents.
\begin{lemma}\label{laplaciano2}
Let $d\ge 2$, $0<\alpha<2$ and $\beta$ be a positive number. Let $g_\beta:\R^d\rightarrow \R$ be the function defined by $g_\beta(x):=~(1+|x|^{2})^{-\frac{\beta}{2}}$. If   $d-\alpha<\beta<d$ then  
$$(-\Delta)^{\frac{\alpha}{2}}g_\beta(x)\simeq -(1+|x|^{2})^{-\frac{\alpha+\beta}{2}}\quad \text{as}\ |x|\rightarrow +\infty.$$
\end{lemma}
Before investigating the first regime corresponding to the interval $\left(\frac{2d}{d+\alpha},\frac{2d-\alpha}{d}\right)$ we recall   some well known results.

 If $\alpha\in (0,2)$ we define the following quantity
\begin{equation}\label{inf_sobolev}
    \mathcal{I}:=\inf_{u\in \mathring{H}^{\frac{\alpha}{2}}(\bbr^d)}\left\{\|u\|^2_{\mathring{H}^{\frac{\alpha}{2}}(\bbr^d)}:u\ge \chi_{B_1}\right\}.
\end{equation}
It can proved that the infimum defined in \eqref{inf_sobolev} is  achieved by a positive, continuous, and  radially non increasing function $\bar{u}$. Moreover, $\bar{u}$ weakly solves 
\begin{equation}\label{optimal_u}
    \begin{cases} (-\Delta)^{\frac{\alpha}{2}}\bar{u}=0 & \text{in}\ \overline{B_1}^{c}, \\ \quad \quad \quad \,\bar{u}=1 & \text{in}\ B_1,
\end{cases}
\end{equation}
and 
\begin{equation}\label{bar_u_decay}
    \bar{u}(x)\simeq \frac{1}{|x|^{d-\alpha}}\quad \text{as}\ |x|\rightarrow +\infty.\end{equation}

\noindent For the proof we refer for example to \cite{optimal}*{Propositions $3.5$, $3.6$}.\\
 \indent We have now all the ingredients to deduce the desired asymptotic decay in the range  $q\in \big(\frac{2d}{d+\alpha},\frac{2d+\alpha}{d+\alpha}\big)\setminus \left\{\frac{2d-\alpha}{d}\right\}.$

\begin{lemma}\label{range_1}
Assume that  \hyperlink{condA}{$(\mathcal{A})$} holds and $q<\frac{2d-\alpha}{d}$. If there exists a non negative, radially non increasing function  $g:\bbr^d\rightarrow \bbr$  such that 
 \begin{equation}\label{lim_g}
 \lim_{|x|\rightarrow +\infty}g(|x|)|x|^d=0,
 \end{equation}
 and 
 \begin{equation}\label{laplac_integr}(-\Delta)^{\frac{\alpha}{2}}V(x)\le g(|x|),
 \end{equation}
 then,  
 $$u_V(x)\simeq \frac{1}{|x|^{d-\alpha}}\quad as\ |x|\rightarrow +\infty. $$
  \end{lemma}
  \begin{proof}
  By \eqref{lim_g}--\eqref{laplac_integr} and \hyperlink{condA}{$(\mathcal{A})$},  we have that $(-\Delta)^{\frac{\alpha}{2}}V\in L^{1}(\bbr^d)\cap L^{\infty}(\R^d)$. In particular, $(I_{\alpha}* (-\Delta)^{\frac{\alpha}{2}}V)(x)=V(x)$ for every $x\in \R^d$ and, by Lemma \ref{newton_theorem},  \begin{equation}\label{upp_V}
  V(x)=\frac{\big| \! \big|(-\Delta)^{\frac{\alpha}{2}}V\big| \!\big|_{L^{1}(\mathbb{R}^d)}}{|x|^{d-\alpha}}+o\left(\frac{1}{|x|^{d-\alpha}}\right)\quad \text{as}\ |x|\rightarrow +\infty.
  \end{equation}
  Moreover, in view of Proposition \ref{u_pos} we obtain positivity of $u_V$. In addition, by combining  the poinwise Euler--Lagrange equation \eqref{u_quasi_ovunque} with  \eqref{upp_V} there exists $R>0$ large enough and a positive constant $C$ depending on $R$ such that 
  \begin{equation}\label{u_upp_range_1}
  u_V(x)\le C (1+|x|)^{-(d-\alpha)}\quad \text{in}\ \overline{B_R}^{c}.
  \end{equation}
Let  $\delta$ be a positive real number  satisfying $d-\alpha<\delta<d$, and $\lambda>0$. We define  the function \begin{equation}\label{v_intro}
	v^{\delta}_{\lambda}(x):=\lambda \left(\bar{u}(|x|)+(1+|x|^{2})^{-\frac{\delta}{2}}\right),
 \end{equation}
 where $\bar{u}$ is the solution of \eqref{optimal_u}.
Note that the inequality $q<\frac{2d-\alpha}{d}$ implies that  $v^{\delta}_\lambda(x)\in L^{q'}(\bbr^d)$ and $(v^{\delta}_\lambda)^{\frac{1}{q-1}}\in L^{\frac{2d}{d+\alpha}}(\R^d)\subset \mathring{H}^{\frac{\alpha}{2}}(\R^d)$.
Moreover, by putting together \cite{entire}*{Corollary 1.4} with Lemma \ref{laplaciano2} we infer $v^{\delta}_{\lambda}\in \mathring{H}^{\frac{\alpha}{2}}(\R^d)$. Furthermore, there exist $R>0$ large enough, $c_1,c_2$ and $c_3$ positive constants  such that for all $\lambda$ sufficiently small
  \begin{equation}\label{v_minoreuguale}
  \begin{split}
  (-\Delta)^{\frac{\alpha}{2}}v^{\delta}_{\lambda} +(v^{\delta}_\lambda)^{\frac{1}{q-1}}-(-\Delta)^{\frac{\alpha}{2}}V\le   -c_1\lambda|x|^{-(\delta+\alpha)}+c_2\lambda^{\frac{1}{q-1}} |x|^{-\frac{d-\alpha}{q-1}} \le 0 \quad \text{in}\ \overline{B_R}^{c},
  \end{split}
  \end{equation}
provided $\delta$ satisfies $d-\alpha<\delta<(d-\alpha)/(q-1)-\alpha$.
 Such  $\delta$ exists again by the restriction $q<\frac{2d-\alpha}{d}$.
  Hence, by  positivity of $u_V$ (Proposition \ref{u_pos}) we can find $\lambda_0>0$ small such that \eqref{v_minoreuguale} holds and  
  \begin{equation}\label{lamda_sopra}
  \sup_{x\in \overline{B_R}}v^{\delta}_{\lambda_0}(x)\le \inf_{x\in \overline{B_R}
  }u(x).
  \end{equation}
  Finally, choosing $\lambda\le \lambda_0$ satisfying \eqref{v_minoreuguale} yields 
  \begin{equation}
  	 v^{\delta}_{\lambda}\le u_V\quad \text{in}\ \overline{B_R}.
  \end{equation} 
 By Lemma \ref{comparison} we infer 
  \begin{equation}\label{eq_17}
  v^{\delta}_{\lambda_0}\le u_V\quad \text{in}\ \R^d,
  \end{equation}
 that combined with \eqref{u_upp_range_1} yields the thesis.
  \end{proof}

\begin{lemma}\label{range_2}
Assume that  \hyperlink{condA}{$(\mathcal{A})$} holds and  $\frac{2d-\alpha}{d}<q<\frac{2d+\alpha}{d+\alpha}$. If
 \begin{equation}\label{laplac_V_0}
 (-\Delta)^{\frac{\alpha}{2}}V(x)\lesssim \frac{1}{|x|^{\frac{\alpha}{2-q}}}\quad \text{as}\ |x|\rightarrow +\infty,
 \end{equation}
 then, 
 $$u_V(x)\simeq \frac{1}{|x|^{\frac{\alpha(q-1)}{2-q}}}\quad \text{as}\ |x|\rightarrow +\infty.  $$ 
  \end{lemma}
  \begin{proof}
  Let $\lambda$ be a positive constant and  $v_{\lambda}(x)=\lambda (1+|x|^{2})^{-\frac{\alpha(q-1)}{2(2-q)}}$. From Lemma \ref{laplaciano2} and  \cite{garofalo}*{Proposition $2.15$}, $(-\Delta)^{\frac{\alpha}{2}}v_\lambda$ is a continuous function such that
  \begin{equation*}
      (-\Delta)^{\frac{\alpha}{2}}v_\lambda(x) \simeq- |x|^{-\frac{\alpha(q-1)}{(2-q)}-\alpha}\in L^{\frac{2d}{d+\alpha}}(\overline{B_1}^{c}).
  \end{equation*}
  Again by \cite{entire}*{Corollary 1.4} we therefore have that $v_\lambda\in\mathring{H}^{\frac{\alpha}{2}}(\bbr^d)\cap L^{q'}(\bbr^d)$. From Lemma \ref{laplaciano2}, there exist $R_1>0$ large, $c_1,c_3<0$, $c_2>0$  such that for all $\lambda>0$ large enough
  \begin{equation}\label{PDE_37}
  \begin{split}
  (-\Delta)^{\frac{\alpha}{2}}v_{\lambda}+bv^{\frac{1}{q-1}}_{\lambda}-(-\Delta)^{\frac{\alpha}{2}}V \ge \left(\lambda c_1+ \lambda^{\frac{1}{q-1}} c_2+c_3\right)|x|^{-\frac{\alpha}{2-q}}\ge 0\quad \text{in}\ \overline{B_{R_1}}^{c}.
  \end{split}
  \end{equation}
Taking into account  boundedness of $u_V$, there exists $\lambda_1$ large  such that \eqref{PDE_37} holds and 
$$\inf_{x\in \overline{B_{R_1}}}v_{\lambda_1}(x)\ge \sup_{x\in \overline{B_{R_1}}}u_V(x).$$
In view of Lemma \ref{comparison} applied to $\Omega=\overline{B_{R_1}}^c$, we get
$$u_V(x)\le v_{\lambda_1}(x)\quad \text{in}\ \R^d. $$
Similarly, again from Lemma \ref{laplaciano2},  there exist $R_2>0$ large such that for all $\lambda>0$ sufficiently small 
\begin{equation}\label{PDE_37_2}
	\begin{split}
		(-\Delta)^{\frac{\alpha}{2}}v_{\lambda}+bv^{\frac{1}{q-1}}_{\lambda}-(-\Delta)^{\frac{\alpha}{2}}V \le 0\quad \text{in}\ \overline{B_{R_2}}^{c}.
	\end{split}
\end{equation}
Then, by Proposition \ref{u_pos} we can choose $\lambda_2>0$  small such that \eqref{PDE_37_2} is valid and 
$$\sup_{x\in \overline{B_{R_2}}}v_{\lambda_1}(x)\le \inf_{x\in \overline{B_{R_2}}}u_V(x).$$
From Lemma \ref{comparison} we deduce that
$$u_V(x)\ge v_{\lambda_2}(x)\quad \text{in} \ \R^d.$$
This concludes the proof of Lemma \ref{range_2}.
\end{proof}

Before investigating the  range $\frac{2d+\alpha}{d+\alpha}<q< 2$ we study the critical cases $q_1=\frac{2d-\alpha}{d}$ and $q_2=\frac{2d+\alpha}{d+\alpha}$. In the first case we will require some extra restrictions on the parameters arising from the following technical lemma. We mention that Lemma \ref{alfa_grande} extends \cite{graphene}*{Lemma 4.7}.
\begin{lemma}\label{alfa_grande}
Assume that either $b=1$, $\alpha\in (0,2)$ and $d>\alpha$, or $b>1$, $\alpha\in (1,2)$ and $d> \alpha+1$. Let $U\in C^2(\bbr^d)$ be a decreasing function of the form 
\begin{equation*}
U(|x|):=\begin{cases} |x|^{-(d-\alpha)}(\log(e|x|))^{-b}& \mbox{if }|x|\ge 1, \\ v(|x|) & \mbox{if } |x|<1,
\end{cases}
\end{equation*}
where $v\in C^2(\overline{B_1})$ is positive and radially decreasing. Then, 
\begin{equation}\label{asy_1}
	(-\Delta)^{\frac{\alpha}{2}}U (|x|)\simeq -\frac{1}{|x|^d(\log|x|)^{b+1}}\quad as\ |x|\rightarrow +\infty. 
\end{equation}

\end{lemma}
\begin{proof}
We start proving the lower bound.  We split the right hand side of  \eqref{Ferrari_integral} in Theorem \ref{lapl_ferrari} (up to the constant $c_{d,\alpha}$) into three parts $I^{2}_{1}, I^r_{2}$ and $I^{\infty}_r$ respectively, 
\begin{equation}
   I^{2}_1=\int_{\tau=1}^2(\cdot),\ I^{r}_2=\int_{\tau=2}^r(\cdot),\ I^{\infty}_r=\int_{\tau=r}^{\infty}(\cdot).
\end{equation}
For the convenience of the reader we also define 
\begin{equation}\label{phi_37}
    \Phi(\tau,r):=U(r)-U(r\tau)+\left(U(r)-U(r/\tau)\right)\tau^{-d+\alpha},
\end{equation}
and we recall that for the entire proof $r$ is assumed to be large enough (at least $r>2$).
We perform the proof by the following steps:
\begin{itemize}
	\item[$(i)$] We prove that there exists $M>0$ independent on $\tau,r$ such that
	\begin{equation}\label{primo_int}
		r^{d-\alpha}\Phi(\tau,r)\ge -\frac{M\log(e^{b+2}r)}{(\log(er))^{b+2}}(\tau-1)^2\quad \forall \tau\in [1,2];
	\end{equation}
\item[$(ii)$] We prove that 
 \begin{equation}\label{minore_00}
	\Phi(\tau,r)\le 0\quad \forall \tau\in [1,2];
\end{equation}
\item[$(iii)$] We prove that 
\begin{equation}\label{point_3_38}
	I^{r}_2\simeq -\frac{1}{r^{d}(\log(r))^{b+1}}, \quad I^{\infty}_r=o\left(\frac{1}{r^{d}(\log(r))^{b+1}}\right)\quad \text{as}\ r\to +\infty.
\end{equation}
\end{itemize}
Let's start with $(i)$, whose proof is similar to the one of \cite{graphene}*{Lemma 4.7}.

If we set $A:=\log(er)$ and $\tau=e^{x}$, the inequality \eqref{primo_int} becomes equivalent to 
\begin{equation}\label{primo_int_2}
\frac{M}{A^{b+1}}(e^x-1)^2+\left[\frac{1}{A^b}-\frac{1}{(A-x)^b}+e^{(-d+\alpha)x}\left(\frac{1}{A^b}-\frac{1}{(A+x)^b}\right)\right]e^{-(1+\alpha)x}\ge 0.
\end{equation}
Inequality \eqref{primo_int_2} is trivially verified for $x=0$.\\
\indent Assume $x\neq 0$.
In this case \eqref{primo_int_2} holds if there exists $M>0$ such that 
\begin{equation}
M\ge \frac{A^{b+1}}{x^2}\left[\frac{1}{(A-x)^b}-\frac{1}{A^b}+e^{x(\alpha-d)}\left(\frac{1}{(A+x)^b}-\frac{1}{A^b}\right)\right]e^{-x(1+\alpha)}:=F(x,A).
\end{equation}
Since by definition $x\in [0,\log(2)]$, we only need to focus on $x$ small. As a matter of fact, assume $x\ge \varepsilon>0$. 
Then, 
\begin{equation*}
\begin{split}
    |F(x,A)| & \le \frac{A}{\varepsilon^2}\left|\left[\frac{1}{(1-\frac{x}{A})^b}-1+e^{-x(d-\alpha)}\left(\frac{1}{(1+\frac{x}{A})^b}-1\right)\right]\right|\\
    & \le \frac{A}{\varepsilon^2}\left(\frac{1}{(1-\frac{x}{A})^b}-\frac{1}{(1+\frac{x}{A})^b}\right)\le \frac{\log(2)}{\varepsilon^2}\left[\frac{A}{x}\left(\frac{1}{(1-\frac{x}{A})^b}-\frac{1}{(1+\frac{x}{A})^b}\right)\right]:=G\left(\frac{x}{A}\right).
\end{split}
\end{equation*}
Now, it's convenient to set $t:=x/A$. Then, if $A\le L$ for some $L>0$ we clearly obtain boundedness of $F$. It remains to show what happens when $A$ goes to infinity (or equivalently  when $t$ goes to zero). In the latter case we have that 
\begin{equation}
    \lim_{t\to 0}G(t)=\frac{2b\log(2)}{\varepsilon^2}<+\infty
\end{equation}
Consequently, to obtain boundedness of $F$ we only need to focus on $x$ going to zero.
To this aim, we rewrite $F$ as 
\begin{equation*}
    \begin{split}
        F(x,A)=F_1(x,A)+F_2(x,A),
    \end{split}
\end{equation*}
where \begin{equation*}
    \begin{split}
    F_1(x,A)& =\frac{A}{x^2}\left(\frac{1}{(1-\frac{x}{A})^b}+\frac{1}{(1+\frac{x}{A})^b}-2\right)e^{-x(d-\alpha)},\\
    F_2(x,A) & = \frac{A}{x^2}(e^{-x(d+1)}-1)\left(\frac{1}{(1+\frac{x}{A})^b}-1\right).
    \end{split}
\end{equation*}
Since $A\ge 1$, we have
\begin{equation*}
    |F_1(x,A)|\le \left(\frac{A}{x}\right)^2\left|\frac{1}{(1-\frac{x}{A})^b}+\frac{1}{(1+\frac{x}{A})^b}-2\right|=G_1\left(\frac{x}{A}\right).
\end{equation*}
Let's set again $t:=x/A$. Then,  if $t\ge \varepsilon>0$ the function $F_1$ defined above is clearly bounded. On the other hand, 
$$\lim_{t\to 0}G_1(t)=b(b+1)<+\infty,$$
proving that $F_1$ is bounded. 
Next, we notice that the function $F_2$ can be written as the product
$$ F_2(x,A)=G_2(x)\cdot G_3\left(\frac{x}{A}\right)$$
where
\begin{equation*}
    G_2(x)=\frac{e^{-x(d+1)}-1}{x},
\end{equation*}
and 
\begin{equation*}
    G_3\left(\frac{x}{A}\right)=\frac{A}{x}\left(\frac{1}{(1+\frac{x}{A})^b}-1\right).
\end{equation*}
Thus, it's enough to deduce boundedness as $x$ goes to zero and $A$ goes to infinity. In this case, it's clear that $G_2$ is bounded as $x$ goes to zero and, by setting again $t=x/A$, we obtain
$$\lim_{t\to 0}G_3(t)=-b,$$
infering boundedness of $G_3$ as well.
In view of the above arguments, there exists a positive costant $M$ satisfying \eqref{primo_int}.

Next, we prove $(ii)$ for sufficiently large $r$,
or equivalently
\begin{equation}\label{minore_0}
	r^{d-\alpha}\Phi(\tau,r)=\frac{1}{\log(er)^b}-\frac{1}{\log(er/\tau)^b}+\left(\frac{1}{\log(er)^b}-\frac{1}{\log(er\tau)^b}\right)\tau^{-d+\alpha}\le 0\quad \forall \tau\in [1,2].
\end{equation}
By simple computations, \eqref{minore_0} is equivalent to
\begin{equation}\label{ineq_tau_small}
	\left(\frac{\log(er)^b-\log(er\tau)^b}{\log(er\tau)^b-\log(er)^b}\right)\frac{\log(er\tau)^b}{\log(er/\tau)^b}\ge \tau^{-d+\alpha}\quad \forall \tau\in [1,2].
\end{equation}
If we set $s:=\frac{\log(\tau)}{\log(er)}$, the left hand side of \eqref{ineq_tau_small} becomes
\begin{equation}\label{s_variable}
	h(s):=\left(\frac{1-(1-s)^b}{(1+s)^b-1}\right)\left(\frac{1+s}{1-s}\right)^b=1+(b+1)s+o(s)\quad \text{as}\ s\rightarrow 0^+.
\end{equation}
In particular, there exists $\bar{s}>0$ such that $h(s)\ge 1$ provided $s\le \bar{s}$. Therefore, if we choose $r$ large enough such that 
$$\frac{\log(2)}{\log(er)}\le \bar{s}, $$ we infer

\begin{equation}\label{quasi_finito}
	\frac{\log(\tau)}{\log(er)}\le \frac{\log(2)}{\log(er)}\le \bar{s}\quad \forall \tau\in [1,2].
\end{equation}
The inequality \eqref{quasi_finito} implies 
\begin{equation*}
	\left(\frac{\log(er)^b-\log(er\tau)^b}{\log(er\tau)^b-\log(er)^b}\right)\frac{\log(er\tau)^b}{\log(er/\tau)^b}\ge 1 \ge  \tau^{-d+\alpha}\quad \forall\tau\in [1,2],
\end{equation*}
from which \eqref{minore_00} follows.

Next, we investigate point $(iii)$. To this aim, we start by estimating the quantity $I^r_2$. Assume first $b>1$.
Then, we have
\begin{equation}\label{ineq_39}
\begin{split}
& I^{r}_{2}=\\
& =r^{-d}\int_{\tau=2}^{r}\!\left[\frac{1}{(\log(er))^b}-\frac{1}{(\log(er/\tau))^b}\!+\!\tau^{-d+\alpha}\left(\frac{1}{(\log(er))^b}-\frac{1}{(\log(er\tau))^b}\right)\! \right ]\!(\tau-1)^{-1-\alpha}F(\tau)d\tau \\
& \ge  r^{-d}\int_{\tau=2}^{r} \left[\frac{1}{(\log(er))^b}-\frac{1}{(\log(er/\tau))^b}\right]\!(\tau-1)^{-1-\alpha}F(\tau)d\tau\\
& \gtrsim r^{-d}\int_{\tau=2}^{r} \left[\frac{1}{(\log(er))^b}-\frac{1}{(\log(er/\tau))^b}\right]\tau^{-1-\alpha}d\tau=r^{-d}\int_{\tau=2}^{r}\left[g(1)-g(\tau)\right]\tau^{-1-\alpha}d\tau.
\end{split}
\end{equation}
Hence,  by Lagrange Theorem, 
\begin{equation}\label{lagrange_39}
\begin{split}
|g(\tau)-g(1)|& \le \sup_{t\in [1,\tau]}|g'(t)|(\tau-1)=\sup_{t\in [1,\tau]}\left(\frac{b}{t(\log(er)-\log(t))^{b+1}}\right) \cdot (\tau-1)\\
&\le   \frac{b(\tau-1)}{(\log(er)-\log(\tau))^{b+1}}.
\end{split}
\end{equation}
Then, putting together \eqref{ineq_39} with \eqref{lagrange_39}
\begin{equation*}
\begin{split}
I^{r}_{2} & \gtrsim -r^{-d}\int_{\tau=2}^{r}\frac{\tau^{-\alpha}}{(\log(er)-\log(\tau))^{b+1}}d\tau=-\frac{r^{-d}}{(\log(er))^{b+1}}\int_{\tau=2}^{r}\frac{\tau^{-\alpha}}{\left(1-\frac{\log(\tau)}{\log(er)}\right)^{b+1}}d\tau\\
& -\frac{r^{-d}}{(\log(er))^{b+1}}\int_{x=\log(2)}^{\log(r)}\frac{e^{(-\alpha+1)x}}{\left(1-\frac{x}{\log(er)}\right)^{b+1}}dx,
\end{split}
\end{equation*}
where $\tau=e^x$. Then, if $\alpha>1$, an application of dominated convergence theorem yields that
\begin{equation}\label{term_2}
	\begin{split}
		\lim_{r\to +\infty}\int_{x=\log(2)}^{\log(r)}\frac{e^{(-\alpha+1)x}}{\left(1-\frac{x}{\log(er)}\right)^{b+1}}dx=\int_{\log(2)}^{\infty}e^{(-\alpha+1)x}dx=\frac{2^{-(\alpha-1)}}{\alpha-1}.
	\end{split}
\end{equation}
From  \eqref{term_2} 
we infer that 
\begin{equation}\label{1_upper}
I^{r}_2\gtrsim -\frac{1}{r^{d}(\log(r))^{b+1}}\quad \text{as}\ r\to +\infty.  
\end{equation}
In what follows we shall prove the converse inequality of \eqref{1_upper}.

If we define the functions $f_{1}(\tau,r)=\frac{1}{(\log(er/\tau))^{b}}$ and $f_2(\tau,r)=\frac{1}{(\log(er\tau))^{b}}$,  by Lagrange Theorem, there exist $t_{\tau,r}, \bar{t}_{\tau,r}\in (1,\tau)$ such that
\begin{equation}\label{lagrange}
\begin{split}
r^{d-\alpha}\Phi(\tau,r)& = \frac{1}{(\log(er))^b}-\frac{1}{(\log(er/\tau))^b}+\tau^{-d+\alpha}\left(\frac{1}{(\log(er))^b}-\frac{1}{(\log(er\tau))^b}\right)\\
&=f_{1}(1,r)-f_{1}(\tau,r)+\tau^{-d+\alpha}\left(f_{2}(1,r)-f_{2}(\tau,r)\right)\\
&= -\frac{b(\tau-1)}{t_{\tau,r}(\log(er)-\log(t_{\tau,r}))^{b+1}}+\tau^{-d+\alpha} \frac{b(\tau-1)}{\bar{t}_{\tau,r}(\log(er)+\log(\bar{t}_{\tau,r}))^{b+1}}\\
& \le -\frac{b(\tau-1)}{\tau (\log(er))^{b+1}}+\tau^{-d+\alpha}\frac{b(\tau-1)}{(\log(er))^{b+1}}.
\end{split}
\end{equation}
In the sequel, we denote by $F$ the function
$$F(\tau):=\tau (\tau+1)^{-1-\alpha}H(\tau),$$ where $H(\tau)$ is the function defined in Theorem \ref{lapl_ferrari},  eq. \eqref{Ferrari_integral}. 
We further recall that $F(\tau)$ is continuous, positive and bounded on the whole of $\R^d$.
Then, if $d> \alpha+1$,   from \eqref{lagrange} we obtain

\begin{equation*}
\begin{split}
& r^{d-\alpha}\int_{\tau=1}^{r}\Phi(\tau,r)(\tau-1)^{-1-\alpha}F(\tau)d\tau\le \\
& \le -\frac{b}{(\log(er))^{b+1}}\int_{\tau=1}^{r}(\tau-1)^{-\alpha}\tau^{-1}F(\tau)d\tau +\frac{b}{(\log(er))^{b+1}}\int_{\tau=1}^{r}(\tau-1)^{-\alpha}F(\tau) {\tau}^{-d+\alpha}d\tau\\
& = -\frac{b}{(\log(er))^{b+1}}\left[\int_{\tau=1}^{r}(\tau-1)^{-\alpha}(\tau^{-1}-\tau^{-d+\alpha})F(\tau)d\tau\right].
\end{split}
\end{equation*}
Since by assumption $\alpha<2$, $d>\alpha+1$ we have that
$$\lim_{\tau\to 1^{+}}\frac{\tau^{-1}-\tau^{-d+\alpha}}{\tau-1}=d-\alpha-1>0$$ and 
$$  \int_{\tau=1}^{1+\varepsilon}(\tau-1)^{-\alpha}(\tau^{-1}-\tau^{-d+\alpha})F(\tau)d\tau<+\infty,\quad \varepsilon>0.$$
Finally, (from boundedness  and positivity of $F$) we get
$$0< \int_{\tau=1}^{r}(\tau-1)^{-\alpha}(\tau^{-1}-\tau^{-d+\alpha})F(\tau)d\tau\le \int_{\tau=1}^{\infty}(\tau-1)^{-\alpha}(\tau^{-1}-\tau^{-d+\alpha})F(\tau)d\tau<\infty.$$
We have therefore proved that 
\begin{equation}\label{I_2_r}
    I^{r}_2\lesssim -\frac{1}{r^{d}(\log(r))^{b+1}}\quad \text{as}\ r\rightarrow +\infty
\end{equation}
provided $d> \alpha+1$. Then, 
by combining \eqref{1_upper} with \eqref{I_2_r} we obtain 
\begin{equation}\label{I_2_over}
    I^{r}_2\simeq -\frac{1}{r^{d}(\log(r))^{b+1}}\quad \text{as}\ r\to +\infty.
\end{equation}
Now, we further claim that
\begin{equation}\label{term_3}
	I^{\infty}_{r}=o\left(\frac{1}{r^d(\log(r))^{b+1}}\right)\quad \text{as}\ r\rightarrow +\infty
\end{equation}
which will complete the proof of $(i)-(ii)-(iii)$.
Indeed, 
$$r^d|I^{\infty}_r|\lesssim r^{d-\alpha}\int_{\tau=r}^{+\infty}|\Phi(\tau,r)|\tau^{-1-\alpha}d\tau \le  $$
\begin{equation*}
	\begin{split}
		& \le r^{d-\alpha}U(r)\int_{\tau=r}^{+\infty}\tau^{-1-\alpha}d\tau+r^{d-\alpha}\int_{\tau=r}^{+\infty}U(r\tau)\tau^{-1-\alpha}d\tau+r^{d-\alpha}\int_{\tau=r}^{+\infty}(v(r/\tau)-U(r))\tau^{-1-d}d\tau\\
		&   \lesssim  r^{d-\alpha}U(r)\int_{\tau=r}^{+\infty}\tau^{-1-\alpha}d\tau+r^{d-\alpha}\int_{\tau=r}^{+\infty}\tau^{-1-d}d\tau \lesssim  r^{-\alpha}\quad \text{as}\ r\rightarrow +\infty.
	\end{split}
\end{equation*}

Validity of \eqref{asy_1} is therefore obtained since, by combining $(i)-(ii)-(iii)$, 
\begin{equation}
-\frac{1}{r^{d}(\log(r))^{b+1}}+o\left(\frac{1}{r^{d}(\log(r))^{b+1}}\right)\lesssim	I^2_1+I^r_2+I^\infty_r \lesssim -\frac{1}{r^{d}(\log(r))^{b+1}}+o\left(\frac{1}{r^{d}(\log(r))^{b+1}}\right)
\end{equation}

Assume now $b=1$. In this case, the function $\phi(\tau,r)\le 0$ for every $\tau\in [1,r]$ and moreover it can be directly computed the limit 
\begin{equation}\label{limite_b=1}
	\begin{split}
	\lim_{r\to +\infty}(\log(r))^2\int_{\tau=2}^{r}\phi(\tau,r)\tau^{-1-\alpha}d\tau =\frac{2^{-d}}{d^2}+\frac{2^{-d}\log(2)}{d}-\frac{2^{-\alpha}(1+\alpha\log(2))}{\alpha^2}.
	\end{split}
\end{equation}
 See \cite{graphene}*{Lemma 4.7} for an explicit computation in the case $d=2$ and $\alpha=1$. From \eqref{limite_b=1} we obtain the thesis.
\end{proof}
\remark \label{lower_42} Note that,  if $b\ge 1$, $\alpha\in (1,2)$, $d\ge 2$ we obtained the following lower bound on the decay given by 
$$(-\Delta)^{\frac{\alpha}{2}}U(|x|)\gtrsim -\frac{1}{|x|^{d}(\log|x|)^{b+1}}\quad \text{as}\ |x|\to +\infty, $$
without the further restriction $d>\alpha+1$ stated in Lemma \ref{alfa_grande}.
Similarly, if $b\ge 1$, $\alpha\in (0,2)$ and $d>\alpha+1$ we have 
$$(-\Delta)^{\frac{\alpha}{2}}U(|x|)\lesssim -\frac{1}{|x|^{d}(\log|x|)^{b+1}}\quad \text{as}\ |x|\to +\infty,$$
without the restriction $\alpha\in (1,2)$ stated in Lemma \ref{alfa_grande}. These facts will play a role (see Remark \ref{remark_soglia1})  in Corollary \ref{L_1_bound} where only an upper bound on the decay rate of the minimizer $\rho_V$ is needed.
\vspace{0.1cm}

The following result prove the asymptotic decay for another class of functions with logarithmic correction. The purporse of Lemma \ref{2_limit} is to find asymptotic decay of the minimizer $\rho_V$ when $q=\frac{2d+\alpha}{d+\alpha}$ (see Lemma \ref{soglia_2}).

\begin{lemma}\label{2_limit}
Let $U\in C^2(\bbr^d)$ be a decreasing function of the form 
\begin{equation}
U(|x|):=\begin{cases} |x|^{-d}(\log(e|x|))^{\frac{d}{\alpha}}& \mbox{if }|x|\ge 1, \\ v(|x|) & \mbox{if } |x|<1,
\end{cases}
\end{equation}
where $v\in C^2(\overline{B_1})$ is positive and radially decreasing.  Then, 
$$(-\Delta)^{\frac{\alpha}{2}}U(x)\simeq -\frac{(\log|x|)^{\frac{d+\alpha}{\alpha}}}{|x|^{d+\alpha}}\quad as\ |x|\rightarrow +\infty. $$
\end{lemma}
\begin{proof}
Similarly to what has been done in Lemma \ref{alfa_grande}, Theorem  \ref{lapl_ferrari} plays a key role. As a matter of fact, in view of \cite{Ferrari}*{Theorem $1.1$}, we can write the following 
\begin{equation}\label{int_laplaciano}
(-\Delta)^{\frac{\alpha}{2}}U(r)=c_{d,\alpha}r^{-\alpha}\int_{\tau=1}^{+\infty}\Phi(\tau,r)G(\tau)d\tau,
\end{equation}
where $G$ can be written in terms of the Gaussian hypergeometric function
$$G(\tau)=\frac{(2\pi)^{\frac{d}{2}}}{\Gamma(d/2)}\tau^{-1-\alpha} {}_{2}F_1(a,b,c,\tau^{-2}),$$
with parameters 
\begin{equation*}
a  =\frac{d+\alpha}{2},\quad b= 1+\frac{\alpha}{2},\quad 
c= \frac{d}{2}.
\end{equation*}
As we have already done for the proof of Lemma \ref{alfa_grande}, we define 
\begin{equation*}
    \begin{split}
        \Phi(\tau,r):=U(r)-U(r\tau)+\left(U(r)-U(r/\tau)\right)\tau^{-d+\alpha},
    \end{split}
\end{equation*}
and we split the integral \eqref{int_laplaciano} into three parts $I^{2}_{1}, I^r_{2}$ and $I^{\infty}_r$ respectively
\begin{equation}
	I^{2}_1=\int_{\tau=1}^2(\cdot),\ I^{r}_2=\int_{\tau=2}^r(\cdot),\ I^{\infty}_r=\int_{\tau=r}^{\infty}(\cdot).
\end{equation}
 Furthermore,   
since 
$${}_{2}F_1(a,b,c,\tau^{-2})=1+\frac{ab}{\tau^2}+o(\tau^{-3})\quad \text{as}\ \tau\to +\infty,$$
for $\tau\ge 2$ we decompose the function ${}_{2}F_1(a,b,c,\tau^{-2})$ as
\begin{equation}\label{split_W}
	{}_{2}F_1(a,b,c,\tau^{-2})=1+W(\tau),
\end{equation}
where $W(\tau)$ is a continuous function such that
\begin{equation}\label{decay_W}
W(\tau)=o(\tau^{-1})\quad \text{as}\ \tau\to +\infty.
\end{equation}
Taking into account \eqref{split_W}, the quantity $I^{r}_{2}$ in turn splits as
\begin{equation}
\begin{split}
I^{r}_{2}& =\frac{(2\pi)^{\frac{d}{2}}}{\Gamma(d/2)} r^{-\alpha}\int_{\tau=2}^{r}\Phi(\tau,r)\tau^{-1-\alpha}d\tau+\frac{(2\pi)^{\frac{d}{2}}}{\Gamma(d/2)}r^{-\alpha}\int_{\tau=2}^{r}\Phi(\tau,r)\tau^{-1-\alpha}W\left(\tau\right)d\tau\\
& =I^r_{2,1}+I^r_{2,2}.
\end{split}
\end{equation}
By performing an explicit computation, if we denote again by $\Gamma(\cdot, \cdot)$ the incomplete Gamma function we have that
\begin{equation}\label{central_limit}
\begin{split}
	&\int_{\tau=2}^{r}\Phi(\tau,r)\tau^{-1-\alpha}d\tau=\\
	& (d + \alpha)^{-\frac{d+\alpha}{\alpha}} r^{-d}\log(er)^{d+\alpha}
	\left(-\Gamma\left(\frac{d+\alpha}{\alpha}, (d + \alpha) ( \log(2er)\right)+
	\Gamma\left(\frac{d+\alpha}{\alpha}, (d + \alpha) (1 + 2 \log(r)\right)\right) \\
	& -
	\frac{\alpha}{d+\alpha} r^{-d} (-1 + \log(er/2))^{\frac{d+\alpha}{
	\alpha}}) +
  \frac{2^{- d}r^{-2d} (-2^d + r^d)}{d}\log(er)^{\frac{d}{\alpha}} + \frac{2
	r^{-d} (2^{-\alpha} - r^{-\alpha})}{\alpha} \log(er)^{\frac{d}{\alpha}}\\
	&=  -
	\frac{\alpha}{d+\alpha} r^{-d} (-1 + \log(er/2))^{\frac{d+\alpha}{
			\alpha}})+o(r^{-d}\log(er)^{\frac{d+\alpha}{\alpha}}) \quad \text{as}\ r\rightarrow +\infty.
\end{split}
	\end{equation}
Hence, from \eqref{central_limit} we conclude the following
\begin{equation}\label{central_limit_2}
\lim_{r\to+\infty} \frac{r^d}{\log(er)^{\frac{d+\alpha}{\alpha}}}\int_{\tau=2}^{r}\Phi(\tau,r)\tau^{-1-\alpha}d\tau=	-\frac{\alpha}{d+\alpha}.
\end{equation}
Furthermore,  we claim that 
\begin{equation}\label{I_2_2}
	I^{r}_{2,2}=o\left(\frac{\log(er)^{\frac{d+\alpha}{\alpha}}}{r^{d+\alpha}}\right)\quad \text{as}\ r \rightarrow +\infty. 
\end{equation}
To derive \eqref{I_2_2}, by simple estimates
\begin{equation}
\begin{split}
r^{\alpha}|I^{r}_{2,2}|
& \le  r^{-d}(\log(er))^{\frac{d}{\alpha}}\int_{\tau=2}^{\infty}\tau^{-1-\alpha}|W(\tau)|d\tau+2^{\frac{d}{\alpha}} r^{-d}(\log(er))^{\frac{d}{\alpha}}\int_{\tau=2}^{\infty}\tau^{-1-\alpha}|W(\tau)|d\tau\\
& +r^{-d}(\log(er))^{\frac{d}{\alpha}}\int_{\tau=2}^{\infty}\tau^{-1-d}|W(\tau)|d\tau+Cr^{-d}(\log(er))^{\frac{d}{\alpha}}\int_{\tau=2}^{\infty}\tau^{-1}|W(\tau)|d\tau\\
& \lesssim r^{-d}(\log(er))^{\frac{d}{\alpha}}=o\left(\frac{(\log(er))^{\frac{d+\alpha}{\alpha}}}{r^{d}}\right)\quad \text{as}\ r\rightarrow +\infty,
\end{split}
\end{equation}
where  we used continuity of $W$ and \eqref{decay_W}. 
This proves the claim.

Next we focus on $\tau\in [1,2]$. Here,  it is useful to write 
$G$ in the form 
$$G(\tau)=(\tau-1)^{-1-\alpha}F(\tau) $$
where $F(\tau)$ is continuous, bounded and never zero in $[1,2]$.
In this way, by Taylor's formula there exists $\xi_\tau\in (1,\tau)$ such that 
\begin{equation}
\begin{split}
r^{\alpha}I_{1,2} & = \int_{\tau=1}^{2}\Phi(\tau,r)(\tau-1)^{-1-\alpha}F(\tau)d\tau\\
& = \int_{\tau=1}^{2}\frac{\partial^2 }{\partial \tau^2}\Phi(\tau,r)_{|\tau=1}(\tau-1)^{1-\alpha}F(\tau)d\tau+\int_{\tau=1}^{2}\frac{\partial^3 }{\partial \tau^3}\Phi(\tau,r)_{|\tau=\xi_\tau}(\tau-1)^{2-\alpha}F(\tau)d\tau.
\end{split}
\end{equation}
To conclude, by an explicit computation,
\begin{equation}\label{add_1}
	\sup_{\tau\in [1,2]}\left|\frac{\partial^3 }{\partial \tau^3}\Phi(\tau,r)\right|\lesssim \frac{(\log(er))^{\frac{d}{\alpha}}}{r^{d}}\quad \text{as}\ r\to +\infty,
\end{equation}
and 
\begin{equation}\label{add_2}
	\left|\frac{\partial^2 }{\partial \tau^2}\Phi(\tau,r)_{|\tau=1}\right|\lesssim \frac{(\log(er))^\frac{d}{\alpha}}{r^{d}} \quad \text{as}\ r\to +\infty.
\end{equation}
From \eqref{add_1} and \eqref{add_2} we derive that
\begin{equation}\label{I_1_2}
I_{1,2}=o\left(\frac{(\log(er))^{\frac{d+\alpha}{\alpha}}}{r^{d+\alpha}}\right)\quad \text{as}\ r \rightarrow +\infty. 
\end{equation}

Assume now $\tau\in [r,+\infty[$. In this case,
\begin{equation}\label{add_3}
\begin{split}
\Phi(\tau,r) & =r^{-d}(\log(er))^{\frac{d}{\alpha}}+\tau^{-d+\alpha}\left(r^{-d}(\log(er))^{\frac{d}{\alpha}}-v(r) \right)-(r\tau)^{-d}(\log(er\tau))^{\frac{d}{\alpha}}\\
& \le  r^{-d}( \log(er))^{\frac{d}{\alpha}}+\underbrace{F(r,\tau)}_{<0}<U(r).
\end{split}
 \end{equation}

Inequality \eqref{add_3} leads to
$$(-\Delta)^{\frac{\alpha}{2}}U(r)\lesssim -\frac{(\log(er))^{\frac{d+\alpha}{\alpha}}}{r^{d+\alpha}}+o\left(\frac{(\log(er))^{\frac{d+\alpha}{\alpha}}}{r^{d+\alpha}}\right)\quad \text{as}\ r\rightarrow +\infty.$$
In order to conclude,  let's also notice that 
\begin{equation}\label{add_4}
|I^{\infty}_r|  \lesssim r^{-\alpha}\int_{\tau=r}^{\infty}|\Phi(\tau,r)|\tau^{-1-\alpha}d\tau \lesssim \frac{\log(er))^{\frac{d}{\alpha}}}{r^{d+2\alpha}}+\frac{1}{r^{d+\alpha}}+\frac{(\log(er))^{\frac{d}{\alpha}}}{r^{2d+\alpha}}.
\end{equation}
Finally, by combining \eqref{central_limit_2}, \eqref{I_2_2}, \eqref{I_1_2} and \eqref{add_4} we deduce
$$(-\Delta)^{\frac{\alpha}{2}}U(r)\simeq - \frac{(\log(er))^{  \frac{d+\alpha}{\alpha}}}{r^{d+\alpha}} \quad \text{as}\ r\rightarrow +\infty.$$
This concludes the proof.
\end{proof}


In the following two Lemmas (Lemma \ref{soglia_1} and \ref{soglia_2}), we finally derive the decay of the minimizer $\rho_V$ at the critical values $q=\frac{2d-\alpha}{d}$ and $q=\frac{2d+\alpha}{d+\alpha}$. For the particular case $d=2$, $\alpha=1$ and $q=\frac{3}{2}$ we refer to \cite{graphene}*{Proposition 4.8}.
\begin{lemma}\label{soglia_1}
Assume that \hyperlink{condA}{$(\mathcal{A})$} holds and $q=\frac{2d-\alpha}{d}$. Assume that either $\alpha\in (1,2)$ and $d>\alpha+1$ or $q=\tfrac{3}{2}$. If 
 \begin{equation}\label{laplac_V_1}
 (-\Delta)^{\frac{\alpha}{2}}V(x)\lesssim \frac{1}{|x|^{d}(\log|x|)^{\frac{d}{\alpha}}}\quad as\ |x|\to +\infty,
 \end{equation}
 then
 \begin{equation}\label{u_46}
     u_V(x)\simeq \frac{1}{|x|^{d-\alpha}(\log|x|)^{\frac{d-\alpha}{\alpha}}}\quad as\ |x|\rightarrow +\infty. 
    \end{equation}
 \end{lemma}
 \begin{proof}
 Let $\lambda>0$. We define $v_{\lambda}(x)=\lambda U(|x|)$ where $U(|x|)$ is the barrier function defined in Lemma  \ref{alfa_grande}, and  $b=\frac{d-\alpha}{\alpha}$.
First of all we notice that $(-\Delta)^{\frac{\alpha}{2}}v_\lambda\in C_b(\bbr^d)\cap L^{1}(\bbr^d)\subset \mathring{H}^{-\frac{\alpha}{2}}(\bbr^d)$. In particular,  $v_\lambda=I_{\alpha}*(-\Delta)^{\frac{\alpha}{2}}v_\lambda\in \mathring{H}^{\frac{\alpha}{2}}(\R^d)$. Moreover, it's easy to see that $v_\lambda\in L^{q'}(\bbr^d)$. 
 Now, we argue like in the proof of Lemma \ref{range_2}. 
 Consequently,  we can find $R>0$ sufficiently large such that for all $\lambda>0$ large enough
  \begin{equation*}
  \begin{split}
  (-\Delta)^{\frac{\alpha}{2}}v_{\lambda}+ v^{\frac{d}{d-\alpha}}_{\lambda}-(-\Delta)^{\frac{\alpha}{2}}V\ge 0\quad \text{in}\ \overline{B_R}^{c} ,
  \end{split}
  \end{equation*}
  and 
  \begin{equation}
     \lambda\inf_{x\in \overline{B_R}}U(|x|)\ge \sup_{x\in \overline{B_R}}u(x) .
    \end{equation}
  Thus, from Lemma \ref{comparison} applied to $\Omega=\overline{B_R}^{c}$ we obtain
  \begin{equation}\label{sup_2}
 u(x)\le v_{\lambda_1}(x)\quad \text{in}\ \mathbb{R}^d.
 \end{equation}
  On the other hand, there exist $R_0>0$ large, such that for all $\lambda>0$ sufficiently small we have
   \begin{equation*}
  \begin{split}
   (-\Delta)^{\frac{\alpha}{2}}v_{\lambda}+b v^{\frac{d}{d-\alpha}}_{\lambda}-(-\Delta)^{\frac{\alpha}{2}}V
   \le 0\quad \text{in}\ \overline{B_{R_0}}^{c}
   \end{split}
  \end{equation*}
  and 
  \begin{equation}\label{reason_lambda}
  \sup_{x\in \overline{B_{R_0}}}v_{\lambda}(x)\le  \inf_{x\in \overline{B_{R_0}}}u(x).
  \end{equation}
Again from Lemma \ref{comparison}  we have 
\begin{equation}\label{sup_1}
	u(x)\ge v_{\lambda}(x)\quad \text{in}\ \bbr^d.
\end{equation}
By putting together inequality \eqref{sup_2} with \eqref{sup_1} we get \eqref{u_46}.
 \end{proof}

\remark\label{remark_soglia1}
Note that, in view of Remark \ref{lower_42}, we still obtain a lower bound (respectively upper bound) on the decay rate of $u_V$ if $\alpha\in (1,2)$ (respectively $d>\alpha+1$). Namely, under the assumptions  \hyperlink{condA}{$(\mathcal{A})$} and $q=\frac{2d-\alpha}{d}$ the following occur:
\begin{itemize}
\item[$(i)$]	If either  $\alpha\in (1,2)$ and $d\ge 2$, or $q=\tfrac{3}{2}$ then
\begin{equation}\label{uineq_46}
	u_V(x)\lesssim \frac{1}{|x|^{d-\alpha}(\log|x|)^{\frac{d-\alpha}{\alpha}}}\quad as\ |x|\rightarrow +\infty;
\end{equation}
\item[$(ii)$] If either  $\alpha\in (0,2)$ and $d>\alpha+1$, or $q=\tfrac{3}{2}$ then
\begin{equation}\label{uinq_462}
	u_V(x)\gtrsim \frac{1}{|x|^{d-\alpha}(\log|x|)^{\frac{d-\alpha}{\alpha}}}\quad as\ |x|\rightarrow +\infty.
\end{equation}
\end{itemize}

\vspace{0.2cm}
Up to our knowledge, the following asymptotic result is new and was never observed before.
 \begin{lemma}\label{soglia_2}
Assume that \hyperlink{condA}{$(\mathcal{A})$} holds and $q=\frac{2d+\alpha}{d+\alpha}$. If 
\begin{equation}\label{laplac_V_2}
(-\Delta)^{\frac{\alpha}{2}}V(x)\lesssim \frac{(\log(e|x|))^{\frac{d+\alpha}{\alpha}}}{|x|^{d+\alpha}}\quad  as\ |x|\to +\infty,
\end{equation}
then 
$$u_V(x)\simeq \frac{(\log|x|)^{\frac{d}{\alpha}}}{|x|^{d}}\quad as\ |x|\rightarrow +\infty. $$
\end{lemma}
 \begin{proof}
 The proof uses Lemma \ref{2_limit} and follows the same lines of Lemma \ref{soglia_1}. 
 \end{proof}
 
We now investigate the range $\frac{2d+\alpha}{d+\alpha}<q<2$.
 \begin{lemma}\label{max_decay}
Assume $\gamma>d$. Let $U\in C^2(\bbr^d)$ be a decreasing function of the form 
\begin{equation}
U(|x|):=\begin{cases} |x|^{-\gamma}& \mbox{if }|x|\ge 1, \\ v(|x|) & \mbox{if } |x|<1,
\end{cases}
\end{equation}
where $v\in C^2(\overline{B_1})$ is positive and radially decreasing. Then, 
$$(-\Delta)^{\frac{\alpha}{2}}U(x)\simeq -\frac{1}{|x|^{d+\alpha}}\quad as\ |x|\rightarrow +\infty. $$
 \end{lemma}
 \begin{proof}
 As usual we use Theorem \ref{lapl_ferrari} and  we split the integral on the right hand side of \eqref{Ferrari_integral} (up to the constant $c_{d,\alpha}$) into three parts  $I^{2}_{1}, I^r_{2}$ and $I^{\infty}_r$ respectively
 \begin{equation}
 	I^{2}_1=\int_{\tau=1}^2(\cdot),\ I^{r}_2=\int_{\tau=2}^r(\cdot),\ I^{\infty}_r=\int_{\tau=r}^{\infty}(\cdot).
 \end{equation}
First of all let's notice that for every $\tau\in [1,r]$
\begin{equation}\label{tau_48}
\begin{split}
\phi(\tau,r)&=U(r)-U(r\tau)+\left(U(r)-U(r/\tau)\right)\tau^{-d+\alpha}\\
&=r^{-\gamma}(1-\tau^{-\gamma}-\tau^{-d+\alpha+\gamma}+\tau^{-d+\alpha}).
\end{split}
\end{equation}
Next, we prove that 
\begin{equation}\label{phi_minore_0}
	\phi(\tau,r)\le 0\quad \forall \tau\in [1,r].
\end{equation}
By \eqref{tau_48},  \eqref{phi_minore_0} 
is equivalent to 
$$\tau^{-\gamma}\le \tau^{-d+\alpha},$$
that is true by the assumption on $\gamma.$
Thus, 
\begin{equation}\label{eq_47_bis}
\begin{split}
    I^{r}_2 &=r^{-\gamma-\alpha}\int_{\tau=2}^{r}\left(1-\tau^{-\gamma}+\tau^{-d+\alpha}(1-\tau^{\gamma})\right)(\tau-1)^{-1-\alpha}F(\tau)d\tau\\
    & \simeq r^{-\gamma-\alpha}\int_{\tau=2}^{r}\left(1-\tau^{-\gamma}+\tau^{-d+\alpha}(1-\tau^{\gamma})\right)\tau^{-1-\alpha}d\tau\\
    & =-\frac{r^{-d-\alpha}}{\gamma-d}+o(r^{-d-\alpha})\quad \text{as}\ r\rightarrow +\infty,
\end{split}
\end{equation}
where by  $F$ we denoted the function
$$F(\tau)=\tau (\tau+1)^{-1-\alpha}H(\tau) ,$$
and $H$ is defined in \cite{Ferrari}*{Theorem $1.1$}.
Furthermore, by Taylor's formula,  there exists  $\xi_\tau\in [1,\tau]$ such that 
\begin{equation}\label{primoint}
    I^{2}_1=r^{-\gamma-\alpha}\int_{\tau=1}^{2}\frac{\phi''(1)}{2}(\tau-1)^{-1-\alpha}F(\tau)d\tau+r^{-\gamma-\alpha}\int_{\tau=1}^{2}\frac{\phi'''(\xi_\tau)}{6}(\tau-1)^{2-\alpha}F(\tau)d\tau.
\end{equation}
Next, from the fact that
\begin{equation*}
    \sup_{t\in [1,2]}|\phi'''(t)|<\infty,
\end{equation*}
 \eqref{primoint} implies 
\begin{equation}\label{stima_mezzo}
I^{2}_1=o(r^{-d-\alpha})\quad \text{as}\ r\rightarrow +\infty.
\end{equation}
Moreover, by definition of $U$, 
\begin{equation}
	\begin{split}
	I^{\infty}_r & =r^{-\gamma-\alpha}\int_{\tau=r}^{\infty}(1-\tau^{-\gamma}+\tau^{-d+\alpha})(\tau-1)^{-1-\alpha}F(\tau)d\tau-r^{-\alpha}\int_{\tau=r}^{\infty}U(r/\tau)(\tau-1)^{-1-\alpha}F(\tau)d\tau\\
	& = -r^{-\alpha}\int_{\tau=r}^{\infty}U(r/\tau)(\tau-1)^{-1-\alpha}F(\tau)d\tau+o(r^{-d-\alpha})\quad \text{as}\ r\to +\infty.
	\end{split}
\end{equation}
As a consequence, since $U$ is decreasing and $(\tau-1)^{-1-\alpha}F(\tau)$ is a positive function such that $(\tau-1)^{-1-\alpha}F(\tau)\simeq \tau^{-1-\alpha}$ as $\tau\to +\infty$, we deduce the two sided estimate
\begin{equation}
-U(0) r^{-\alpha}\int_{\tau=r}^{\infty}\tau^{-d-1}d\tau+o(r^{-d-\alpha})\lesssim	I^{\infty}_r\lesssim - U(1)r^{-\alpha}\int_{\tau=r}^{\infty}\tau^{-d-1}d\tau+o(r^{-d-\alpha}).
\end{equation}
The above two sided bound yields
\begin{equation}\label{stima_fine}
I^{\infty}_r\simeq -r^{-d-\alpha}+o(r^{-d-\alpha})\quad \text{as}\ r\rightarrow +\infty.
\end{equation}
Putting together \eqref{eq_47_bis}, \eqref{stima_mezzo} and \eqref{stima_fine} we derive the thesis.

 \end{proof}
 As a consequence,  we formulate the following result. For the sake of simplicity, we shorten the proof since it follows the same lines of the proof of Lemma \ref{soglia_1} or \ref{soglia_2}.
 \begin{lemma}\label{m_minore}
Assume that \hyperlink{condA}{$(\mathcal{A})$} holds and  $\frac{2d+\alpha}{d+\alpha}<q<2$. If
\begin{equation}\label{laplac_V_3}
(-\Delta)^{\frac{\alpha}{2}}V(x) \lesssim \frac{1}{|x|^{d+\alpha}}\quad  as\ |x|\to +\infty,
\end{equation}
then
$$u_{V}(x)\simeq \frac{1}{|x|^{(d+\alpha)(q-1)}}\quad as\ |x|\rightarrow +\infty. $$
\end{lemma}
\begin{proof}
First of all we notice that the function $U(x)$ defined in Lemma \ref{max_decay}, with $\gamma=(d+\alpha)(q-1)$, belongs to $\mathring{H}^{\frac{\alpha}{2}}(\bbr^d)\cap L^{q'}(\bbr^d)$. Then, arguing as we did for example in the proof of Lemma \ref{soglia_1} we obtain the thesis.
\end{proof}

\subsection*{Linear case: $q=2$.}
As we have pointed out in Remark \ref{easy_remark}, under the assumptions \hyperlink{condA}{$(\mathcal{A})$} we have positivity of the minimizer $\rho_V$. Moreover, using some well known results we can prove the precise asymptotic decay of $\rho_V$.
\begin{lemma}\label{q=2}
	Assume \hyperlink{condA}{$(\mathcal{A})$} holds. If $q=2$ and 
	$$(-\Delta)^{\frac{\alpha}{2}}V(x)\lesssim \frac{1}{|x|^{d+\alpha}}\quad \text{as}\ |x|\to +\infty$$
	then
	$$\rho_V(x)\simeq \frac{1}{|x|^{d+\alpha}}\quad \text{as}\ |x|\to +\infty.$$
\end{lemma}
\begin{proof}
	Since by assumption $(-\Delta)^{\frac{\alpha}{2}}V\ge 0$, the function $u_V(=\rho_V)$  weakly solves 
	$$(-\Delta)^{\frac{\alpha}{2}}u_V+u_V\ge 0\quad \text{in}\ \overline{B_1}^c.$$ 
	Now, from \cite{positive solutions}*{Lemma 4.2} there exists a positive, continuous function $v$ belonging to $H^{\alpha}(\R^d)$ weakly and pointwisely solving the equation
	\begin{equation*}
		(-\Delta)^{\frac{\alpha}{2}}v+v=0\quad \text{in}\ \overline{B_1}^c,
	\end{equation*}
	and $v(x)\gtrsim \frac{1}{|x|^{d+\alpha}}$ as $|x|\to +\infty$.
\end{proof}
In particular, from positivity of $u_V$, for all $\lambda>0$ small enough  the function $v_\lambda(x)=\lambda v(x)$ satisfies
\begin{equation}
	(-\Delta)^{\frac{\alpha}{2}}v_\lambda+v_\lambda=0\quad \text{in}\ \overline{B_1}^c
\end{equation}
and $v_\lambda\le u_V$ in $B_1$. Then, by Lemma \ref{comparison} we obtain that 
\begin{equation}\label{eq_51_1}
	\frac{1}{|x|^{d+\alpha}}\lesssim v_\lambda(x)\le u_V(x)\quad \text{as}\ |x|\to +\infty.
\end{equation}

On the other hand, taking into account the decay assumptions on $(-\Delta)^{\frac{\alpha}{2}}V$, the function $u_V$ weakly solves 
\begin{equation}\label{weak_ball}
	(-\Delta)^{\frac{\alpha}{2}}u_V+u_V\le \frac{C}{(1+|x|^2)^{\frac{d+\alpha}{2}}}\quad \text{in}\ \overline{B_R}^c
\end{equation}
for some positive constant $C$, and $R>0$ sufficiently large.
Then, by \cite{asymptotics}*{Lemma A.1} there exists $v\in H^{\frac{\alpha}{2}}(\R^d)$ such that $v=u_V$ in $B_R$ and  weakly sastisfying 
\begin{equation}\label{ineq_53}
	(-\Delta)^{\frac{\alpha}{2}}v+v=\frac{C}{(1+|x|^2)^{\frac{d+\alpha}{2}}}\quad \text{in}\ \overline{B_R}^c.
\end{equation}
By combining \eqref{weak_ball} with \eqref{ineq_53} we obtain that the function $U_V:=u_V-v$ weakly solves 
\begin{equation}
	\begin{cases} (-\Delta)^{\frac{\alpha}{2}}U_V+U_V\le 0\quad &\mbox{in}\ \overline{B_R}^c, \\ U_V=0 & \mbox{in }\!  \overline{B_R}.
	\end{cases}
\end{equation}
Then, by the comparison principle (see e.g. \cite{asymptotics}*{Lemma A.3}) we conclude that $U_V\le 0$ in $\R^d$.
Finally, (see \cite{asymptotics}*{Lemma 3.5 and A.1}) we recall that $v$ satifisfies 
\begin{equation}\label{decay_for_v}
	\limsup_{|x|\to +\infty}|x|^{d+\alpha}v(x)<\infty
\end{equation}
The thesis follows by putting together \eqref{eq_51_1} with \eqref{decay_for_v}
\subsection*{Superlinear case: $q>2$.}

In the following subsection, we first show how to recover positivity of $u_V$ (and so of $\rho_V)$ when $q\in (2,+\infty)$.
This result is strongly related to the non local nature of the fractional Laplace operator $(-\Delta)^{\frac{\alpha}{2}}$. Indeed, such result fails in the local case of the Laplacian $-\Delta$ (see \cite{Diaz}*{Corollary $1.10$, Remark $1.5$}). We refer  also to \cite{positive solutions}*{Theorem 1.3} for a similar argument.
\begin{proposition}\label{superlinear}
	Assume that  \hyperlink{condA}{$(\mathcal{A})$}  holds and $q>2$.  Then $\rho_V(x)>0$ for all $x\in \bbr^d$. 
\end{proposition}
\begin{proof}
 The idea is to obtain enough regularity in order to use the singular integral representation  \eqref{sing_laplac} for the fractional Laplacian. Notice that non negativity of $\rho_V$ implies that  $I_\alpha*\rho_V$ is always a convergent integral (see again \cite{BB}*{Corollaries 5.2, 5.3}). In what follows, we prove that $u_V$ defined by \eqref{U},  belongs to $C^{\alpha+\varepsilon}_{loc}(\bbr^d)$ for some $\varepsilon>0$ (see \eqref{holder_integrable} for the definition of the latter space). In particular, from such regularity, $u_V$ is proved to be not only a weak (and so distributional) solution of \eqref{PDE_lap_frac} but also a pointwise solution of the latter.

 First of all we notice that, from non negativity of $\rho_V$ and boundedness of $V$, the Euler--Lagrange equation \eqref{quasi_ovunque_sec_4} implies boundedness of $\rho_V$ and $I_{\alpha}*\rho_V$. Hence, by \cite{Silvestre}*{Proposition $2.9$} we have the following
 \begin{itemize}
 	\item[$(i)$] If $\alpha\le 1$   then $I_\alpha*\rho_V\in C^{0,\beta}(\bbr^d)$ for every $\beta<\alpha$;
 	\item [$(ii)$] If $\alpha>1$ then $I_\alpha*\rho_V\in C^{1,\beta}(\bbr^d)$ for every $\beta<\alpha-1$.
 \end{itemize}
	\vspace{0.1cm}
	\noindent \underline{Case $0<\alpha<1$.}
	 In this case, by arguing as in Corollary \ref{non_local_regularity} and Remark \ref{regul_q_big}, we can improve $(i)$ up  to  $I_\alpha*\rho_V\in C^{0,\beta}_{loc}(\bbr^d)$ for every $\beta\in (0,\beta_0)$, where $\beta_0$ is defined as 
	\begin{equation*}
		\beta_0=\min\left\{1,\frac{\alpha(q-1)}{q-2}\right\}.
	\end{equation*}
	In particular, since $\tfrac{\alpha(q-1)}{q-2}>\alpha$ and  $V$ is regular, the Euler---Lagrange equation \eqref{u_quasi_ovunque} implies that $u_V\in C^{0,\gamma}_{loc}(\bbr^d)$ for some $\alpha<\gamma<1$. Thus, by \cite{garofalo}*{Proposition $2.15$}, formula \eqref{sing_laplac} holds true and $(-\Delta)^{\frac{\alpha}{2}}u_V\in C(\bbr^d)$.
	\vspace{0.3cm}
	
	\noindent\underline{Case $1<\alpha<2$.}
	By  \cite{Silvestre}*{Proposition $2,9$} we obtain that $I_\alpha*\rho_V\in C^{1,\beta}(\bbr^d)$ for every $\beta\in (0,\alpha-1)$. Thus, from \eqref{u_quasi_ovunque} and regularity of $V$ we have that $u_V\in C^{1,\beta}_{loc}(\bbr^d)$ for every $\beta\in (0,\alpha-1)$. Now we consider three  subcases. 
	\begin{itemize}
		\item[$(i)$] If $\alpha+\frac{1}{q-1}<2$, we first apply Proposition \ref{holder_riesz} obtaining that $I_\alpha*\rho_V\in C^{1,\beta_1}_{loc}(\bbr^d)$ with $\beta_1=\alpha+\frac{1}{q-1}-1$. Then, putting together \eqref{u_quasi_ovunque} with regularity of $V$, we derive that $u_V\in C^{1,\min\left\{\beta_1, \alpha-1 +\varepsilon \right\}}_{loc}(\bbr^d)$. Since $\beta_1>\alpha-1$, \cite{garofalo}*{Proposition $2.15$} again implies that formula \eqref{sing_laplac} holds.
		\item[$(ii)$] If $\alpha+\frac{1}{q-1}=2$, we can choose $0<\delta<\frac{1}{q-1}$ obtaining (arguing as in the above case) that $I_{\alpha}*\rho_V\in C^{1,\gamma^*}(\bbr^d)$, where $\gamma^*:=\alpha+\frac{1}{q-1}-1-\delta$. Then $u_V\in C^{1,\min\left\{\gamma^*, \alpha-1+\varepsilon\right\}}_{loc}(\bbr^d)$.
		\item [$(iii)$] If $2<\alpha+\frac{1}{q-1}<3$, arguing as in case $(i)$, we deduce that $I_\alpha*\rho_V\in  C^{2,\beta_2}_{loc}(\R^d)$ where $\beta_2=\alpha+\frac{1}{q-1}-2$. Then, taking into account \eqref{u_quasi_ovunque} and regularity of $V$, we conclude again validity of \eqref{sing_laplac} for $(-\Delta)^{\frac{\alpha}{2}}u_V$.
	\end{itemize}
	Note that since $q>2$ and $\alpha\in (1,2)$, cases $(i),(ii)$ or $(iii)$ are the only admissible ones.
	\vspace{0.3cm}
	
	\noindent \underline{Case $\alpha=1$.} 
	This case follows by a similar argument.

As a consequence of the above analysis, all of $u_V, I_{\alpha}*\rho_V$ and $V$ belong to $C^{\alpha+\varepsilon}_{loc}(\R^d)\cap \mathcal{L}^{1}_{\alpha}(\R^d)$.  Next, we prove that $u_V$ is positive in the whole of $\R^d$. Assume by contradiction that there exists $x_0$ such that $u_V(x_0)=0$ and consider a ball $B$ around $x_0$. Then, by \cite{X.book}*{Corollary  $2.2.29$} applied to $f=(-\Delta)^{\frac{\alpha}{2}}V-u^{\frac{1}{q-1}}_V\in L^{1}_{loc}(\R^d)$, (and continuity of the functions involved) we have that
	\begin{equation}\label{pointwise_PDE}
		(-\Delta)^{\frac{\alpha}{2}}u_V(x)=-u^{\frac{1}{q-1}}_V(x)+(-\Delta)^{\frac{\alpha}{2}}V(x)\quad \forall x\in B,
	\end{equation}
	where in \eqref{pointwise_PDE} the operator $(-\Delta)^{\frac{\alpha}{2}}$ is understood in the sense of \eqref{sing_laplac}.
	In particular, from non negativity of $u_V$, we have that
	\begin{equation}
		\label{neg_laplac}
		(-\Delta)^{\frac{\alpha}{2}}u_V(x_0)=-\frac{C_{d,\alpha}}{2}\int_{\bbr^d}\frac{u_V(x_0+y)+u_V(x_0-y)}{|y|^{d+\alpha}}dy<0.
	\end{equation}
	On the other hand, by definition of $x_0$, \eqref{pointwise_PDE}, \eqref{neg_laplac} and non negativity of $(-\Delta)^{\frac{\alpha}{2}}V$,  we obtain
	\begin{equation}
		0\le (-\Delta)^{\frac{\alpha}{2}}V(x_0)<0,
	\end{equation}
	leading to a contradiction. We have therefore proved that $u_V$ (and so $\rho_V$) is positive. This concludes the proof.
\end{proof}
\begin{lemma}\label{lower_bound_48}
	Assume that  \hyperlink{condA}{$(\mathcal{A})$} holds and $q>2$. If  $$(-\Delta)^{\frac{\alpha}{2}}V(x)\lesssim \frac{1}{|x|^{d+\alpha}}\quad \text{as}\ |x|\to +\infty$$ 
then 
	\begin{equation}\label{lower_48}
	u_V(x)\simeq \frac{1}{|x|^{(d+\alpha)(q-1)}}\quad \text{as}\ |x|\to +\infty.
	\end{equation}
\end{lemma}
\begin{proof}
	Let  $\gamma=(d+\alpha)(q-1)$ and $U$ be the function defined by
	\begin{equation}\label{U_49}
		U(x)=\begin{cases} v(|x|) & \mbox{if }|x|\le 1\\ |x|^{-\gamma}, & \mbox{if }|x|>1,
		\end{cases}
	\end{equation}
	where $v\in C^2(\overline{B_1})$ is such that $U$ is $C^2(\R^d)$ and radially decreasing.  
	By  Lemma \ref{max_decay}, there exists $c>0$ and $\overline{R}>0$, such that 
	\begin{equation}\label{ineq_bis_49}
		(-\Delta)^{\frac{\alpha}{2}}U(x)\le -\frac{c}{|x|^{d+\alpha}}\quad \text{if}\ |x|\ge \overline{R}.
	\end{equation}
Next, we claim that there exists $R>0$ large, $\lambda=\lambda(R)>0$ such that 
\begin{equation}\label{claim_40}
	\begin{cases} 	&(-\Delta)^{\frac{\alpha}{2}}v_{\lambda(R),R}+(v_{\lambda(R),R})^{\frac{1}{q-1}}\le 0  \ \quad  \mbox{if}\ |x|> 1,\\ &  v_{\lambda(R),R}(x)\le u_V(x)\quad \quad \quad \qquad\ \qquad \, \mbox{if } |x|\le 1.
	\end{cases}
\end{equation}
	where $v_{\lambda,R}$ is defined by 
	\begin{equation}\label{def_48}
		v_{\lambda,R}(x)=\lambda U(Rx),
	\end{equation}
	and $U$ is defined in \eqref{U_49}.
	It's well known (cf. \cite{garofalo}*{Lemma $2.6$}) that 
	\begin{equation}\label{risc_49}
		(-\Delta)^{\frac{\alpha}{2}}v_{\lambda,R}(x)=\lambda R^\alpha ((-\Delta)^{\frac{\alpha}{2}} U)(Rx).
	\end{equation}
	Then, if $R\ge \max\left\{\overline{R},1\right\}$, by combining \eqref{ineq_bis_49} with \eqref{risc_49} we infer 
	\begin{equation}\label{ineq_49_minore}
		\begin{split}
			(-\Delta)^{\frac{\alpha}{2}}v_{\lambda,R}(x)+(v_{\lambda,R}(x))^{\frac{1}{q-1}} & \le \left(-\frac{c \lambda}{ R^d}+\frac{\lambda^{\frac{1}{q-1}}}{R^{d+\alpha}}\right) \frac{1}{|x|^{d+\alpha}}\quad \text{in}\ \overline{B_1}^{c}\\
			& \le 0 \quad \text{in}\ \overline {B_1}^{c},
		\end{split}
	\end{equation}
	provided $\lambda\ge (c^{-1} R^{-\alpha})^{\frac{q-1}{q-2}}$. In particular, if we consider $\lambda=\lambda(R)=(c^{-1} R^{-\alpha})^{\frac{q-1}{q-2}}$ then the inequality \eqref{ineq_49_minore} holds true. 
	Furthermore, if $|x|\le 1$, by monotonicity of $U$ and Proposition \ref{superlinear} we have that 
	\begin{equation*}
		v_{\lambda(R),R}(x)=\lambda(R)U(Rx)\le (c^{-1}R^{-\alpha})^{\frac{q-1}{q-2}}U(0)\le \inf_{x\in \overline{B_1}}u_V(x),
	\end{equation*}
provided $R$ is large enough.
	Putting together \eqref{ineq_49_minore} with the above inequality we obtain \eqref{claim_40}. Next, it's easy to verify that for all $R>0$ the function $v_{\lambda(R),R}$ belongs to $ \mathring{H}^{\frac{\alpha}{2}}(\R^d)\cap L^{q'}(\R^d)$ and $v^{\frac{1}{q-1}}_{\lambda(R),R}\in \mathring{H}^{-\frac{\alpha}{2}}(\R^d)\cap L^{q}(\R^d)$. Then, by applying Lemma \ref{comparison} to $\Omega=\R^d$, we obtain that
	\begin{equation}\label{first_40} u_V(x)\ge v_{\lambda(R),R}(x)\quad \forall x\in \R^d,
		\end{equation}
	provided $R$ is large enough.
	Next, we prove the upper bound. To this aim we employ  again the function $v_{\lambda,R}$ defined in \eqref{def_48} and we claim that there exists $R>0$ small, $\lambda=\lambda(R)>0$ and $\tilde{R}=\tilde{R}(R)$ such that 
	\begin{equation}\label{claim_41}
	\begin{cases} 	&(-\Delta)^{\frac{\alpha}{2}}v_{\lambda(R),R}+(v_{\lambda(R),R})^{\frac{1}{q-1}}\ge (-\Delta)^{\frac{\alpha}{2}}V  \ \quad  \mbox{if}\ |x|> \tilde{R}(R),\\ &  v_{\lambda(R),R}(x)\ge u_V(x)\quad \quad \quad \qquad\ \qquad\qquad\quad\ \, \mbox{if } |x|\le  \tilde{R}(R).
	\end{cases}
\end{equation}
By  Lemma \ref{max_decay}, there exists $c_1>0$ and $R_1>0$, such that 
\begin{equation}\label{ineq_bis_49_2}
	(-\Delta)^{\frac{\alpha}{2}}U(x)\ge -\frac{c_1}{|x|^{d+\alpha}}\quad \text{if}\ |x|\ge R_1.
\end{equation}
Furthermore, by the decay assumption on $(-\Delta)^{\frac{\alpha}{2}}V$ there exists $c_2>0$ and $R_2>0$ such that 
\begin{equation}\label{v_41}
	(-\Delta)^{\frac{\alpha}{2}}V(x)\le \frac{c_2}{|x|^{d+\alpha}}\quad \text{if}\ |x|\ge R_2.
\end{equation}
Hence, by combining \eqref{risc_49}, \eqref{ineq_bis_49_2} with \eqref{v_41} we infer that 
\begin{equation}\label{lower_41}
\begin{split}
	(-\Delta)^{\frac{\alpha}{2}}v_{\lambda, R}+(v_{\lambda,R})^{\frac{1}{q-1}}-(-\Delta)^{\frac{\alpha}{2}}V\ge \frac{1}{|x|^{d+\alpha}}\left(-\frac{\lambda c_1}{R^d}+\frac{\lambda^{\frac{1}{q-1}}}{R^{d+\alpha}}-c_2\right)\quad \text{if}\ |x|> \max\left\{\frac{1}{R}, \frac{R_1}{R}, R_2\right\}.
\end{split}
\end{equation}
By an explicit computation, if we choose $\lambda=\lambda(R)=(c^{-1}_1 R^{-\frac{\alpha}{2}})^{\frac{q-1}{q-2}}$, there exists $\bar{R}$ small such that for every $R\le \bar{R}$ the right hand side of \eqref{lower_41} is positive. In particular, we can assume that $R$ is small enough such that  $\max\left\{\frac{1}{R}, \frac{R_1}{R}, R_2\right\}=R^{-1}\max\left\{R_1,1\right\}$. Furthermore, if $|x|\le R^{-1}\max\left\{R_1,1\right\}$, monotonicity of $U$ yields have that 
\begin{equation}
	v_{\lambda(R),R}(x)= \lambda(R)U(Rx)\ge (c^{-1}_1 R^{-\frac{\alpha}{2}})^{\frac{q-1}{q-2}} \left(\max\left\{R_1,1\right\}\right)^{-(d+\alpha)(q-1)}\ge \|u_V\|_{L^{\infty}(\R^d)},
\end{equation}
provided $R$ is small enough. We have therefore proved that \eqref{claim_41} holds with $\tilde{R}(R)=R^{-1}\max\left\{R_1,1\right\}$ and $R$ sufficiently small. Hence, again by applying Lemma \ref{comparison} we conclude that there exists $R>0$ small such that 
\begin{equation}\label{second_41}
	u_V(x)\le v_{\lambda(R),R}(x)\quad \forall x\in \R^d.
\end{equation}
By combining \eqref{first_40} with \eqref{second_41} we derive \eqref{lower_48}.
\end{proof}

\subsection*{Proof of Theorems \ref{decay_intro} and \ref{decay_intro2}.}
\remark\label{restrizione_deca} For the sake of clarity, we  highlights below all the decay assumptions on $(-\Delta)^{\frac{\alpha}{2}}V$ required to prove Lemmas \ref{range_1},  \ref{range_2}, \ref{soglia_1},  \ref{soglia_2}, \ref{m_minore}, \ref{q=2} and  \ref{lower_bound_48}: 
\begin{itemize} 
	\item  If $\frac{2d}{d+\alpha} <q<\frac{2d-\alpha}{d}$ we require  \eqref{lim_g} and \eqref{laplac_integr};
	\item If $q=\frac{2d-\alpha}{d}$ and, either $\alpha\in (1,2)$ and $d>\alpha+1$ or $q=\tfrac{3}{2}$, we require $$(-\Delta)^{\frac{\alpha}{2}}V(x)\lesssim \frac{1}{|x|^{d}(\log|x|)^{\frac{d}{\alpha}}}\quad \text{as}\ |x|\to +\infty;$$  
	
	\item If $\frac{2d-\alpha}{d}<q<\frac{2d+\alpha}{d+\alpha}$ we require
	$$(-\Delta)^{\frac{\alpha}{2}}V(x)\lesssim \frac{1}{|x|^{\frac{\alpha}{2-q}}}\quad \text{as}\ |x|\to +\infty;$$
	\item If $q=\frac{2d+\alpha}{d+\alpha}$ we require $$(-\Delta)^{\frac{\alpha}{2}}V(x)\lesssim \frac{(\log|x|)^{\frac{d+\alpha}{\alpha}}}{|x|^{d+\alpha}}\quad \text{as}\ |x|\to +\infty;$$
	\item If $q>\frac{2d+\alpha}{d+\alpha}$ we require 
	$$ (-\Delta)^{\frac{\alpha}{2}}V(x)\lesssim \frac{1}{|x|^{d+\alpha}}\quad \text{as}\ |x|\to +\infty. $$
\end{itemize}
\indent In particular, if \begin{equation}\label{uiversal_decay}
	(-\Delta)^{\frac{\alpha}{2}}V(x)\lesssim \frac{1}{|x|^{d+\alpha}}\quad \text{as}\ |x|\to +\infty,
\end{equation}
all of the above assumptions are satisfied and so, for simplicity, in the statement of Theorems \ref{decay_intro}, \ref{decay_intro2} we require \eqref{uiversal_decay} instead of distinguishing between the different regimes of $q$. 

\vspace{0.3cm}

\noindent \textbf{\large{Proof of Theorem \ref{decay_intro}}.}
\begin{proof}
\noindent
By combining \eqref{laplac_V_fast} with Lemma \ref{newton_theorem} we deduce that $V\in L^{\infty}(\R^d)$. Furthermore,  by combining the results of Lemmas \ref{range_1},  \ref{range_2}, \ref{soglia_1},  \ref{soglia_2}, \ref{m_minore}, \ref{q=2} and \ref{lower_bound_48} we obtain the desired decay estimates $(i),(ii),(iii),(iv)$ and $(v)$. It remains only to prove that 
\begin{equation}\label{sharp_decay_bis}
	\lim_{|x|\to +\infty}|x|^{\frac{d-\alpha}{q-1}}\rho_V(x)=\left[A_{\alpha}\left(\big| \! \big|(-\Delta)^{\frac{\alpha}{2}}V\big| \!\big|_{L^{1}(\mathbb{R}^d)}-\left\|\rho_V\right\|_{L^{1}(\mathbb{R}^d)}\right)\right]^{\frac{1}{q-1}}
\end{equation}
provided $\frac{2d}{d+\alpha}<q<\frac{2d-\alpha}{d}$.
 To this aim, by the decay assumption \eqref{uiversal_decay}, Lemma \ref{newton_theorem} implies that
 \begin{equation}\label{V_1}
 	V(x)= \frac{A_\alpha\big| \! \big|(-\Delta)^{\frac{\alpha}{2}}V\big| \!\big|_{L^{1}(\mathbb{R}^d)}}{|x|^{d-\alpha}}+o\left(\frac{1}{|x|^{d-\alpha}}\right) \quad \text{as}\ |x|\rightarrow +\infty.
 \end{equation}
 Moreover, $(i)$ ensures that in such regimes of $q$ the minimizer $\rho_V$ satisfies the assumption of Lemma \ref{newton_theorem}. In particular, 
 \begin{equation}\label{rho_1}
 	(I_{\alpha}*\rho_V)(x)= \frac{A_\alpha\left\|\rho_V\right\|_{L^{1}(\mathbb{R}^d)}}{|x|^{d-\alpha}}+o\left(\frac{1}{|x|^{d-\alpha}}\right)\quad \text{as}\ |x|\rightarrow +\infty. 
 \end{equation}
 Combining \eqref{V_1} and \eqref{rho_1} with the Euler--Lagrange equation \eqref{rho_puntuale}
 we obtain the following asymptotic expansion for $\rho_V$
\begin{equation}\label{uni_decau}
	\rho_V(x)=\frac{\left[A_{\alpha}\left(\big| \! \big|(-\Delta)^{\frac{\alpha}{2}}V\big| \!\big|_{L^{1}(\mathbb{R}^d)}-\left\|\rho_V\right\|_{L^{1}(\mathbb{R}^d)}\right)\right]^{\frac{1}{q-1}}}{|x|^{\frac{d-\alpha}{q-1}}}+o\left(\frac{1}{|x|^{\frac{d-\alpha}{q-1}}}\right)\quad \text{as}\ |x|\rightarrow +\infty.
\end{equation}
In order to conlude it's enough to notice that, by positivity of $\rho_V$ and validity of $(i)$, the coefficient of $|x|^{-\frac{d-\alpha}{q-1}}$ in the right hand side of \eqref{uni_decau} must be positive, concluding the proof.
\end{proof}
Next, we apply Theorem \ref{decay_intro} to estimate the $L^1$--norm of the minimizer. As before, the decay assumption on $(-\Delta)^{\frac{\alpha}{2}}V$ in Corollary \ref{L_1_bound} can be weakened accordingly to the range of $q$.

\begin{corollary}\label{L_1_bound}
	Assume that \hyperlink{condA}{$(\mathcal{A})$} holds. If  $$(-\Delta)^{\frac{\alpha}{2}}V(x)\lesssim \frac{1}{|x|^{d+\alpha}}\quad as\ |x|\to +\infty$$
	then the following possibilities hold: 
	\begin{itemize}
		\item[$(i)$] If  $ \frac{2d}{d+\alpha}<q<\frac{2d-\alpha}{d}$  then 
		$$\left\|\rho_V\right\|_{L^{1}(\mathbb{R}^d)}<\big| \! \big|(-\Delta)^{\frac{\alpha}{2}}V\big| \!\big|_{L^{1}(\mathbb{R}^d)}; $$
		\item[$(ii)$] If $q=\frac{2d-\alpha}{d}$ and,  either $1< \alpha<2$ or $q=\tfrac{3}{2}$,  then
		$$\left\|\rho_V\right\|_{L^{1}(\mathbb{R}^d)}=\big| \! \big|(-\Delta)^{\frac{\alpha}{2}}V\big| \!\big|_{L^{1}(\mathbb{R}^d)};$$
		\item[$(iii)$] If $q>\frac{2d-\alpha}{d}$ then 
		$$\left\|\rho_V\right\|_{L^{1}(\mathbb{R}^d)}=\big| \! \big|(-\Delta)^{\frac{\alpha}{2}}V\big| \!\big|_{L^{1}(\mathbb{R}^d)}.$$
	\end{itemize}
\end{corollary}
\begin{proof}
	Proceding as in the proof of Theorem \ref{decay_intro}, 
	we obtain \begin{equation}\label{uni_decau_bis}
		\rho_V(x)=\frac{\left[A_{\alpha}\left(\big| \! \big|(-\Delta)^{\frac{\alpha}{2}}V\big| \!\big|_{L^{1}(\mathbb{R}^d)}-\left\|\rho_V\right\|_{L^{1}(\mathbb{R}^d)}\right)\right]^{\frac{1}{q-1}}}{|x|^{\frac{d-\alpha}{q-1}}}+o\left(\frac{1}{|x|^{\frac{d-\alpha}{q-1}}}\right)\quad \text{as}\ |x|\rightarrow +\infty.
	\end{equation}
	Thus, under the assumptions of Lemma \ref{range_1}, again from positivity of $\rho_V$  we conclude validity of $(i)$.
	On the other hand, 
Theorem \ref{decay_intro}--$(ii)$--$(iii)$--$(iv)$--$(v)$ and Remarks \ref{lower_42}--\ref{remark_soglia1}  tell us that $\rho_V$ decays faster than $|x|^{-\frac{d-\alpha}{q-1}}$ at infinity and so the coefficient of $|x|^{-\frac{d-\alpha}{q-1}}$ in the right hand side of \eqref{uni_decau_bis} is zero.
%
This completes the proof.
\end{proof}
In what follows, before studying the case when the minimizer is sign changing,  we apply Corollary \ref{L_1_bound} to an important  family of potentials $V$.
\begin{corollary}\label{V_z}
	Let $d\ge 2$, $0<\alpha<d$ and $V_Z$ be the function defined by
	\begin{equation}\label{t_1}
		V_Z(x)=\frac{Z  A_\alpha}{(1+|x|^2)^{\frac{d-\alpha}{2}}},
	\end{equation}
	where $Z$ is positive constant and $A_\alpha$ defined in \eqref{coulomb}. The following statements hold: 
	\begin{itemize}
		\item[$(i)$] If $\frac{2d}{d+\alpha}<q<\frac{2d-\alpha}{d}$ then
		$$ 0<\left\|\rho_{V_Z}\right\|_{L^{1}(\bbr^d)}<Z;$$
		\item[$(ii)$] If $q=\frac{2d-\alpha}{d}$ and, either $1<\alpha<2$ or $q=\tfrac{3}{2}$, then
		$$\left\|\rho_{V_Z}\right\|_{L^{1}(\bbr^d)}=Z;$$
		\item[$(iii)$] If $q>\frac{2d-\alpha}{d}$ then 
		$$\left\|\rho_{V_Z}\right\|_{L^{1}(\bbr^d)}=Z.$$
	\end{itemize}
\end{corollary}
\begin{proof}
	By \cite{critical_exp}*{Theorem $1$}, (see also \cite{minimizer}*{Theorem $1.1$}) we deduce that 
	$V_Z$ defined by \eqref{t_1} belongs to $\mathring{H}^{\frac{\alpha}{2}}(\bbr^d)$ and solves 
	\begin{equation*}
		(-\Delta)^{\frac{\alpha}{2}}V_Z(x)=C(d,\alpha,Z)V_Z^{\frac{d+\alpha}{d-\alpha}}(x),
	\end{equation*}
	where  $C(d,\alpha,Z)$ is a  suitable normalisation costant. In particular, $V_Z$ satisfies all the conditions of Corollary \ref{L_1_bound}.
	The thesis follows combining
	\begin{equation*}
		A_\alpha\big| \! \big|(-\Delta)^{\frac{\alpha}{2}}V_Z\big| \!\big|_{L^{1}(\mathbb{R}^d)}=\lim_{|x|\to +\infty}|x|^{d-\alpha}V_Z(x)= Z A_\alpha,
	\end{equation*}
	with Corollary \ref{L_1_bound}.
\end{proof}


\vspace{0.2cm}
Next, we deduce asymptotic upper bounds for the absolute value of the minimizer ($|\rho_V|$) without requiring any sign restriction. In particular, notice that the upper bounds are the same as the one proved in Theorem \ref{decay_intro}.

\vspace{0.3cm}

\noindent \textbf{\large{Proof of Theorem \ref{decay_intro2}}.}
\begin{proof}
	By the decay assumption \eqref{laplac_V_fast_2} on $(-\Delta)^{\frac{\alpha}{2}}V$ and the regularity of $V$, we conclude that $V=I_{\alpha}* ((-\Delta)^{\frac{\alpha}{2}}V)$. In particular $V(x)\lesssim (1+|x|)^{\alpha-d}$ (cf. Lemma \ref{newton_theorem}). Moreover, again by an adaptation of \cite{Vincenzo}*{Theorem $3.3$} we infer that
	\begin{equation}
		(-\Delta)^{\frac{\alpha}{2}}|u_V|+|u_V|^{\frac{1}{q-1}}\le |(-\Delta)^{\frac{\alpha}{2}}V|\quad \text{in}\ \mathcal{D}'(\R^d).
	\end{equation}
	Note that, by combining the regularity of $V$ with \eqref{laplac_V_fast_2}  we further have $|(-\Delta)^{\frac{\alpha}{2}}V|\in L^{\frac{2d}{d+\alpha}}(\R^d)\subset \mathring{H}^{-\frac{\alpha}{2}}(\R^d)$. Moreover, since by assumption  $|\rho_V|=|u_V|^{\frac{1}{q-1}}\in \mathring{H}^{-\frac{\alpha}{2}}(\R^d)$  and $|u_V|\in \mathring{H}^{\frac{\alpha}{2}}(\R^d)$ (see \eqref{scalar_H}), the function $|u_V|$ weakly solves  (by density)
	\begin{equation}\label{PDE_52}
		(-\Delta)^{\frac{\alpha}{2}}|u_V|+|u_V|^{\frac{1}{q-1}}\le |(-\Delta)^{\frac{\alpha}{2}}V|\quad \text{in}\ \mathring{H}^{-\frac{\alpha}{2}}(\R^d).
	\end{equation}
	Notice that the PDE obtained by taking the equality in \eqref{PDE_52} satisfies a weak comparison princinple in the spirit of Lemma \ref{comparison}.
	Then, if $q\neq 2$ it's enough to use the same upper barries provided in Lemmas \ref{range_1}, \ref{range_2}, \ref{soglia_2}, \ref{m_minore}, \ref{lower_bound_48} and Remark \ref{remark_soglia1}. 
	On the other hand, If $q=2$ we can argue as in Lemma \ref{q=2}.
\end{proof}
\remark As we have already noticed for the proof of Theorem \ref{decay_intro}, the decay assumption \eqref{laplac_V_fast_2} can be replaced with the ones of Lemmas \ref{range_1},  \ref{range_2}, \ref{soglia_1},  \ref{soglia_2} and  \ref{m_minore} accordingly to the range of $q$.

\appendix
\section{Appendix}\label{Appendix}
For the sake of completess, we provide in the following lemma the asymptotic decay for the fractional Laplacian of a power--type function for any negative exponent different from $\alpha-d$. The proof is similar to the one of Lemma \ref{max_decay}.
 \begin{lemma}\label{general_power}
	Let $d\ge 2$ and $0<\alpha<2$, $\gamma>0$ and $\gamma\neq d-\alpha$. Let $U\in C^2(\bbr^d)$ be a decreasing function of the form 
	\begin{equation}
		U(|x|):=\begin{cases} |x|^{-\gamma}& \mbox{if }|x|\ge 1, \\ v(|x|) & \mbox{if } |x|<1,
		\end{cases}
	\end{equation}
	where $v\in C^2(\overline{B_1})$ is positive and radially decreasing. Then the following possibilities hold:
	\begin{itemize}
		\item[$(i)$] If $0<\gamma<d-\alpha$ then	$$(-\Delta)^{\frac{\alpha}{2}}U(x)\simeq \frac{1}{|x|^{\gamma+\alpha}}\quad as\ |x|\rightarrow +\infty; $$
		\item[$(ii)$] If $d-\alpha<\gamma<d$ then  
		$$(-\Delta)^{\frac{\alpha}{2}}U(x)\simeq - \frac{1}{|x|^{\gamma+\alpha}}\quad as\ |x|\rightarrow +\infty; $$
		\item[$(iii)$] If $\gamma=d$ then
		$$(-\Delta)^{\frac{\alpha}{2}}U(x)\simeq - \frac{\log(|x|)}{|x|^{d+\alpha}}\quad as\ |x|\rightarrow +\infty; $$
		\item[$(iv)$] If $\gamma>d$ then 
		$$(-\Delta)^{\frac{\alpha}{2}}U(x)\simeq -\frac{1}{|x|^{d+\alpha}}\quad as\ |x|\rightarrow +\infty. $$
\end{itemize}
In particular, if $\gamma>\frac{d-\alpha}{2}$ then $U\in \mathring{H}^{\frac{\alpha}{2}}(\R^d)$.
\end{lemma}
\begin{proof}
	If $\gamma>d$ we directly refer to Lemma \ref{max_decay}. If not, 
	as we did in the previous section, we apply Theorem \ref{lapl_ferrari} and  we split the integral on the right hand side of \eqref{Ferrari_integral} (up to the constant $c_{d,\alpha}$) into three terms  $I^{2}_{1}, I^r_{2}$ and $I^{\infty}_r$ respectively
	\begin{equation}
		I^{2}_1=\int_{\tau=1}^2(\cdot),\ I^{r}_2=\int_{\tau=2}^r(\cdot),\ I^{\infty}_r=\int_{\tau=r}^{\infty}(\cdot).
	\end{equation}
	First of all let's notice that for every $\tau\in [1,r]$
	\begin{equation}\label{phi_bh_bis}
		\begin{split}
			\phi(\tau,r)&=U(r)-U(r\tau)+\left(U(r)-U(r/\tau)\right)\tau^{-d+\alpha}\\
			&=r^{-\gamma}(1-\tau^{-\gamma}-\tau^{-d+\alpha+\gamma}+\tau^{-d+\alpha}).
		\end{split}
	\end{equation}
	From  \eqref{phi_bh_bis} we see that 
	\begin{itemize}
		\item[$(1)$] If $\gamma>d-\alpha$ then $\phi(\tau,r)\le 0\quad \forall \tau\in [1,r]$;
\item[$(2)$] If $\gamma<d-\alpha$ then $\phi(\tau,r)\ge 0\quad \forall \tau\in [1,r]$.
\end{itemize}
	Thus, by arguing as in the proof of Lemma \ref{max_decay}, if $r$ is sufficiently large we obtain the following possibilities:
	\begin{equation}\label{eq_47_bis_bis}
		\begin{split}
			I^{r}_2 &=r^{-\gamma-\alpha}\int_{\tau=2}^{r}\left(1-\tau^{-\gamma}+\tau^{-d+\alpha}(1-\tau^{\gamma})\right)(\tau-1)^{-1-\alpha}F(\tau)d\tau\\
			& \simeq r^{-\gamma-\alpha}\int_{\tau=2}^{r}\left(1-\tau^{-\gamma}+\tau^{-d+\alpha}(1-\tau^{\gamma})\right)\tau^{-1-\alpha}d\tau\\
			& =\begin{cases}r^{-\gamma-\alpha}\left[\frac{2^{-\alpha}}{\alpha}-\frac{2^{-\gamma-\alpha}}{\gamma+\alpha}+\frac{2^{\gamma-d}}{\gamma-d}+\frac{2^{-d}}{d}\right]+o(r^{-\gamma-\alpha}), & \mbox{if } \gamma<d \\ -r^{-d-\alpha} \log(r)+o(r^{-d-\alpha} \log(r)), & \mbox{if }\gamma=d.
			\end{cases}
		\end{split}
	\end{equation}
	Furthermore, if we denote by  $F$  the function
	$$F(\tau)=\tau (\tau+1)^{-1-\alpha}H(\tau) ,$$
	and $H$ is defined in \cite{Ferrari}*{Theorem $1.1$},   there exists  $\xi_\tau\in [1,\tau]$ such that 
	\begin{equation}\label{primoint_bis}
		I^{2}_1=r^{-\gamma-\alpha}\left[\int_{\tau=1}^{2}\frac{\phi''(1)}{2}(\tau-1)^{-1-\alpha}F(\tau)d\tau+\int_{\tau=1}^{2}\frac{\phi'''(\xi_\tau)}{6}(\tau-1)^{2-\alpha}F(\tau)d\tau\right].
	\end{equation}
In particular, 
	\begin{equation}\label{stima_mezzo_bis}
	I^{2}_1=\begin{cases} C_\gamma r^{-\gamma-\alpha}, & \mbox{if}\ \gamma<d\\
		o(r^{-d-\alpha} \log(r)), & \mbox{if }\gamma=d,
	\end{cases}
\end{equation}
where $C_\gamma$ is the coefficient of $r^{-\gamma-\alpha}$ in \eqref{primoint_bis}. Note that we stressed the dependence on $\gamma$ since from $(1)$--$(2)$ we have that $C_\gamma$ is negative (respectively positive) if $\gamma>d-\alpha$ (respectively $\gamma<d-\alpha$).
	Next, similarly to the proof of Lemma \ref{max_decay} we obtain that 
	\begin{equation}\label{stima_fine_bis}
	I_r^{\infty}=\begin{cases} o(r^{-\gamma-\alpha}), & \mbox{if } \gamma<d \\ o(r^{-d-\alpha} \log(r)), & \mbox{if }\gamma=d
	\end{cases}
	\end{equation}
Putting together \eqref{eq_47_bis_bis}, \eqref{stima_mezzo_bis} and \eqref{stima_fine_bis} we derive the asymptotic decay.

Finally, if $\gamma>\frac{d-\alpha}{2}$, from the decay rates just proved and continuity of $(-\Delta)^{\frac{\alpha}{2}}U$ (see again \cite{garofalo}*{Proposition 2.15}),   it's straightforward to see that $(-\Delta)^{\frac{\alpha}{2}}U\in L^{\frac{2d}{d+\alpha}}(\R^d)\subset \mathring{H}^{-\frac{\alpha}{2}}(\R^d)$. The conclusion follows for example from Remark \cite{entire}*{Corollary 1.4}.
	
\end{proof}


\vspace{0.3cm}

\noindent\textbf{Acknowledgments.} This research was funded by EPSRC Maths DTP 2020. The author is very thankful to Prof. V. Moroz for providing helpful remarks and feedbacks.


\begin{thebibliography}{999}

\bib{abatangelo}{article}{
  author={Abatangelo, Nicola},
title={Higher-order fractional Laplacians: an overview},
conference={
	title={Bruno Pini Mathematical Analysis Seminar 2021},
},
book={
	series={Bruno Pini Math. Anal. Semin.},
	volume={12(1)},
	publisher={Univ. Bologna, Alma Mater Stud., Bologna},
},
date={(2022)},
pages={53--80}
}

\bib{abatangelo_2}{article}{
author={N. Abatangelo},
author={S. Jarohs},
author={A. Salda{\~n}a},
title={On the maximum principle for higher-order fractional Laplacians},
journal={arXiv:1607.00929},
date={2016},
}

\bib{special_function}{book}{
	author={Abramowitz, Milton},
	author={Stegun, Irene A.},
	title={Handbook of mathematical functions with formulas, graphs, and
		mathematical tables},
	series={National Bureau of Standards Applied Mathematics Series},
	volume={No. 55},
	publisher={U. S. Government Printing Office, Washington, DC},
	date={(1964)},
	pages={xiv+1046}
}


\bib{potential}{book}{ 
author={D. R. Adams and L. I. Hedberg},
title={Function spaces and potential theory},
publisher={Grundlehren der Mathematischen Wissenschaften [Fundamental Principles of Mathematical Sciences], Springer-Verlag, Berlin},
volume={314}, 
date={(1996)},
}



\bib{riarrangio_1}{article}{
  author={Almgren, Frederick J., Jr.},
author={Lieb, Elliott H.},
title={Symmetric decreasing rearrangement is sometimes continuous},
journal={J. Amer. Math. Soc.},
volume={2},
date={1989},
number={4},
pages={683--773}
}

\bib{Vincenzo}{article}{
author={V. Ambrosio},
title={On the fractional relativistic Schrödinger operator},
journal={Journal of Differential Equations},
volume={308},
date={2022},
pages={327--368}
}

\bib{classicbook}{book}{
    author={H. Bahouri},
    author={J.Y. Chemin}, 
    author={
    R. Danchin},
    title={Fourier analysis and nonlinear partial differential equations}, publisher={Grundlehren der Mathematischen Wissenschaften [Fundamental Principles of Mathematical Sciences], Springer, Heidelberg},
    volume={343}, 
    date={(2011)},
}




\bib{lieb_3}{article}{
	author={B\'{e}nilan, Philippe},
	author={Brezis, Ha\u{\i}m},
	title={Nonlinear problems related to the Thomas-Fermi equation},
	journal={J. Evol. Equ.},
	volume={3},
	date={2003},
	number={4},
	pages={673--770}
	}

\bib{a_stable}{article}{
author = {Bogdan, Krzysztof},
author={Byczkowski, Tomasz},
title = {Potential theory for the $\alpha$-stable Schrödinger operator on bounded Lipschitz domains},
journal = {Studia Mathematica},
number = {1},
pages = {53--92},
volume = {133},
date = {1999}
}




\bib{optimal}{article}{
author={L. Brasco},
author={S. Mosconi},
author={M. Squassina},
title={Optimal decay of extremals for the fractional Sobolev inequality},
date={2016},
journal={Calculus of Variations and Partial Differential Equations},
volume={55},
number={2},
pages={Art. 23, 32}
}

\bib{Brezis_1}{article}{
author={Brezis, H. },
title={Semilinear equations in $\R^N$ without condition at infinity},
journal={Applied Mathematics and Optimization},
date={1984},
volume={12},
number={1},
pages={271-282}
}

\bib{cartan}{article}{
author={Brezis, Ha\"{\i}m},
author={Browder, Felix},
title={A property of Sobolev spaces},
journal={Comm. Partial Differential Equations},
volume={4},
date={1979},
number={9},
pages={1077--1083}
}



\bib{Carrillo-NA}{article}{
	author={Calvez, V.},
	author={Carrillo, J. A.},
	author={Hoffmann, F.},
	title={Equilibria of homogeneous functionals in the fair-competition
		regime},
	journal={Nonlinear Anal.},
	volume={159},
	date={2017},
	pages={85--128},
}

\bib{Carrillo-NA-2021}{article}{
	author={Calvez, Vincent},
	author={Carrillo, Jos\'{e} Antonio},
	author={Hoffmann, Franca},
	title={Uniqueness of stationary states for singular Keller-Segel type
		models},
	journal={Nonlinear Anal.},
	volume={205},
	date={2021},
	pages = {Paper No. 112222, 24},
}




\bib{Carrillo-CalcVar}{article}{
	author={Carrillo, Jos\'{e} A.},
	author={Hoffmann, Franca},
	author={Mainini, Edoardo},
	author={Volzone, Bruno},
	title={Ground states in the diffusion-dominated regime},
	journal={Calc. Var. Partial Differential Equations},
	volume={57},
	date={2018},
	number={5},
	pages={Art. 127, 28},
}





\bib{Volzone}{article}{
	author={Chan, Hardy},
	author={Gonz\'{a}lez, Mar\'{\i}a del Mar},
	author={Huang, Yanghong},
	author={Mainini, Edoardo},
	author={Volzone, Bruno},
	title={Uniqueness of entire ground states for the fractional plasma
		problem},
	journal={Calc. Var. Partial Differential Equations},
	volume={59},
	date={2020},
	number={6},
	pages={Paper No. 195, 42},
}


\bib{chandler}{article}{
author={Chandler-Wilde, S.N.},
author={Hewett, D.P.},
author= {Moiola, A.},
title={Sobolev Spaces on Non-Lipschitz Subsets of $\R^n$ with Application to Boundary Integral Equations on Fractal Screens},
journal={Integr. Equ. Oper. Theory},
volume= {87},
date= {2017},
number= {2},
pages= {179--224},
}


\bib{critical_exp}{article}{
author={W. Chen},
author={C. Li},
author={B. Ou},
title={Classification of solutions for an integral equation},
journal={Comm. Pure Appl. Math.},
volume= {59},
date= {2006},
number= {3},
pages = {330--343},
}

\bib{minimizer}{article}{
author={A. Cotsiolis},
author={N. Tavoularis}, 
title={Best constants for Sobolev inequalities for higher order fractional derivatives},
journal={J. Math. Anal. Appl.}, 
volume={295},
date={2004},
number={1},
pages={225–236}
}

\bib{Valdinoci}{article}{
author={E. Di Nezza},
author={G. Palatucci},
author={E. Valdinoci},
title={Hitchhiker's guide to the fractional Sobolev spaces},
journal={Bulletin des Sciences Mathématiques},
volume = {136},
date = {2012},
number= {5},
pages = {521--573},
}


\bib{Diaz}{book}{
author={Díaz, J. I.},
title={Nonlinear partial differential equations and free boundaries},
volume={106},
note={Elliptic equations Research Notes in Mathematics},
publisher={Pitman (Advanced Publishing Program), Boston, MA}, 
date={(1985)},
pages={vii+323pp},
}




\bib{dyda}{article}{
author={Dyda, B.},
author={Kuznetsov, A.},
 author= {Kwaśnicki, M.},
title={Fractional Laplace Operator and Meijer G-function},
journal={Constr. Approx.},
pages={427–448},
date={2017},
volume={45},
number={3}
}


\bib{P55}{article}{
	author={du Plessis, Nicolaas},
	title={Some theorems about the Riesz fractional integral},
	journal={Trans. Amer. Math. Soc.},
	volume={80},
	date={1955},
	pages={124--134},
	issn={0002-9947},
}



\bib{decomposition}{article}{
    author={J. Duoandikoetxea},
    title={Fractional integrals on radial functions with applications to weighted inequalities},
    journal={Ann. Mat. Pura Appl.},
    volume={192},
    date={2013}, 
    number={4},
    pages={553–568}, 
}

\bib{entire}{article}{
author = {M. M. Fall},
journal = {Proceedings of the American Mathematical Society},
number = {6},
pages = {2587--2592},
publisher = {American Mathematical Society},
title = {Entire s-harmonic functions are affine},
volume = {144},
date = {2016}
}


\bib{positive solutions}{article}{
author={P. Felmer},
author={A. Quaas},
author={J. Tan},
title={Positive solutions of the nonlinear Schr\"odinger equation
with the fractional Laplacian}, 
journal={Proc. Roy. Soc. Edinburgh Sect. A Math.},
date={2012},
volume={142},
number={6},
pages={1237–1262}
}

\bib{X.book}{book}{
	title={Integro-Differential Elliptic Equations},
	author={X. Fern\'{a}ndez-Real},
	author={X. Ros-Oton},
	publisher={Progress in Mathematics, Birkhauser},
	date={(2024) (forthcoming)}
}


\bib{Ferrari}{article}{
    author={F. Ferrari and I. E. Verbitsky},
    title={Radial fractional Laplace operators and Hessian inequalities},
    journal={J. Differential Equations},
    volume={253},
    date={2012},
    number={1},
    pages={244–272},
}
\bib{Frank}{article}{
 author={Frank, Rupert L.},
author={Lenzmann, Enno},
author={Silvestre, Luis},
title={Uniqueness of radial solutions for the fractional Laplacian},
journal={Comm. Pure Appl. Math.},
volume={69},
date={2016},
number={9},
pages={1671--1726}
}

\bib{Fukushima}{book}{
	author = {M. Fukushima},
	author= {Y. Oshima}, 
	author={M. Takeda},
	title = {Dirichlet Forms and Symmetric Markov Processes},
	publisher = {De Gruyter},
	address = {Berlin, New York},
pages = {x+489},
	date = {(2010)},
	
}

\bib{asymptotics}{article}{
	author={M. Gallo},
	title={Asymptotic decay of solutions for sublinear fractional Choquard equations},
	date={2024},
	journal={Nonlinear Analysis},
	volume={242}
}

\bib{riesz_hold}{article}{
	title={Boundedness properties of fractional integral operators associated to non-doubling measures},
	author={J. Garc{\'i}a-Cuerva and A. E. Gatto},
	journal={Studia Mathematica},
	volume= {162},
	date = {2004},
	number= {3},
	pages= {245--261},
}



\bib{garofalo}{article}{
author={Garofalo, Nicola},
title={Fractional thoughts},
conference={
	title={New developments in the analysis of nonlocal operators},
},
book={
	series={Contemp. Math.},
	volume={723},
	publisher={Amer. Math. Soc., (Providence), RI},
},
date={(2019)},
pages={1--135}
}
\bib{goldestein}{article}{
	author={Goldstein, Jerome A.},
	author={Rieder, Gis\`ele Ruiz},
	title={A rigorous modified Thomas-Fermi theory for atomic systems},
	journal={J. Math. Phys.},
	volume={28},
	date={1987},
	number={5},
	pages={1198--1202},
}

\bib{book_triebel}{book}{
author= {Grafakos, Loukas},
title = {Modern {F}ourier analysis},
series= {Graduate Texts in Mathematics},
volume= {250},
publisher = {Springer, New York},
date= {(2014)},
pages= {xvi+624}
}


\bib{BB}{article}{
author={D. Greco}, 
title={Brezis--Browder type results for Fractional Sobolev spaces and applications},
journal={arXiv:2407.06442}}

\bib{PhDthesis}{article}{
	author={D. Greco},
	title={Thomas–Fermi type variational problems with low regularity, PhD Thesis, Swansea University},
	date={2024}
}




\bib{Katsnelson}{article}{
author={Katsnelson, M.I.},
title={Nonlinear screening of charge impurities in graphene}, 
journal={Phys. Rev. B },
volume={74}, 
date={2006},
pages={201401}
}


\bib{mazya}{article}{
	author={V.P. Khavin},
	author={V.G. Maz'ya},
	title={Non-linear potential theory},
	date={1972},
	journal={Russ. Math. Surv.},
	volume={27},
	number={71}
}

\bib{Samko}{book}{
	author={Kilbas, Anatoly A.},
	author={Marichev, Oleg I.},
		author={Samko, Stefan G.},
	title={Fractional integrals and derivatives},
	publisher={Gordon and Breach Science Publishers, Yverdon},
	date={(1993)},
	pages={xxxvi+976},
}




\bib{ten}{article}{
	author={Kwaśnicki, Mateusz},
	title={Ten equivalent definitions of the fractional laplace operator},
	journal={Fractional Calculus and Applied Analysis}, 
	volume={20}, 
	number={1},
	date={2017},
	pages={7--51}
}



\bib{landkof}{book}{
	author={Landkof, N. S.},
	title={Foundations of modern potential theory},
	series={Die Grundlehren der mathematischen Wissenschaften},
	volume={Band 180},
	publisher={Springer-Verlag, New York-Heidelberg},
	date={(1972)},
	pages={x+424},
}
\bib{holderloc}{article}{
	author={Lemm, Marius},
	title={On the H\"older regularity for the fractional Schr\"odinger equation and its improvement for radial data},
	journal={Communications in Partial Differential Equations},
	volume= {41},
	date = {2016},
	number= {11},
	pages = {1761--1792}
}

\bib{lieb_dim3}{article}{
	author={E. H. Lieb},
	title={Thomas-Fermi and related theories of atoms and molecules},
	journal={Rev. Mod. Phys.},
	volume={53},
	date={1981},
	pages={603– 641}
}

\bib{Lieb-Loss}{book}{
	author={Lieb, Elliott H.},
	author={Loss, Michael},
	title={Analysis},
	series={Graduate Studies in Mathematics},
	volume={14},
	edition={2},
	publisher={American Mathematical Society, Providence, RI},
	date={(2001)},
	pages={xxii+346},
}

\bib{Lieb-Simon}{article}{
	title={The Thomas-Fermi theory of atoms, molecules and solids},
	author={E. H. Lieb and B. Simon},
	journal={Advances in Mathematics},
	date={1977},
	volume={23},
	pages={22-116}
}





\bib{orbital}{article}{
	 author={Lu, Jianfeng},
	author={Moroz, Vitaly},
	author={Muratov, Cyrill B.},
	title={Orbital-free density functional theory of out-of-plane charge
		screening in graphene},
	journal={J. Nonlinear Sci.},
	volume={25},
	date={2015},
	number={6},
	pages={1391--1430}
}



\bib{h_0}{book}{
author={W. McLean},
title={Strongly Elliptic Systems and Boundary Integral Equations},
publisher={CUP, Cambridge}, 
date={(2000)}
}
\bib{Mizuta}{book}
{author={Mizuta, Yoshihiro},
	title={Potential theory in Euclidean spaces}, 
	publisher={GAKUTO International Series Mathematical Sciences and Applications, 6. Gakkōtosho Co., Ltd., Tokyo},
	date={(1996)},
	 volume = {6},
	pages={viii+341}
}




\bib{graphene}{article}{
author={V. Moroz},
author={C. B. Muratov},
title={Thomas-Fermi theory of out-of-plane charge screening in graphene},
journal={Proc. R. Soc. Lond. Ser. A.}, 
date={2024},
}
\bib{Musina}{article}{
author={Musina, R.},
author={Nazarov, A.I}, 
title={A note on truncations in fractional Sobolev spaces},
journal={Bull. Math. Sci.},
volume = {9},
date= {2019},
number = {1}
}




\bib{multiplier}{article}{
author={Palatucci, Giampiero},
author={A. Pisante},
title={Improved Sobolev embeddings, profile decomposition, and concentration-compactness for fractional Sobolev spaces},
journal={Calculus of Variations and Partial Differential Equations},
volume={50},
number= {3-4},
date={2014},
pages={799-829}
}



\bib{maximum}{book}{
   author={Pucci, Patrizia},
author={Serrin, James},
title={The maximum principle},
series={Progress in Nonlinear Differential Equations and their
	Applications},
volume={73},
publisher={Birkh\"{a}user Verlag, Basel},
date={(2007)},
pages={x+235}

}

\bib{Rieder}{article}{
author = {Rieder, G. R.},
title = {Mathematical contributions to Thomas-Fermi theory},
journal={Tulane Univ., New Orleans, LA (USA)},
number ={} ,
volume ={} ,
pages={},
date = {1986},
}



\bib{Ros}{article}{
author={X. Ros Oton},
author={J. Serra},
title={Regularity theory for general stable operators},
journal={J. Differ. Equ},
volume={260}, 
date={2016},
number= {12},
pages={8675–8715},
}



\bib{Silvestre}{article}{
    author={L. Silvestre},
    title={Regularity of the obstacle problem for a fractional power of the Laplace operator},
    journal={Comm. Pure Appl. Math},
    date={2007},
    volume={60},
    number={1}, 
    pages={67–112},
}

\bib{Stein}{book}{
author={E. M. Stein},
title={Singular integrals and differentiability properties of functions},
publisher={Princeton Mathematical Series, Princeton University Press,         Princeton, N.J.},
date={(1970)}
}

\bib{Veron}{article}{
    author={L. Veron}, 
    title={Comportement asymptotique des solutions d'equations elliptiques semi-lineares dans $\R^N$},
    journal={ Ann. Mat. Pura Appl.},
    volume={127},
    number={4},
    date={1981},
    pages={25–50} ,
}





\end{thebibliography}
\end{document}